\documentclass{amsart}

\usepackage[top=3cm, bottom=3cm, left=2cm, right=2cm]{geometry} 
\usepackage{amsmath, amssymb}
\usepackage[english]{babel}
\usepackage[utf8]{inputenc}
\usepackage[T1]{fontenc}
\usepackage{mathrsfs}
\usepackage{amsthm}
\usepackage{enumerate} 
\usepackage{hyperref}
\usepackage{enumitem}
\usepackage{comment}
\usepackage{mathtools}
\usepackage{esint}
\usepackage{faktor}
\usepackage{xfrac}  
\usepackage{appendix}
\usepackage{color}
\usepackage[capitalize]{cleveref}
\usepackage[colorinlistoftodos,prependcaption]{todonotes}
\usepackage{xcolor}

\newtheorem{theorem}{Theorem}[section]

\newtheorem{corollary}[theorem]{Corollary}
\newtheorem{lemma}[theorem]{Lemma}
\newtheorem{remark}[theorem]{Remark}

\newtheorem{defi}[theorem]{Definition}
\newtheorem{claim}[theorem]{Claim}

\newcommand{\scal}[2]{\left\langle #1,#2 \right\rangle}
\newcommand{\g}{\nabla}

\newcommand{\di}{\mathrm{div}}
\newcommand{\lap}{\Delta}
\newcommand{\dr}{\partial}
\newcommand{\vol}{\mathrm{vol}}
\newcommand{\loc}{\mathrm{loc}}

\newcommand{\diam}{\mathrm{diam}}

\newcommand{\supp}{\mathrm{supp}}
\newcommand{\dist}{\mathrm{dist}}

\newcommand{\tr}{\mathrm{tr}}

\newcommand{\II}{\mathrm{I\!I}}

\newcommand{\Diff}{\mathrm{Diff}}

\newcommand{\BMO}{\mathrm{BMO}}

\newcommand{\id}{\mathrm{id}}

\newcommand{\Riem}{\mathrm{Riem}}
\newcommand{\Ric}{\mathrm{Ric}}

\newcommand{\Scal}{\mathrm{Scal}}
\newcommand{\Sch}{\mathrm{Sch}}
\newcommand{\Cot}{\mathrm{Cot}}

\newcommand{\Hess}{\mathrm{Hess}}

\newcommand{\SOo}[1]{\mathrm{SO}\left( #1,\mathbb{R} \right)}

\newcommand{\B}{\mathbb{B}}
\newcommand{\D}{\mathbb{D}}
\newcommand{\R}{\mathbb{R}}
\newcommand{\C}{\mathbb{C}}
\newcommand{\N}{\mathbb{N}}

\newcommand{\s}{\mathbb{S}}

\newcommand{\Mr}{\mathcal{M}}

\newcommand{\Br}{\mathcal{B}}

\newcommand{\Hr}{\mathcal{H}}

\newcommand{\Dr}{\mathcal{D}}

\newcommand{\ust}{\underset}

\newcommand{\vp}{\varphi}
\newcommand{\ve}{\varepsilon}

\title{Huber Theorem revisited in dimensions $2$ and $4$}

\author{Paul Laurain}
\address{Departement de Mathématiques, Université Gustave Eiffel, France}
\email{paul.laurain@univ-eiffel.fr}

\author{Dorian Martino}
\address{Department of Mathematics, ETH Zürich, Rämistrasse 101, 8092 Zürich, Switzerland}
\email{dorian.martino@math.ethz.ch}

\begin{document}
	
	\begin{abstract}
		We study the second Huber theorem in dimensions 2 and 4. In dimension 2, we prove a new version assuming that the Gauss curvature lies in a negative Sobolev space using Coulomb frames. In dimension $4$, given a metric having a pointwise singularity with $L^p$-bounds on the Bach tensor, we construct a conformal metric which is regular across the singularity. To do so, we introduce another Coulomb-type condition, similar to the case of Yang--Mills connections. This enables us to obtain a conformal metric satisfying an $\varepsilon$-regularity property. We obtain a generalization of the two-dimensional case that can be applied to study the singularities of Bach-flat metrics and immersions with second fundamental forms in $W^{2,\frac{4}{3}+\varepsilon}$. 
	\end{abstract}
	
	\maketitle
	

	\section{Introduction}
	
	Isolated singularities of manifolds arise frequently in geometric variational problems, for instance in the study of harmonic maps and their generalizations \cite{ai2017,dalio2015,laurain2013,laurain2014}, Willmore surfaces \cite{bernard2014,laurain2018,martino2023}, but also Einstein and Bach-flat metrics \cite{anderson1990,anderson2005,tian2005,tian2008} or Yang--Mills connections \cite{riviere2020}. In such problems, one usually have an a priori estimate taking the form of an $L^p$ bound on some curvature quantity (such as the second fundamental form for "extrinsic" problems or the Riemann tensor for "intrinsic" problems). A natural question is to understand the geometric behaviour across the singularity: is it possible to find a suitable extension across the singularity? 
	In 1958, Huber proved, in his seminal paper \cite{Huber}, the following result on the topology of surfaces with finite total negative curvature.
	\begin{theorem}\label{huber1}
		Let $(M,g)$ be an orientable Riemannian surface of class $C^2$, open and complete such that 
		$$
		\int_M K^-_g \, d\vol_g <+\infty,
		$$
		where $K^-_g \coloneq \max(0,-K_g)$ and $K_g$ is the Gauss curvature. Then $M$ has finite topological type and moreover
		$$
		\int_M \vert K_g\vert \, d\vol_g <+\infty.
		$$
	\end{theorem}
	Here finite topological type means that there exists a closed (i.e. compact and without boundary) surface $\overline{M}$ and a finite set of points  $p_1,\dots,p_N \in \overline{M}$ such that $M$ is homeomorphic to $\overline{M}\setminus \{p_1,\dots,p_N\}$. Once the topological type is fixed, a natural question concerns the conformal type, i.e. which kind of Riemann structure is supported by such surfaces. Thanks to Hopf--Rinow theorem, $(M,g)$ is complete if every divergent path on $M$ has infinite length. In the same paper \cite{Huber}, Huber also describes the conformal class of $(M,g)$ in \Cref{huber1} by proving the following result.
	
	\begin{theorem}\label{huber2}
		Let $(M,g)$ be an orientable surface of class $C^2$, open and complete such that 
		$$
		\int_M K^-_g \, dv_g <+\infty.
		$$
		Then $M$ has finite conformal type.
	\end{theorem}
	Here finite conformal type means that there exists a closed Riemann surface $\overline{M}$ and a finite set of points  $p_1,\dots,p_N \in \overline{M}$ such that $M$ is conformally equivalent to $\overline{M}\setminus \{p_1,\dots,p_N\}$. In other words, there exist a $C^1$-diffeomorphism $\psi\colon M \rightarrow \overline{M}\setminus \{p_1,\dots,p_N\}$ and a metric $g_0$ on $\overline{M}$ having constant Gaussian curvature, such that $\psi^*(g_0)=e^{2u}g$ for some continuous function $u\colon M\rightarrow \R$.\\
	
	In 1995, inspired by the work of Toro \cite{Toro}, Müller and \v{S}ver\'{a}k \cite{muller1995} proved an extrinsic version of the Huber theorem: assuming for initial regularity only the fact that the second fundamental form lies in $L^2$, they proved that those surfaces admit a Lipschitz conformal parametrization. Around a potential singularity, they prove the following theorem, where $\D$ is the unit disk centred at the origin in $\C$.
	
	\begin{theorem}[\cite{muller1995}, Lemma A.5 in \cite{Ri14}]\label{th:imm}  Let $\Phi\colon \D\setminus\{0\} \rightarrow \R^m$ be a conformal immersion such that $\Phi \in W^{1,2}( \D\setminus\{0\})$ and its Gauss map is also in $W^{1,2}( \D\setminus\{0\})$. Then there exists an integer $\theta_0\geq 1$ such that, up to translation, the following holds for some constant $C>1$ in a neighbourhood of the origin
		$$
		C^{-1}\, |z|^{\theta_0} \leq \vert \Phi (z)\vert \leq C\, \vert z\vert^{\theta_0} \qquad \text{ and } \qquad  C^{-1}\, |z|^{\theta_0-1} \leq \vert \nabla \Phi (z)\vert \leq C\, \vert z\vert^{\theta_0-1}.
		$$
		Moreover, if we denote the conformal factor
		$$
		\lambda(z) = \frac{1}{2}\log\left(\frac{1}{2} \vert \nabla \Phi (z)\vert^2 \right)= (\theta_0-1)\log(\vert z\vert)+ u(z),
		$$
		then the remainder $u$ lies in $W^{2,1}(\D)$ and we have
		\begin{align*}
			\begin{cases}
				\g \lambda \in L^{2}(\D) &\text{ if } \theta_0=1,\\[1mm]
				\g \lambda \in L^{(2,\infty)}(\D) &\text{ if } \theta_0\geq 2.
			\end{cases}
		\end{align*}
	\end{theorem}	  
	
	Here $L^{(2,\infty)}$ denotes the weak-$L^2$ space also known as Lorentz space, see \cite[Section 1.4]{Gra1} for a precise definition. The key ingredient in the proof of \Cref{th:imm} is the fact that the Gauss curvature is the Jacobian of the gradient of the Gauss map. Hence the Gauss curvature lies in a strictly smaller space than $L^1$, namely in the Hardy space $\mathcal{H}^1$, see the seminal work of Coiffman, Lions, Meyer and Semmes \cite{CMLS}. Thus, \Cref{th:imm} follows by combining the $\Hr^1$-$\BMO$ duality and the Liouville equation.\\
	
	The main goal of this paper is to provide possible generalizations to \Cref{huber2} and \Cref{th:imm} in higher dimensions and more specifically in the 4-dimensional setting. That is to say, given a metric $g$ of class $C^{\infty}$ on $\B^4\setminus \{0\}$, where $\B^4$ stands for the unit 4-dimensional Euclidean ball centred at the origin\footnote{If we consider the case where $g$ is invariant by the action of a finite group of $\SOo{4}$, we can also consider the case of orbifolds.}, under which condition can we find a metric $h$ in the conformal class of $g$ that can be extended through the singularity? This question has already been addressed by various authors, who focused mainly on the generalization of \Cref{huber1}, namely the topological part. The first result in this direction has been obtained in 2002 by Chang, Qing and Yang for the conformally flat case.
	\begin{theorem}[Theorem 2 of \cite{CQY}]\label{th:CQY}
		Let $(\Omega \subset S^4, g=e^{2w}g_0)$ be a complete conformally flat metric satisfying:
		\begin{enumerate}[label=(\roman*)]
			\item There exists $0<c_1<c_2$ such that scalar curvature $\Scal_g$ verifies $c_1 \leq \Scal_g \leq c_2$ and $|\g \Scal_g|_g \leq c_2$,
			\item The Ricci curvature has a lower bound,
			\item The $Q$-curvature is integrable, i.e.
			$$\int_\Omega \vert Q_g\vert \, d\vol_g <+\infty .$$
		\end{enumerate}
		Then it holds $\Omega=S^4\setminus \{p_1,\dots ,p_k\}$.
	\end{theorem}
	Independently, Carron and Herzlich \cite{CH02} proved in 2002 that an open set $\Omega$ of a compact manifold $(M^n,g_0)$ (for $n>2$) endowed with a complete conformal metric $g=e^{2u}g_0$ is also given by $\Omega=M\setminus \{p_1,\ldots,p_k\}$, under some conditions on the volume growth of balls for $g$ and the finiteness of the $L^{\frac{n}{2}}$-norm of the Ricci tensor of $g$. This result has later been improved by Carron in 2020 \cite[Corollary C]{Carron20}. Finally in 2022, Chen and Li reached the following optimal version.
	\begin{theorem}[Theorem 1.4 in \cite{CL22}]\label{th:CL}
		Let $\Omega$ an open set of $(M^n,g_0)$ a Riemannian manifold of dimension $n>2$. Assume that $\Omega$ is endowed with a complete Riemannian metric $g$
		which is conformal to $g_0|_{\Omega}$ and such that:
		$$\int_\Omega \vert \mathrm{Ric}\vert^{\frac{n}{2}}\, d\vol_g < +\infty. $$
		Then  $M\setminus \Omega$ is discrete. In particular, if $(M,g_0)$ is compact then it holds $\Omega=M\setminus \{p_1,\dots ,p_J\}$ for some points $p_1,\ldots,p_J\in M$.
	\end{theorem}
	
	\Cref{th:CQY} and \Cref{th:CL} are in the spirit of the first Huber theorem, i.e. controlling the topology of the domain. But it is important to notice that those results are not a complete generalization of \Cref{huber1}, since they need $g$ to be a priori conformally equivalent to a background metric $g_0$ which extends smoothly outside of the domain of definition of $g$. The content of \Cref{huber2} is to prove, in dimension 2, that such an extension exists. If one would like a strict equivalent of Huber's theorem, one should also prove the existence of such an extension.	In 1997, Herzlich \cite{Her97} studied the conformal compactification of asymptotically flat Riemannian manifolds $(M^n,g)$ of dimension $n\geq 3$ and proved the following result, see \Cref{sec:asympt_flat} for the definitions.
	\begin{theorem}[Theorem C in \cite{Her97}]\label{th:Herzlich}
		Any asymptotically flat Riemannian manifold, where the Weyl curvature tensor decays as  $r^{-2-\epsilon}$  at infinity and the Cotton-York tensor decays as  $r^{-3-\epsilon}$  (with  $\epsilon > 0$), is obtained via stereographic projection from some compact $C^0$-class Riemannian manifold.
	\end{theorem}
	The main idea of the proof of the above result is to study the conformal connections introduced by Cartan \cite{cartan1924} in 1924. He proves that the conformal fiber bundle does not degenerate at infinity, and is actually a Hölder-continuous bundle. The underlying analysis is based on the study of elliptic equations on weighted Sobolev spaces. Considering an inversion, this theorem can be translated to give an extension to a metric defined on $\B\setminus \{0\}$ and having finite volume.\\
	
	Our main motivation is to develop a similar result that would be applicable to variational problems such as the study of compactness properties of Bach-flat metrics or generalizations of Willmore surfaces in dimension 4. In these questions, the setting is usually a sequence of manifolds which converges strongly everywhere except at a finite number of points and one would like to understand the behaviour of these manifolds across these singularities. This is the so-called bubbling convergence which usually follows from an $\ve$-regularity property. For 2-dimensional problems see for instance the works \cite{bernard2014,laurain2014,dalio2023} and the references therein. In  dimension 4, see for instance in the study of Einstein metrics or Bach-flat metrics, Anderson \cite{anderson1990,anderson2005} and Tian--Viaclovksy \cite{tian2005,tian2008} proved that the only possible singularities occurring along a sequence of bounded energy are orbifold singularities, see also \cite{biquard2014}. \\
	
	Our main contribution is to change the nature of the assumptions of \Cref{th:Herzlich} into $L^p$-bounds on various curvature tensors, and to consider our initial manifold to be a ball minus its center or slight generalizations (in other words, we assume that the topology of the underlying manifold is already known). In this article, we focus on $L^p$-bounds on the Bach tensor, namely we consider a metric $g$ that is $C^{\infty}$ on $\B^4\setminus \{0\}$ with Riemann tensor in $L^2(\B^g,g)$ and Bach tensor in $L^p(\B^4,g)$ for some $p>1$. As opposed to \Cref{th:Herzlich}, we do not assume any pointwise estimate on the metric near the singularity. Moreover, as though the extension of the metric across the singularity in \Cref{th:Herzlich} can be guessed using the asymptotic behaviour $g_{ij}(z)\ust{z\to \infty}{\sim} \delta_{ij}$, we do not have any a priori candidate of an extension that would satisfy better regularity properties in our setting. Therefore, we need to rely on the analysis of the partial differential equations induced by the assumption that the Bach tensor lies in $L^p$. We prove the following result. 
	
	\begin{theorem}\label{th:global_Huber_4D}
		Consider a $C^{\infty}$ orbifold metric $g$ on $X\coloneq C(\s^3/\Gamma)\setminus \{0\}$\footnote{In other words, the pull-back of the metric $g$ is $C^{\infty}$ on $\B^4\setminus \{0\}$, and does not a priori extends smoothly to $\B^4$, which would provide a smooth orbifold metric.}, where $\Gamma$ is a finite subgroup of $\SOo{4}$ and $C(\s^3/\Gamma)$ is the cone over $\s^3/\Gamma$. Assume that $g$ satisfies the following assumptions:
		\begin{enumerate}
			\item\label{hyp:finite_vol} The volume of $X$ is finite: $\vol_g(X)<+\infty$.
			\item\label{hyp:integrability} We have $\Riem^g\in L^2(X,g)$ and  $B^g\in L^p(X,g)$.
			\item\label{hyp:Sobolev} There exists a constant $\gamma_S>0$ such that the following Sobolev inequality holds:
			\begin{align*}
				\forall u\in C^{\infty}_c(X),\qquad
				\|u\|_{L^4(X,g)}^2 \leq \gamma_S \int_{X} |\g u|^2_g \, d\vol_g.
			\end{align*}
			\item\label{hyp:Laplacian} There exists a constant $\gamma_L>0$ such that the following regularity estimate holds. For any $u\in C^{\infty}_c(X\setminus \{0\})$, it holds
			\begin{align*}
				\|\g u\|_{L^4(X,g)} \leq \gamma_L \|\lap_g u\|_{L^2(X,g)} .
			\end{align*}
		\end{enumerate}
		There exists a constant $\ve>0$ depending only on $\vol_g(X),\gamma_S,\gamma_L$ such that the following holds. Assume that
		\begin{align*}
			\|\Riem^g\|_{L^2(X,g)} \leq \ve.
		\end{align*}
		Then there exist a diffeomorphism $\psi\colon X\to X$, a conformal factor $u\in C^{\infty}(X)$, and a $W^{2,2p}$ orbifold metric $h$ on $C(\s^3/\Gamma)$ which is smooth on $X$ such that
		\begin{align*}
			\psi^*g = e^{2u}h.
		\end{align*}
	\end{theorem}
	
	To the best of the authors' knowledge, this is the first time that a background metric is constructed in a generic manner in the context of Huber's theorem in dimension higher than 2. We can separate the assumptions of \Cref{th:global_Huber_4D} into two categories: on the one hand \eqref{hyp:finite_vol} and \eqref{hyp:integrability} are natural geometric assumptions, on the other hand \eqref{hyp:Sobolev} and \eqref{hyp:Laplacian} are analytical assumptions which permit to transform the geometric assumptions into estimates. These analytical assumptions are satisfied on any "reasonable" manifold, for instance, after compactification of an asymptotically flat manifold. We discuss them in \Cref{sec:main}.  Moreover, they are all satisfied by a $W^{2,2p}$-metric, which implies that our assumptions are minimal in some sense. As a consequence of \Cref{th:global_Huber_4D}, we prove the following version of Herzlich's result. The second part can also be considered as an improvement of $\ve$-regularity of Bach-flat metric of Tian and Viaclovsky, see Theorem 1.1 in \cite{tian2005}.
	
	\begin{corollary}\label{cor:AF}	
		Let $(M^4,g)$ be an asymptotically flat Riemannian manifold with the Riemann tensor in $L^2(M)$ and the Bach tensor in $L^p(M)$ for some $p>1$. Then $(M,g)$ is obtained via stereographic projection from a compact $W^{2,2p}$-class Riemannian manifold (and in particular the $C^{0,\alpha}$-class for some $\alpha>0$ depending on $p$). In particular, if $(M,g)$ is Bach-flat, it can be obtain via stereographic projection from a compact $W^{2,2p}$-class Riemannian manifold for any $p>1$, and therefore also $C^{0,\alpha}$ for any $0<\alpha<1$. 
	\end{corollary}
	
	\Cref{th:global_Huber_4D} also applies to the extrinsic setting. Indeed, by Gauss--Codazzi equations, we know that the Riemann tensor is quadratic in the second fundamental form of the given immersion. Thus, the Bach tensor depends on the second derivatives of the second fundamental form. We obtain the following result, which can be understood as a study of branch points of immersions having a second fundamental form in $W^{2,\frac{4}{3}+\ve}$ for some $\ve>0$.
	\begin{theorem}\label{th:immersions}
		Let $\ve\in(0,\frac{1}{10})$ and $d\geq 5$ be an integer. There exists $\delta>0$ depending only on $d$ and $\ve$ such that the following holds. Consider $\Phi\colon \B^4\setminus \{0\}\to\R^d$ be a smooth immersion such that its second fundamental form $\II$ satisfies
		\begin{align*}
			\int_{\B^4} |\g^2\II|^{\frac{4}{3}+\ve}_{g_{\Phi}} + |\g\II|^{2+\ve}_{g_{\Phi}} + |\II|^{4+\ve}_{g_{\Phi}} +1\, d\vol_{g_{\Phi}} <+\infty.
		\end{align*}
		Assume that
		\begin{align*}
			\int_{\B^4}  |\II|^4_{g_{\Phi}}\, d\vol_{g_{\Phi}} <\delta.
		\end{align*}
		Then, there exist a diffeomorphism $f\in \Diff(\B^4\setminus \{0\})$, a conformal factor $u\in C^{\infty}(\B^4\setminus \{0\})$ and a metric $h$ on $\B^4$ of class $W^{2,2p}(\B^4)$ for some $p=p(\ve)>1$ such that $g_{\Phi\circ f} = e^{2u} h$.
	\end{theorem}

	The main idea behind the proof of \Cref{th:global_Huber_4D} is to consider the tractor bundle and find a Coulomb gauge as it can be done for Yang--Mills connections by minimizing the $L^2$-norm of the connections of this bundle. Indeed, Korzyński and Lewandowski \cite{korzy2003} in 2003 (see also \cite{gover2008} for the formalism of tractor calculus) proved that Bach-flat metrics correspond to the conformal class whose tractor bundle is Yang--Mills. The heuristic is that if the curvature of a bundle is regular across the singularity, then one should be able to extend the bundle itself across the singularity. Starting from this, one can expect to find an equivalent formulation for the Coulomb gauge introduced by Uhlenbeck \cite{uhlenbeck1982}, see also the survey of Rivière \cite{riviere2020}. As a first step before the 4-dimensional case, we develop this idea in the 2-dimensional situation.\\
	
	In 2-dimensional problems, Coulomb frames have been introduced by Hélein to prove the regularity of harmonic maps \cite{helein1991}. He later proved in \cite{helein2002} that these frames could be used to construct conformal coordinates for surfaces with $L^2$ second fundamental form, see also the recent survey \cite{lan2025}. The first result of this paper constitutes a version of \Cref{huber2} where the initial regularity assumption on the Gauss curvature is no more $L^1$ but merely lies in a negative Sobolev space. In particular, we have obtain the following version of Huber's theorem, which is a consequence of the local version stated in \Cref{th:Huber2D}.
	
	\begin{theorem}\label{th:global_Huber_2D}
		Consider $\Sigma$ a closed Riemann surface and a finite number of distinct $p_1,\ldots,p_J\in \Sigma$. Let $g$ be a smooth metric on $\Sigma\setminus \{p_1,\ldots,p_J\}$ such that near every $p_j$, there exist coordinates on $\D\setminus \{0\}$, where $p_j$ is sent to $0$ and such that
		\begin{enumerate}
				\item There exists a coframe $(\theta^1,\theta^2)$ for $g$, with degree zero\footnote{See \Cref{sec:examples} for a discussion on this condition.}, such that its connection form $\theta_{12}$ lies in $L^{(2,\infty)}(\D)$ for the Euclidean metric.
			
			\item There exists $m\in \R$ such that 
			\begin{align*}
				K_g\, \sqrt{\det (g_{ij})} \in m\, \delta_0 + W^{-1,2}(\D).
			\end{align*}
			
			\item There exists $C>0$ and $\delta\in(0,1)$ such that for any $0<r<\frac{1}{2}$, it holds
			\begin{align*}
				\vol_g\left( \D_{2r}\setminus \D_r \right) \leq C\, r^{2-2m-\delta},
			\end{align*}
			where $\D_r$ and $\D_{2r}$ denote the disks in $\C$ centred at the origin with radius $r$ and $2r$ respectively.
			
			\item There exists $\Lambda_{\max}>0$ and a map $\Lambda\colon \D\setminus \{0\}\to (0,\Lambda_{\max})$ such that for every $x\in \D\setminus \{0\}$ and every ellipsoid $E\subset \D\setminus \{0\}$ centred at $x$, it holds
			\begin{align*}
				\vol_g(E) \leq \Lambda(x)\, \ell_g(\dr E)^2.
			\end{align*}
			We have the integrability property
			\begin{align*}
				\int_{\D} \frac{( (4\pi\Lambda(x))^2-1)^2 }{|x|^2}\, dx <+\infty.
			\end{align*}
		\end{enumerate}

		Then, there exist $\alpha_1,\ldots,\alpha_J\in \R$, a diffeomorphism $\psi\colon \Sigma\setminus\{p_1,\ldots,p_J\}\to \Sigma\setminus \{p_1,\ldots,p_J\}$, a smooth metric $h$ on $\Sigma$ with constant curvature and a function $\vp \in  W^{1,2}(\Sigma,h)\cap C^{\infty}(\Sigma\setminus \{p_1,\ldots,p_J\})$ and 
		\begin{align*}
			\psi^*g = \dist_h(p_1,\cdot)^{2\alpha_1} \cdots  d_h(p_J,\cdot)^{2\alpha_J}\, e^{2\vp}\, h.
		\end{align*}
		Furthermore, $\vp$ is a weak solution to the Liouville equation
		\begin{align*}
			-(\lap_h \vp)\, d\vol_h = \sum_{j=1}^J 2\pi\, \alpha_j\,  \delta_{p_j} + K_{	\psi^*g} \, d\vol_{\psi^*g} - K_h\, d\vol_h \qquad \text{ in }\Dr'(\Sigma,h).
		\end{align*}
	\end{theorem}
	
	The necessity of the assumption is discussed in \Cref{sec:examples}. Our starting point for \Cref{th:global_Huber_2D} is the fact that the 2-form $K_g\, d\vol_g$ is given by the exterior derivative of the connection form of any coframe. This provides a definition of the Gauss curvature in a negative Sobolev space in \Cref{sec:def_K}, which constitutes the main difference with the classical Huber theorem. The assumption (2) is more about the 2-form $K_g\, d\vol_g$ than on $K_g$ itself, allowing it to lie in a negative Sobolev space but also imposing the condition that across the singularity, it has to be given by a Dirac mass. A crucial idea to find good conformal coordinates, notably used in the proof of Theorem~\ref{th:imm}, is to produce a Coulomb frame. Our hypothesis (1) ensures a minimal regularity of the starting frame. On the other hand, we have replaced the completeness assumption in \Cref{huber2} by (3), a compatible growth of the area with singularity behaviour, which is made necessary by Example (d) in \Cref{sec:examples}. Finally, it seems necessary to have some control of the local isoperimetric constant (4) to conclude, see also \cite{LeguilRosenberg} for similar assumption. The assumptions (1) and (2) arise naturally when one considers sequences of surfaces with $L^2$-bounded second fundamental form, but to obtain (4) we would need at least an $L^p$-bounded second fundamental form for $p>2$, see \cite{brendle2021}, which is not a conformally invariant assumption. But beyond the result, what is more interesting is the proof, where thanks to these analytic assumptions, we are able to begin the search for a Coulomb frame which will give rise to the conformal coordinates and to a conformal factor that, up to a standard $z^n$ singularity, extends to a $W^{1,2}$ function and even a continuous fonction if $K_g\, \sqrt{\det (g_{ij})} \in m\, \delta_0 + \mathcal H^1_{loc}(\D)$. This is the main idea behind our four dimensional result, as explain now\footnote{An interesting observation following from the comparison between the 2-dimensional and the 4-dimensional cases is that the factor $|z|^{\alpha}$ arises in dimension 2 only because $\B^n\setminus \{0\}$ is not simply connected for $n=2$, but is simply connected for $n=4$.}.\\

	Coming back to the four dimensional case, one could also try to find a preferred frame that would provide the conformal coordinates with respect to a preferred metric. However, finding the right energy on frames to consider is unclear since minimizing the connection forms does not seem to produce an integrable frame because of the Weyl obstruction. Instead, we consider the tractor bundle as explained previously, and find an equivalent formulation for the Coulomb gauge. The condition for a connection $D=d+A$ over some bundle on the unit ball $\B^4$ to be a Coulomb gauge is the condition $d^*A=0$. This system turns out to be the Euler--Lagrange system of the functional $\int_{\B} |A|^2$. In the case of the tractor bundle, the connection is mainly given by the Schouten tensor. We refer to \cite{curry2018,vyatkin} for an introduction to tractor calculus. Thanks to Chern--Gauss--Bonnet theorem in dimension 4, conformal variations of the $L^2$-norm of the Schouten tensor are equivalent to conformal variations of the $L^2$-norm of its trace. Hence, our problem of finding an optimal metric reduces to minimizing the $L^2$ norm of the scalar curvature in the conformal class of our given metric, see \Cref{strategy} for more details. \\
	
	From a purely PDE view-point, the proof of \Cref{th:global_Huber_4D} can be described as follows. One can understand the Bach tensor as the given data in a system satisfied by the Schouten tensor:
	\begin{align*}
		\lap_g \Sch^g  = B^g + \Hess_g J^g + \Riem^g\star \Sch^g,
	\end{align*}
	where $J^g= \tr_g\, \Sch^g$ and the term $\Riem^g\star \Sch^g$ stands for a linear combination of various contractions of the Riemann tensor $\Riem^g$ and the Schouten tensor. Tracing this system provides no information since the Bach tensor is traceless (this is a consequence of the conformal invariance). Roughly, this means that the Bach tensor is actually the data in a system satisfied only by the traceless part of the Schouten tensor. In order to recover an information on the full tensor, one needs to prescribe the trace part $J^g$. In our setting, we work on degenerating annuli $\B\setminus \B_r$, with $r\to 0$. A first idea would be to solve a Yamabe problem (without caring too much about the boundary conditions) in order to find a metric $g_0 =e^{2u}g$ such that $J^{g_0}$ is constant. Then, the analysis would essentially be contained in the work of Tian and Viaclovsky \cite{tian2005,tian2008}. However, the resolution of such a problem requires some additional assumptions such as having a positive Yamabe invariant, see for instance \cite{escobar1994,marques2007} and \Cref{rk:J_cst}. In some sense our $\ve$-regularity, \Cref{pr:eps_regp} comes as a generalization of the one of Tian--Viaclovsky, see \Cref{rk:notcsc}. In order to provide a general approach, one needs to solve an equation of higher order (an equation having the same number of derivative as $\lap_g \Sch^g$) and this is the reason of the assumptions \ref{hyp:Laplacian} in \Cref{th:global_Huber_4D}. Our approach provides a metric $g_0=e^{2u} g$ which satisfies $\lap_{g_0} J^{g_0} = 0$ and preserves a bound on the volume of $g_0$. Thus, $J^{g_0}$ is as regular as authorized by the geometric setting and the regularity of the metric $g_0$ is now fully determined by the Bach tensor $B^{g_0}=e^{-2u} B^g$. An interesting observation is that the $Q$-curvature of the metric $g_0$ now satisfies the pointwise bound $|Q^{g_0}|\leq C|\Riem^{g_0}|^2_{g_0}$ and thus, lies in $L^p(X,g_0)$ with $p>1$.\\
	
	This strategy applies only to the dimension 4 since the Bach tensor is conformally invariant in dimension 4 and the Chern--Gauss--Bonnet formula is valid only in even dimension. However, a similar approach might be fruitful in higher dimension by considering the Fefferman--Graham obstruction tensors, see \cite{fefferman1985}. We also believe that this approach could be used in the same spirit as Coulomb gauges for Willmore surfaces \cite{bernard2014,laurain2018,martino2023}, harmonic maps \cite{laurain2014} and Yang--Mills connections \cite{gauvrit2024,riviere2020,uhlenbeck1982}, in order to study the regularity and compactness properties of Bach-flat metrics. Indeed, we are able to provide (at least locally) a conformal metric satisfying an $\ve$-regularity property. It would be interesting to know whether we can solve $\lap_{g_0} J^{g_0}=0$ globally.\\
	
	It seems not possible to reduce the assumption $B^g\in L^p$ for $p>1$ to $p=1$. It would be interesting to look for possible counter-examples. Many directions of generalizations of \Cref{th:global_Huber_4D} can be considered. First, one can study the higher regularity, as in \cite{Her97}, were we assume the Bach tensor to be in $W^{k,p}$ up to the singularity in some coordinates, and deduce that the background metric belongs to $W^{k+2,2p}$. \Cref{th:global_Huber_4D} can be applied to Bach-flat metrics, but also the potential singularities arising along the Bach flow \cite{bahuaud2011,chen2023} since it is valid for a Bach tensor in $L^2$. Another important direction, would be to integrate by parts the Bach tensor in order to obtain conditions on the Cotton tensor (which would ask for one derivative less). In particular, the same conclusion might still hold if instead of assuming $B^g\in L^p$, we assume $\Cot^g\in L^{\frac{4p}{4-p}}$ for some $p>1$.
	
	\subsubsection*{Structure of the article:}	In \Cref{sec:Huber2D}, we prove \Cref{th:global_Huber_2D}. In \Cref{sec:Huber4D}, we prove \Cref{th:global_Huber_4D}. In \Cref{sec:asympt_flat}, we discuss the applications of \Cref{th:global_Huber_4D} to asymptotically flat manifolds, and in particular the proof of \Cref{cor:AF}. In \Cref{sec:immersion}, we prove \Cref{th:immersions}.

	\subsubsection*{Acknowledgements:}
	The authors would like to thank Rod Gover, Gilles Carron and Tristan Rivière for their expertise and important discussions. The second author is supported by Swiss National Science Foundation, project SNF 200020\textunderscore 219429.  
	
	\section{Huber Theorem for surfaces}\label{sec:Huber2D}
	
	\textit{Notations:} Given $r>0$, we denote $\D_r$ the Euclidean disk of $\C$ of radius $r$ centred at the origin. We denote $\D=\D_1$.
	We denote $\xi = (dx)^2 + (dy)^2$ the flat Euclidean metric.
	
	\subsection{Gauss curvature and moving frames.}\label{sec:examples}
	Before coming to the main result of this section, let us give the definition of the Gauss curvature in the moving frame approach, see \cite{doCarmo} for details. Let $g$ a smooth metric on $\D$, and $(E_1,E_2)$ a smooth orthonormal moving frame (obtain through Gramm--Schmidt algorithm for instance), then we consider the dual frame $\omega_1,\omega_2$ which gives  the smooth one form $\omega_{12}$ (the connection form) defined by\footnote{It is the unique one form such that $d\omega_1=\omega_{12}\wedge \omega_2$ and $d\omega_2=-\omega_{12}\wedge \omega_1$ .} 
	\begin{equation}
		\label{connection} 
		\omega_{12}=d\omega_1(E_1,E_2)\, \omega_1+d\omega_2(E_1,E_2)\, \omega_2.
	\end{equation}
	Then we have 
	\begin{equation}
		\label{Liouville}
		d\omega_{12}=-K_g \, \omega_1\wedge \omega_2=-K_g\, d\vol_g.
	\end{equation}
	Hence for a smooth metric $g$ on $\D\setminus \{0\}$, we perform the same construction but $K_g\, d\vol_g$ will be defined as $-d\omega_{12}$ in the sense of distributions across the origin. In particular there will be some possible Dirac mass at the origin. We list below a few examples.\\
	
	\textit{Example $(a)$, the flat case:} $g=\xi $.\\
	We choose $(E_1,E_2)= \left(\frac{\partial}{\partial x},\frac{\partial}{\partial y}\right)$, then it holds $\omega_{12}=0$ and $K_g \, d\vol_g=0$ on $\D$.\\
	
	Despite the simplicity of this example, we can introduce a first subtility. For any choice of $E_1$ on $\D$, its degree on $\partial \D$ is always $0$. However, this is not the case anymore on $\D\setminus\{0\}$. We will see that, in distributional sense, $d\omega_{12}$ depends on the frame because of the topology of $\D\setminus\{0\}$. \\

	\textit{Example $(a)bis$:}\\
	Let us consider $(E_1,E_2)= \left(\frac{\partial}{\partial r},\frac{1}{r}\frac{\partial}{\partial \theta}\right)$ which is orthonormal for $\xi$ on $\D\setminus\{0\}$. Then we easily see that $(\omega_1,\omega_2)=(dr,rd\theta)$ and $\omega_{12}=d\theta$ which gives in the distributional sense\footnote{We have $d\theta = *d(\log r)$, so that $dd\theta = d*d(\log r) = -(\lap \log r)\, dx\wedge dy = 2\pi\, \delta_0\, dx\wedge dy.$}
	$$-d\omega_{12}=-2\pi\, \delta_0\, dx\wedge dy \qquad  \text{ in } \Dr'(\D).$$
	However, the flat metric as no Dirac at the origin. To understand this, let us interpret this using the Gauss--Bonnet formula. By page 94 of \cite{doCarmo}, we have, in general, along a curve $\gamma$ which admits $X$ as a unit tangent, 
	$$k_g=\omega_{12}\left( X \right)\qquad \text{ on } \gamma.$$
	In our particular case, we obtain 
	$$k_g=\omega_{12}\left( \frac{\partial}{\partial \theta}\right) \qquad \text{ on } \partial \D.$$
	Denoting $d\sigma$ the length measure induced by $g$ on $\dr\D$, it holds 
	$$
	\int_\D -d\omega_{12}=-\int_{\partial \D} \omega_{12}=-\int_{\dr\D} k_g\, d\sigma.
	$$
	In order to recover the usual Gauss--Bonnet formula, we need to write
	$$\int_\D (-d\omega_{12}+2\pi\delta_0) +\int k_g\, d\sigma =2\pi.$$
	Therefore, we have to set 
	$$K_g\, d\vol_g	\coloneq -d\omega_{12}+2\pi\delta_0.$$
	
	More generally, if we have a first frame $(E_1,E_2)$ and a second one $(E_1', E_2')$ such that $E_1'$ on partial $\partial \D$ is a unit tangent vector, i.e. positively proportionnal	$\frac{\partial }{\partial \theta}$ and unitary for $g$. Then, from page 95 of \cite{doCarmo}, we have
	$$(\omega_{12}-\omega_{12}')=d\varphi \qquad \text{on }\dr\D,$$
	where $\varphi$ is the angle between $E_1$ and $E_1'$, that is to say the angle between $E_1$ and $\frac{\partial }{\partial \theta}$. Hence we obtain 
	$$
	\int_{\D} -d\omega_{12}=\int_{\partial \D} -\omega_{12}=-\int_{\partial \D} (\omega_{12}'+d\varphi) =-\int_{\dr\D} k_g\, d\sigma-2\pi\left(\mathrm{deg}\left({E_1}_{\vert\dr \D}\right)-1\right).
	$$
	where $\mathrm{deg}\left({E_1}_{\vert\dr \D}\right)$ is the degree of $E_1$ on $\partial \D$.
	Therefore, the general formula is 
	$$
	K_g\, d\vol_{g}=-d\omega_{12}+ 2\pi\, \deg\left({E_1}_{\vert\dr \D}\right)\, \delta_0.
	$$
	For instance, if $E_1=\frac{\partial }{\partial x}$ as in Example $(a)$, we recover the classical formula. We can also conformally modify the frame in order to have $\omega_{12}\in L^2$ by setting $(E_1,E_2)= \frac{1}{r}\left(\frac{\partial}{\partial r},\frac{1}{r}\frac{\partial}{\partial \theta}\right)$ but at the cost of having $K_{\frac{\delta}{r^2}}\, d\vol_{\frac{\delta}{r^2}}\not \in W^{-1,2}$.\\
	
	\textit{Example $(b)$, polynomial singularity:} \\
	The case where $g=n^2\vert z\vert^{2n-2} \xi$ is the pullback of the standard metric by the conformal map $f(z)=z^n$. We set  $e^{2\lambda}=n^2\vert z\vert^{2n-2}$ on $\D\setminus\{0\}$, and we  choose 
	$$
	(E_1,E_2)= \left(e^{-\lambda}\frac{\partial}{\partial x},e^{-\lambda}\frac{\partial}{\partial y}\right).
	$$
	This gives $\omega_{12}= -\lambda_y \, dx +\lambda_x \, dy=*d\lambda$, hence $\omega_{12}$ is in $L^1(\D)$, we can define its distibutionnal derivative on $\D$, and we get
	$$K_g \,d\vol_g=-d\omega_{12}=-2(n-1) \pi\, \delta_0\, dx\wedge dy.$$
	Hence, we easily check that the Gauss--Bonnet formula is valid on $\D$.\\
	
	\textit{Example $(c)$, essential singularity:} the case where $g=\vert z\vert^{-4}\, \vert e^\frac{2}{z}\vert\, \xi$ is the pullback of the standard metric by the conformal map $f(z)=e^\frac{1}{z}$. On $\D\setminus\{0\}$, we set 
	$$
		e^{2\lambda}=\vert z\vert^{-4} \vert e^\frac{2}{z}\vert = \frac{1}{|z|^4}\, e^{\frac{2}{|z|^2}\, \Re(z)}.
	$$ 
	In other words, it holds
	\begin{align}\label{eq:formula_lambda}
		\lambda = 2\log(|z|) + \frac{\Re(z)}{|z|^2}.
	\end{align}
	We then choose 
	$$
		(E_1,E_2)= \left(e^{-\lambda}\frac{\partial}{\partial x},e^{-\lambda}\frac{\partial}{\partial y}\right).
	$$
	As above, it holds $\omega_{12}= *d\lambda$, hence $\omega_{12}$ verifies $d\omega_{12}=0$ but does not lie in $L^1(\D)$. In order to define $K_g\, d\vol_g = -d\omega_{12}$ in a distributional sense, we take a test function $\vp\in C^{\infty}_c(\D)$ and need to know whether the following quantity has a limit or not as $\ve\to 0$ 
	\begin{align*}
		 \int_{\dr\D_{\ve}} \vp\, \omega_{12} = -\int_{\dr \D_{\ve}} \vp(x)\, \dr_r\lambda(x)\, dx.
	\end{align*}
	By \eqref{eq:formula_lambda}, we obtain 
	\begin{align*}
		\int_{\dr \D_{\ve}} \vp(x)\, \dr_r\lambda(x)\, dx & \ust{\ve\to 0}{=} 4\pi\, \vp(0)   + \int_{\s^1}  \left(\vp(0) +\ve \int_0^1 \dr_r \vp(\ve t\theta)\, dt \right)\, \left( \frac{-\cos(\theta)}{\ve^2} \right)\, \ve\, d\theta +o(1) \\[2mm]
		& \ust{\ve\to 0}{=} 4\pi\, \vp(0)   - \int_{\s^1} \cos(\theta) \left( \int_0^1 \dr_r \vp(\ve t\theta)\, dt \right)\, d\theta +o(1).
	\end{align*}
	Hence, the quantity $d\omega_{12}$ does not defines a distribution of order 0 at the origin.\\
	
		\textit{Example $(d)$, the annulus:} Let consider $\phi\colon \D \rightarrow \D\setminus \D_{1/2}$ given by $\phi(r,\theta)=\left(\frac{1}{2}+\frac{r}{2},\theta\right)$, then we can choose 
		$$
		(E_1,E_2)= \left(2\frac{\partial}{\partial r},\frac{1}{\frac{1}{2}+\frac{r}{2}}\frac{\partial}{\partial \theta}\right).
		$$ 
		We easily check that $\omega_{12}=d\theta$ and as in Example $(a)$, setting $K_g\,d\vol_g=-d\omega_{12}+2\pi \delta_0\equiv 0$ on $\D$, we easily check the validity of the Gauss--Bonnet formula
		$$\int_{\D} K_g\, d\vol_{g} +\int k_g\, d\sigma =2\pi.$$
		Here the underlying manifold is not an annulus with a smooth metric but a disc with a singular metric. The map $\phi$ is not conformal and we now check that this singular metric cannot be conformal to the flat one on the disk with a conformal factor in $W^{1,2}$. Indeed if there exists $\varphi \in  W^{1,2}(\D)$ such that $(\D\setminus \{0\}, e^{2\varphi}\delta)$ is isometric to the flat annulus, then, for every $\ve>0$, there exists a radius $r_{\ve}\in (\ve, 2\ve)$, such that, for every $p,q>1$ satisfying $\frac{1}{p}+\frac{1}{q}=1$, it holds
		$$ 
		\ve \int_{\partial \D_{r_{\ve}}} e^{\varphi}\, d\sigma = \int_{\D_{2\ve}\setminus \D_{\ve}} e^{\varphi} \, d\vol_{\xi}
		\leq  \left( \int_{\D_{2\ve}\setminus \D_{\ve}} e^{p\varphi}\, d\vol_{\xi} \right)^{\frac{1}{p}}\, \ve^{\frac{2}{q}}.
		$$ 
	 	By John--Nirenberg inequality (or Moser--Trudinger inequality), there will be a closed curve in the annulus enclosing $0$ (hence with diameter bounded below) with length $O(\epsilon^{1-\delta})$ for any $\delta>0$, which will be a contradiction.\\
	
	\subsection{Statement of the Huber theorem in dimension 2.}
	Examples $(a)bis$, $(c)$ and $(d)$ are the typical examples that we exclude with our assumptions. The main goal of this section is the following result.
	
	\begin{theorem}\label{th:Huber2D}
		Let $g$ be a smooth metric on $\D\setminus \{0\}$ such that
		\begin{enumerate}
			\item\label{it:initial_coframe} There exists a coframe $(\omega^1,\omega^2)$ on $\D\setminus \{0\}$ associated to a frame $(E_1,E_2)$ such that $\mathrm{deg}\left({E_1}_{\vert \partial \D_1}\right)=0$\footnote{Here $(E_1,E_2)$is the dual frame, moreover this condition is not restrictive when the frame is a limit of frame defined on the whole disc.} and its connection form $\omega_{12}$ lies in $L^{(2,\infty)}(\D)$.
			\item\label{it:K} There exists $m\in\R$ such that 
			\begin{align*}
				K_g\, \sqrt{\det g} \in m\, \delta_0 + W^{-1,2}(\D).
			\end{align*}
			
			\item\label{it:vol_growth} There exists $C>0$ and $\delta\in(0,1)$ such that for any $r\in(0,\frac{1}{2})$, it holds 
			\begin{align*}
				\vol_g(\D_{2r}\setminus \D_r) \leq C\, r^{2 - 2m-\delta}.
			\end{align*}
			
			\item\label{it:isop} There exists a function $\Lambda\colon \D\setminus \{0\}\to (0,\Lambda_{\max})$ for some $\Lambda_{\max}>0$, such that for any ellipsoid $E\Subset \D\setminus \{0\}$ centred at a point $x\in \D\setminus \{0\}$, it holds
			\begin{align*}
				\vol_g(E) \leq \Lambda(x)\, \ell_g(\dr E)^2.
			\end{align*}
			We assume that 
			\begin{align*}
				\int_{\D} \frac{((4\pi\Lambda(x))^2-1)^2}{|x|^2}\, dx <+\infty.
			\end{align*}
		\end{enumerate}
		Then, there exist an open set $\Omega\subset \R^2$ containing $0$, a $C^\infty$-diffeomorphism $\psi \colon \Omega\setminus \{0\}\to \D\setminus \{0\}$ and a function $\vp\in  W^{1,2}(\Omega)$ satisfying $\vp\in C^\infty(\Omega\setminus \{0\})$ and
		\begin{align*}
			\psi^*g = |x|^{-2m} e^{2\vp} \xi,
		\end{align*}
		where $\xi$ is the flat metric on $\D$. Furthermore, $\vp$ is a weak solution to the equation 
		$$
			-\lap_\xi \vp = -m\, \delta_0+ K_{\psi^*g} \, \sqrt{\det(\psi^*g)} \in W^{-1,2}(\D).
		$$		
	\end{theorem}
	The proof of \Cref{th:global_Huber_2D} follows by observing that \Cref{th:Huber2D} actually proves that the complex structure induced by $g$ on $\Sigma \setminus \{p_1,\dots,p_K\}$ extends smoothly to the full surface $\Sigma$. Since $\Sigma$ is closed, each complex structure is associated to a metric of constant curvature. Thus, we obtain \Cref{th:global_Huber_2D}.

	\subsection{Compatibility the hypothesis with the definition of the Gaussian curvature in a distributional way.}\label{sec:def_K}
	
	In this section, we use the notations of \Cref{th:Huber2D} and show how the Gaussian curvature in a distributional manner is coherent with  \Cref{it:initial_coframe} and \Cref{it:K}. To be as general as possible, we do not assume that the degree of $E_1$ vanishes from the begining.\\
	
	We consider the coframe $(\omega^1,\omega^2)$ given by \Cref{it:initial_coframe}. The connection form $\omega_{12}$ is the unique 1-form such that 
	\begin{align*}
		\begin{cases}
			d\omega^1 = \omega_{12}\wedge \omega^2 ,\\[2mm]
			d\omega^2 = -\omega_{12} \wedge \omega^1,
		\end{cases} \qquad \text{ in }\D\setminus \{0\}.
	\end{align*}
	Since we assume $\omega_{12}\in L^{(2,\infty)}(\D)$, for any $\ve\in(0,\frac{1}{2})$, there exists $r_{\ve}\in(\frac{\ve}{2},\ve)$ such that 
	\begin{align*}
		\int_{\dr \D_{r_{\ve}}} |\omega_{12}|_{\xi} & \leq \frac{2}{\ve} \int_{\ve/2}^{\ve} \left( \int_{\dr \D_r} |\omega_{12}|_{\xi} \right)\, dr = \frac{2}{\ve} \int_{\D_{\ve}\setminus \D_{\ve/2}} |\omega_{12}|_{\xi}.
	\end{align*}
	By Hölder inequality, we obtain 
	\begin{align}\label{eq:bound_slice}
		\int_{\dr \D_{r_{\ve}}} |\omega_{12}|_{\xi} \leq C\, \|\omega_{12}\|_{L^{(2,\infty)}(\D_{\ve}\setminus \D_{\ve/2},\xi)}.
	\end{align}
	Thus, there exist $a \in \R$ and a sequence $r_i \to 0$ such that 
	\begin{align}\label{eq:limit_slice}
		\int_{\dr \D_{r_i}} \omega_{12} \xrightarrow[i\to +\infty]{} \alpha.
	\end{align}
	Concerning the Gaussian curvature, we have 
	\begin{align}\label{eq:K_dtheta}
		-d\omega_{12} = K_g\, \omega^1\wedge \omega^2 = K_g\, d\vol_g \text{ on } \D\setminus\{0\}.
	\end{align}
	We define the Gaussian curvature as a distribution in $\Dr'(\D)$ as follows,
	$$K_g\, d\vol_g =- d\omega_{12} + 2\pi\mathrm{deg}\left({E_1}_{\vert\D}\right)\delta_0,$$
	where $d\omega_{12}$ lies in $W^{-1,(2,\infty)}(\D)$ by assumption. Hence, given $\vp\in C^{\infty}_c(\D)$, we have
	\begin{align}\label{eq:def_K}
		\scal{d\omega_{12}}{\vp}_{\Dr',\Dr}&=-\int_\D d\vp \wedge \omega_{12}\\
		&= \lim_{i\to +\infty} \left(  \int_{\dr \D_{r_i}} \vp\, \omega_{12} +
		 \int_{\D\setminus \D_{r_i}} \vp\, d\omega_{12} \right) \\
		 &=2\pi \alpha \vp(0)-  \lim_{i\to +\infty} \int_{\D\setminus \D_{r_i}} \vp (K_g\, d\vol_g)_{\vert \D\setminus\{0\}} .
	\end{align}
	Hence, by \Cref{it:K}, there exists $K\in  W^{-1,2}$, such that 
	\begin{align}
		& d\omega_{12}= 2\pi \alpha\delta_0-K \label{do12},\\[2mm]
		& K_g\, d\vol_g =2\pi \left(\deg\left({E_1}_{\vert\dr \D}\right) -\alpha \right)\delta_0+ K. \nonumber
	\end{align}
	this provides $\mathrm{deg}\left({E_1}_{\vert\dr \D}\right)-\alpha=m$. So the singularity of the Gauss curvaturde is given by the singularity of the connection of the  frame renormalised by the degree of the frame.
	In the setting of the theorem it simply gives that 
	$$m=-\alpha =\lim_{r_i \rightarrow 0} 	\int_{\dr \D_{r_i}} \omega_{12}\, \llcorner\, \dr_{\theta},$$
	hence as a by product the limit does not depends on the $r_i$ sequence.

	\subsection[Existence of a frame with connection form in L2]{Existence of a frame with connection form in $L^2(\D,g)$.}
	\label{xitog}
	
	In this section, we show that the assumptions \Cref{it:K} implies the existence of a frame orthonormal for the metric $r^{2m}g$ whose connection form lies in $L^2(\D,\xi)$, and not just in $L^{(2,\infty)}(\D,\xi)$. This implies in particular that $K_{r^{2m}g}\, \sqrt{\det(r^{2m}g)}\in W^{-1,2}(\D,\xi)$, that is to say the Gauss curvature $K_{r^{2m}g}\, d\vol_{r^{2m}g}$ is more regular than $K_g\, d\vol_g$.\\
	
	\textit{Step 1: Definition of the new coframe.}\\
	We consider the Hodge decomposition, see chapter of \cite{iwaniec}, of $\omega_{12}$ for the Euclidean metric: there exists $a\in W^{1,(2,\infty)}(\D,\xi)$, $b\in W^{1,(2,\infty)}(\D,\xi; \Lambda^2 \R^2)$ and $h\in L^{(2,\infty)}(\D,\xi;\Lambda^1\R^2)$ such that (here the operator $d^*$ is computed with respect to $\xi$)
	\begin{align*}
		\omega_{12} = da + d^* b + h \text{ on } \D,
	\end{align*}
	with 
	\begin{align*}
		d^*a=0,\, db = 0\text{ on } \D, \qquad \text{and }\qquad 
		\begin{cases} 
			d h = 0 \text{ on } \D ,\\[1mm]
			d^* h = 0 \text{ on } \D.
		\end{cases} 
	\end{align*}
	By \eqref{do12}, we have the following system on the 1-form $(d^*b-\alpha\, d\theta)\in L^{(2,\infty)}(\D)$\footnote{It holds $d^*d\theta=0$ in $\Dr'(\D)$ since for any $\vp\in C^{\infty}_c(\D)$, we have 
	\begin{align*}
		\lim_{r\to 0} \int_{\D\setminus \D_r} \scal{d\theta}{d\vp}_{\xi} = \lim_{r\to 0} \left( \int_{\dr \D_r} \underbrace{(d\theta \, \llcorner \, \nu)}_{=0} \vp - \int_{\D\setminus \D_r} \underbrace{(d^*d\theta)}_{=*dd\log (\rho) =0}\, \vp \right) =0.
	\end{align*}
	}
	\begin{align*}
		\begin{cases}
			d(d^*b +m\, d\theta) = -K=-K_g\, d\vol_g + 2 \pi \, m\, \delta_0 \in W^{-1,2}(\D), \\[3mm]
			d^*(d^*b +m\, d\theta) = m\, *dd(\log(\rho)) = 0 \qquad \text{ in }\D.
		\end{cases}
	\end{align*}
	By Gaffney's inequality, see for instance \cite{Troyanov10}, we obtain that 
	\begin{equation}
		\label{L2}
		d^*b+m\, d\theta\in L^2(\D_{1/2}).
	\end{equation}
	We consider the coframe $(w^1,w^2)$ defined by 
	\begin{align*}
		w^1 + i\, w^2 \coloneq e^{m\log(r)+ia}\, (\omega^1 + i\, \omega^2) \qquad \text{ in }\D_{1/2}\setminus \{0\}.
	\end{align*}
	In other words, we define
	\begin{align*}
		\begin{cases} 
			\displaystyle w^1 = r^{m} \left[ \cos(a)\, \omega^1 - \sin(a)\, \omega^2 \right], \\[2mm]
			\displaystyle w^2 = r^{m} \left[ \sin(a)\, \omega^1 + \cos(a)\, \omega^2 \right].
		\end{cases} 
	\end{align*}
	We have
	\begin{align*}
		dw^1 & = \frac{m}{r}\, dr\wedge w^1 + r^{m}\left[ -\sin(a)\, da\wedge \omega^1 +\cos(a)\, \omega_{12}\wedge \omega^2 - \cos(a)\, da\wedge \omega^2 +\sin(a)\, \omega_{12}\wedge \omega^1 \right] \\[2mm]
		& = m\, d(\log r)\wedge w^1 + \left[ \omega_{12} - da   \right]\wedge w^2 \\[2mm]
		& = m\, [*_g d(\log r)]\wedge w^2 + \left[ \omega_{12} - da   \right]\wedge w^2 \\[2mm]
		& = \left[ \omega_{12} -da +m\, *_g d(\log r) \right]\wedge w^2.
	\end{align*}
	As well for $w^2$, we obtain 
	\begin{align*}
		dw^2 = -\left[ \omega_{12} -da + m\, *_g d(\log r) \right]\wedge w^1.
	\end{align*}
	Hence,  it holds 
	\begin{align}
		w_{12} & = \omega_{12} -da +m\, *_g d(\log r) \nonumber \\[2mm]
		 & =(d^*b +m\, *_g d(\log r)) + h. \label{eq:good_connection}
	\end{align}
	We obtain 
	\begin{align*}
		 w^1\otimes w^1 + w^2 \otimes w^2 
		& = r^{2m}\left[ \cos(a)\, \omega^1 - \sin(a)\, \omega^2 \right]\otimes \left[ \cos(a)\, \omega^1 - \sin(a)\, \omega^2 \right] \\[2mm]
		& \qquad + r^{2m}\left[ \sin(a)\, \omega^1 + \cos(a)\, \omega^2 \right]\otimes \left[ \sin(a)\, \omega^1 + \cos(a)\, \omega^2 \right] \\[2mm]
		& = r^{2m}\left[ \omega^1\otimes \omega^1 + \omega^2\otimes \omega^2\right].
	\end{align*}
	In other words, the frame $(w^1,w^2)$ is orthonormal for the metric $r^{2m}g$.\\

	\textit{Step 2: The curvature form $w_{12}$ lies in $L^2(\D,\xi)$.}\\
	In order to prove that $w_{12}$ lies in $L^2(\D,\xi)$, we will apply \Cref{it:isop} to show in \Cref{cl:eigenvalues} that the quotient of the eigenvalues of $g$ are controlled. This will allow us in \Cref{cl:log} to show that $*_gd\log(r)$ is close to $d\theta$ in the $L^2(\D,\xi)$ topology (here is an improvement of regularity at the Lorentz scale). Then by \eqref{eq:good_connection} and \eqref{L2}, we will have that  
	$$w_{12}\in L^2(\D,\xi).$$
	Consequently, the Gauss curvature of the metric $r^{2m}g$ verifies
	\begin{align}
		K_{r^{2m}g}\, \sqrt{\det r^{2m}g} =-dw_{12}  \in W^{-1,2}(\D).
	\end{align}
	\begin{claim}\label{cl:eigenvalues}
		For any $x\in \D\setminus \{0\}$, the matrix $h^{ij}(x)\coloneq \sqrt{\det g(x)}\, g^{ij}(x)$ verifies the pointwise inequalities between matrices
		\begin{align*}
			\delta^{ij} \leq  h^{ij}(x) < (4\pi\Lambda(x))^2\, \delta^{ij}.
		\end{align*}
	\end{claim}
	
	\begin{proof}
		We fix $x\in \D\setminus \{0\}$. Let $(e_1,e_2)$ be the eigenvectors of $g_{ij}(x)$ with unit Euclidean norm. We then have 
		\begin{align*}
			\begin{cases}
				g_{ij}(x)\, e_1^i\, e_1^j = \lambda_1(x), \\[2mm]
				g_{ij}(x)\, e_2^i\, e_2^j = \lambda_2(x), \\[2mm]
				g_{ij}(x)\, e_1^i\, e_2^j =0.
			\end{cases}
		\end{align*}
		Up to exchanging $e_1$ and $e_2$, we can assume that $\lambda_1\leq \lambda_2$.	Given $\alpha>0$, we consider the open set 
		\begin{align*}
			U_{\alpha}\coloneq \left\{ x + t\, e_1 + s\, \left(\frac{\lambda_1}{\lambda_2}\right)^{\frac{1}{2}}\, e_2 : t^2 + s^2 < \alpha^2 \right\}.
		\end{align*}
		For $\alpha$ small enough, we have $U_{\alpha}\subset  \D_{|x|/2}(x)\Subset \D\setminus \{0\}$. By \Cref{it:isop}, it holds 
		\begin{align*}
			 \vol_g(U_{\alpha}) \leq \Lambda(x)\, \ell_g(\dr U_{\alpha})^2.
		\end{align*}
		As $\alpha\to 0$, we have the following asymptotic expansions. On one hand, it holds
		\begin{align*}
			\ell_g(\dr U_{\alpha}) \ust{\alpha\to 0}{=} \alpha \int_{\s^1} \sqrt{\lambda_1\, \cos(t)^2 + \lambda_1 \, \sin(t)^2}\, dt + o(\alpha).
		\end{align*}
		On the other hand, we have 
		\begin{align*}
			\vol_g(U_{\alpha}) \ust{\alpha\to 0}{=} \alpha^2\, \pi\, \sqrt{\det g(x)} + o(\alpha^2).
		\end{align*}
		Hence, we obtain 
		\begin{align*}
			\pi \sqrt{\lambda_1\, \lambda_2} \leq 4\pi^2\, \Lambda(x)\, \lambda_1 .
		\end{align*}
		Hence, we obtain $\lambda_2 \leq (4\pi\Lambda)^{2}\, \lambda_1$. Consequently, the matrix $\sqrt{\det g}\, g^{ij}$ verifies the following pointwise inequalities between matrices (assuming $\lambda_1\leq \lambda_2$)
		\begin{align*}
			\begin{cases}
				\displaystyle \sqrt{\det g}\, g^{ij} \leq  \frac{ \lambda_2 }{ \lambda_1 } \leq (4\pi\Lambda(x))^2 , \\[4mm]
				\displaystyle \sqrt{\det g}\, g^{ij} \geq \frac{\lambda_1}{\lambda_2} \geq 1.
			\end{cases}
		\end{align*}
	\end{proof}
	
	As a corollary, we obtain that the 1-form $(*_gd[\log r]-d\theta)$ lies in $L^2(\D,\xi)$.
	
	\begin{claim}\label{cl:log}
		It holds 
		\begin{align*}
			\int_{\D} \big|  *_gd[\log r]-d\theta \big|^2_{\xi} <+\infty.
		\end{align*}
	\end{claim}
	
	\begin{proof}
		It holds
		\begin{align*}
			*_gd\log(r)  = *_h\left(\frac{dr}{r}\right) = *_h\left( \frac{x\, dx + y\, dy}{x^2 + y^2} \right)  = -\left( \frac{x}{r^2}\, h^{12} + \frac{y}{r^2}\, h^{22} \right)\, dx + \left( \frac{x}{r^2}\, h^{11} + \frac{y}{r^2}\, h^{21} \right)\, dy.
		\end{align*}
		We have 
		\begin{align*}
			*_{\xi} *_gd\log(r)  & = *_{\xi}\left[ -\left( \frac{x}{r^2}\, h^{12} + \frac{y}{r^2}\, h^{22} \right)\, dx + \left( \frac{x}{r^2}\, h^{11} + \frac{y}{r^2}\, h^{21} \right)\, dy \right] \\[2mm]
			& = -\left( \frac{x}{r^2}\, h^{12} + \frac{y}{r^2}\, h^{22} \right)\, dy - \left( \frac{x}{r^2}\, h^{11} + \frac{y}{r^2}\, h^{21} \right)\, dx.
		\end{align*}
		Hence, it holds 
		\begin{align*}
			 *_{\xi} *_gd\log(r) + d\log(r) = -\left( \frac{x}{r^2}\, h^{12} + \frac{y}{r^2}\, (h^{22} - 1) \right)\, dy - \left( \frac{x}{r^2}\, (h^{11}-1) + \frac{y}{r^2}\, h^{21} \right)\, dx.
		\end{align*}
		We conclude the proof using \Cref{cl:eigenvalues} together with the integrability condition \Cref{it:isop}.
	\end{proof}
	
	Given any 1-form $\omega = \omega_i\, dx^i$, we have
	\begin{align*}
		*_g \omega = \sqrt{\det(g)}\ \left[ -\left(\omega_1\, g^{12} + \omega_2\, g^{22}\right)\, dx^1 + \left(\omega_1\, g^{11} + \omega_2\, g^{21}\right)\, dx^2 \right].
	\end{align*}
	As a consequence of \Cref{cl:eigenvalues}, we have the pointwise inequality
	\begin{align*}
		|*_g\omega|_{\xi} \leq C(\Lambda_{\max})\, |\omega|_{\xi}.
	\end{align*}
	Using the connection form introduced in \eqref{eq:good_connection}, we obtain by \Cref{cl:log}
	\begin{align*}
		\int_{\D} |w_{12}|^2_g\, d\vol_g = \int_{\D} w_{12}\wedge *_g w_{12} \leq C(\Lambda_{\max}) \int_{\D} |w_{12}|^2_{\xi} <+\infty.
	\end{align*}

	\begin{remark}
		The property $w_{12}\in L^2(\D,g)$ provides the upper bound $\vol_{r^{-2m}g}(B_{r^{-2m}g}(x,r))\leq C\, r^2$ for geodesic balls. Indeed, denoting $\tilde{g} \coloneq r^{-2m} g$ for simplicity, we have
		\begin{align*}
			\frac{d}{dr}\vol_g(B_{\tilde{g}}(x,r)) = \ell_{\tilde{g}}(\dr B_g(x,r)), \qquad \frac{d^2}{dr^2}\vol_{\tilde{g}}(B_{\tilde{g}}(x,r)) = \int_{\dr B_{\tilde{g}}(x,r)} k_{\tilde{g}}\, d\vol_{\tilde{g}}.
		\end{align*}
		Using Gauss--Bonnet theorem, we obtain 
		\begin{align*}
			\frac{d^2}{dr^2}\vol_{\tilde{g}}(B_{\tilde{g}}(x,r)) = 2\pi\chi(B_{\tilde{g}}(x,r)) - \int_{B_{\tilde{g}}(x,r)} K_{\tilde{g}}\, d\vol_{\tilde{g}} =  2\pi\chi(B_{\tilde{g}}(x,r)) + \int_{B_{\tilde{g}}(x,r)} dw_{12}. 
		\end{align*}
		Integrating by parts, we obtain 
		\begin{align*}
			\frac{d^2}{dr^2}\vol_{\tilde{g}}(B_{\tilde{g}}(x,r)) = 2\pi\chi(B_{\tilde{g}}(x,r)) + \int_{\dr B_{\tilde{g}}(x,r)} w_{12}.
		\end{align*}
		Using that $\vol_{\tilde{g}}(B_{\tilde{g}}(x,0))=0$ and $\frac{d}{dr}\vol_{\tilde{g}}(B_{\tilde{g}}(x,r))|_{r=0}=0$, we obtain
		\begin{align*}
			\vol_{\tilde{g}}(B_{\tilde{g}}(x,r)) & = \int_0^r \int_0^s \left(2\pi\chi(B_{\tilde{g}}(x,t)) + \int_{\dr B_{\tilde{g}}(x,t)} w_{12} \right)\ dt\, ds \\[2mm]
			& \leq C\, r^2 + \int_0^r \left( \int_{B_{\tilde{g}}(x,s)} w_{12} \right)\, ds \\[2mm]
			& \leq C\, r^2 + r\, \vol_{\tilde{g}}(B_{\tilde{g}}(x,r))^{\frac{1}{2}}\, \|w_{12}\|_{L^2(B_{\tilde{g}}(x,r),{\tilde{g}})}.
		\end{align*}
		By Young inequality, we obtain 
		\begin{align*}
			\vol_{\tilde{g}}(B_{\tilde{g}}(x,r)) \leq C\left( 1 + \|w_{12}\|_{L^2(\D,g)}^2 \right) r^2.
		\end{align*}
		Consequently, if we had a control on the diameter for $g$ of closed curve, then the assumption \Cref{it:vol_growth} would not be necessary.
	\end{remark}

	\subsection{Existence of a Coulomb frame and conformal coordinates.}
	In this section, we prove that there exist an open set $\Omega\subset \C$, a point $p\in\Omega$ and a diffeomorphism $\psi\colon \Omega\setminus \{p\}\to \D\setminus \{0\}$ such that $\psi^*g$ is conformal to the Euclidean metric.\\

	We denote $h\coloneq (\det g)^{-\frac{1}{2}}\, g$, which verifies by \Cref{cl:eigenvalues} and \Cref{it:isop} the following pointwise inequalities between matrices
	\begin{align}\label{eq:ineq_h}
		\frac{\delta_{ij}}{(4\pi\Lambda_{\max})^2} < h_{ij} <  \delta_{ij}.
	\end{align}
	By the previous section, the following quantity is well-defined and finite 
	\begin{align}\label{eq:Coulomb}
		\begin{aligned} 
		E & \coloneq \inf\left\{ \int_{\D} |\theta_{12}|^2_g\, d\vol_g : \theta^1 + i\theta^2 = e^{iu} ( w^1 + i w^2),\ u\in W^{1,2}(\D) \right\} \\[3mm]
		& = \inf\left\{ \int_{\D} |\theta_{12}|^2_h\, d\vol_h : \theta^1 + i\theta^2 =  e^{iu} ( w^1 + i w^2),\ u\in W^{1,2}(\D) \right\}.
		\end{aligned} 
	\end{align}
	If $u\in W^{1,2}(\D)$, we have by \eqref{eq:ineq_h}, that
	\begin{align*}
		\int_{\D} |du|^2_h\, d\vol_h = \int_{\D} h^{ij}\, (\dr_i u)\, (\dr_j u)\, \sqrt{\det h}\, dx > \frac{1}{C} \int_{\D} |du|^2_{\xi}.
	\end{align*}
	Hence, there exists a minimizing frame  $(\alpha^1,\alpha^2)$ reaching $E$ with $\alpha_{12}\in L^2(\D,\xi)\cap L^2(\D,g)$. It satisfies the following Euler--Lagrange equation in $\Dr'(\D)$ proved in \cite[Lemma 4.1.3]{helein2002}:
	\begin{align}\label{eq:syst_Coulomb1}
		\begin{cases}
			d\left( *_h \alpha_{12} \right) = 0 & \text{ in }\D,\\[2mm]
			i^*(*_h  \alpha_{12}) = 0 & \text{ on }\dr\D,
		\end{cases}
	\end{align}
	where $i\colon \dr\D\rightarrow \overline{\D}$ is the inclusion map. Since $*_h\alpha_{12} = *_h w_{12} + *_h du \in L^2(\D)$, we can apply the Poincaré Lemma and obtain a function $\mu\in W^{1,2}(\D)$ such that 
	\begin{align}\label{eq:conf_factor}
		\begin{cases}
			*_h \alpha_{12} = d\mu & \text{ in }\D,\\[2mm]
			\mu = 0 & \text{ on }\dr \D.
		\end{cases}
	\end{align}
	We denote $(e_1,e_2)$ the associated frame of $(\alpha^1,\alpha^2)$ for $r^{-2m}g$. As proved in Lemma 5.4.1 and in the proof of Theorem 5.4.3 (Step 3) in \cite{helein2002}, we have 
	\begin{align}\label{eq:frame}
		\begin{cases}
			\displaystyle \left[ e^{\mu}\, e_1, e^{\mu}\, e_2 \right] =0 & \text{ in }\D\setminus \{0\} ,\\[2mm]
			\displaystyle d(e^{-\mu}\alpha^1) = d(e^{-\mu}\alpha^2) =0 & \text{ in }\D\setminus \{0\}.
		\end{cases}
	\end{align}
	By Frobenius theorem, see for instance \cite[Theorem 1.64]{warner1983} or \cite[Chapter 6]{spivak1999}, there exists an open set set $\Omega\subset \R^2$ and a maximal solution $\psi\colon \Omega\to \D\setminus \{0\}$ such that 
	\begin{align*}
		e^{\mu}\, e_i\circ\psi = \dr_i \psi \qquad \text{ in }\D\setminus \{0\}.
	\end{align*}
	Moreover we have by Poincaré lemma in $\D^*$, the existence of functions $\vp^i\in C^{\infty}(\D\setminus \{0\})$ and constants $\kappa^i\in\R$ such that 
	\begin{align*}
		e^{-\mu}\, \alpha^i = d\vp^i + \kappa^i\, d\theta \qquad \text{ in }\D\setminus \{0\}.
	\end{align*}
	As a consequence of \Cref{it:vol_growth}, we obtain that $\kappa^i=0$.
	
	\begin{claim}\label{cl:singularity}
		Up to changing $\D\setminus \{0\}$ by $\D_r\setminus \{0\}$ for $r>0$ small enough in the definition of $E$, it holds  $\kappa^i=0$. Moreover, we have a sequence of radii $(r_k)_{k\in\N}\subset (0,r)$ converging to 0 as $k\to +\infty$ such that 
		\begin{align}\label{eq:length}
			\int_{\dr \D_{r_k}} \left( |d\vp^1|_{\xi} + |d\vp^2|_{\xi}\right)\, dx \xrightarrow[k\to +\infty]{} 0.
		\end{align}
	\end{claim}
	\begin{proof}	
		By definition, it holds
		\begin{align*}
			\forall r\in(0,1),\qquad 2\pi\, r\, \kappa^i  = \int_{\dr \D_r} e^{-\mu}\, \alpha^i(\dr_{\theta})\, d\theta.
		\end{align*}
		For each $\ve\in (0,\frac{1}{2})$, there exists $r_{\ve}\in (\ve,2\ve)$ such that 
		\begin{align*}
			\int_{\dr \D_{r_{\ve}}} \left( |e^{-\mu}\, \alpha^1|_{\xi} + |e^{-\mu}\, \alpha^2|_{\xi}\right) & \leq \frac{2}{\ve} \int_{\D_{2\ve}\setminus \D_{\ve}} \left( |e^{-\mu}\, \alpha^1|_{\xi} + |e^{-\mu}\, \alpha^2|_{\xi} \right).
		\end{align*}
		By \eqref{eq:ineq_h}, we obtain 
		\begin{align}
			\int_{\dr \D_{r_{\ve}}} \left( |e^{-\mu}\, \alpha^1|_{\xi} + |e^{-\mu}\, \alpha^2|_{\xi}\right) & \leq \frac{C}{\ve} \int_{\D_{2\ve}\setminus \D_{\ve}} |e^{-\mu(x)}|\, \left(\det |x|^{2m}g(x)\right)^{\frac{1}{4}}\ dx \nonumber \\[2mm]
			& \leq C\, \|e^{-\mu}\|_{L^2(\D_{2\ve}\setminus \D_{\ve})}\, \frac{\vol_g(\D_{2\ve}\setminus \D_{\ve})^{1/2}}{\ve^{-m+1}}.\label{eq:small_alpha}
		\end{align}
		By \Cref{it:vol_growth}, we obtain 
		\begin{align*}
			|\kappa^1| + |\kappa^2| \leq C_{\delta}\, \|e^{-\mu}\|_{L^2(\D_{2\ve}\setminus \D_{\ve})}\, \ve^{-\delta}.
		\end{align*}
		Using Moser--Trudinger inequality (see for instance \cite[Theorem 7.15]{gilbarg2001}), we obtain $e^{-\mu}\in L^p(\D,\xi)$ for any $p\in[1,+\infty)$. Hence, can apply Hölder inequality to the left-hand side and letting $\ve\to 0$, we obtain $\kappa^i=0$. Hence, we have $e^{-\mu}\alpha^i = d\vp^i$ and the estimate \eqref{eq:small_alpha} leads to \eqref{eq:length}.
	\end{proof}
	
	Since $\kappa^i=0$, we obtain $\alpha^i =e^{\mu}\, d\vp^i$. Hence, we obtain 
	\begin{align*}
		\delta^i_j = \alpha^i(e_j)\circ\psi = e^{-\mu}\, \alpha^i\left( e^{\mu}\, e_j\right)\circ\psi = d\vp^i \left( \dr_j \psi\right) = \dr_j(\vp^i\circ\psi).
	\end{align*}
	Hence, $\psi$ is invertible with inverse $\psi^{-1} = (\vp^1,\vp^2)$. Moreover, it holds 
	\begin{align*}
		\psi^*(r^{2m}g) = \psi^*\left( \alpha^1\otimes \alpha^1 + \alpha^2\otimes\alpha^2 \right) = e^{-2\mu}\, \xi.
	\end{align*}
	In order to ensure that we obtained isothermal coordinates on a pointed open set, we now show that the preimages $\psi^{-1}(\D_s\setminus \{0\})$ are open sets with diameter converging to 0. This is a consequence of \eqref{eq:length}: if $s\leq r_k$ (where $r_k$ is defined in \Cref{cl:singularity}), then we have
	\begin{align*}
		\diam\Big( \psi^{-1}(\D_s\setminus \{0\}) \Big) \leq \diam\Big( \psi^{-1}(\D_{r_k}\setminus \{0\}) \Big) \leq \diam\Big( \psi^{-1}(\dr \D_{r_k}) \Big). 
	\end{align*}
	Hence, we obtain 
	\begin{align*}
		\diam\Big( \psi^{-1}(\D_s\setminus \{0\}) \Big) \leq \mathrm{Length}_{\xi}\Big( \psi^{-1}(\dr \D_{r_k}) \Big) \leq \int_{\dr \D_{r_k}} \Big( |d\vp^1|^2 + |d\vp^2|^2 \Big)^{\frac{1}{2}} \xrightarrow[k\to +\infty]{}0.
	\end{align*}
	Thus, the open set $\psi^{-1}(\D\setminus \{0\})$ is of the form $\Omega\setminus \{p\}$ for some open set $\Omega\subset \C$ and $p\in\Omega$.
	
	\subsection{Conformal factor and Liouville equation}
	Up to translation of the domain $\Omega$ obtained in the previous section, we can assume that $p=0$. Up to shrinking $\Omega$, we consider the situation $\Omega\setminus \{p\}=\D\setminus \{0\}$. We now have a metric $g=r^{-2m}\, e^{2u}\, \xi$ where the conformal factor $u\colon \D\setminus\{0\}\to \R$ (given by $u\coloneq \mu\circ\psi$) is a solution to \eqref{eq:conf_factor}, that is to say
	\begin{align*}
		-\lap_{\xi} u = K_{r^{2m}g}\, \sqrt{\det r^{2m}g} =  K_g\, \sqrt{\det g}- m\delta_0\in W^{-1,2}(\D).
	\end{align*}
	By elliptic regularity, we obtain $u\in W^{1,2}(\D_{1/2},\xi)$. Moreover, there exists a universal constant $C>0$ such that the following estimate holds
	\begin{align*}
		\|\g u\|_{L^2(\D_{1/2})} \leq C\, \left\| K_g\, \sqrt{\det g}- m\delta_0 \right\|_{W^{-1,2}(\D,\xi)} + C\, \|u\|_{L^2(\D,\xi)}.
	\end{align*}
	Moreover, we have by \Cref{it:vol_growth}
	\begin{align*}
		\|u\|_{L^2(\D,\xi)} \leq \|e^{2u}\|_{L^2(\D,\xi)}  = \vol_{r^{-2m}g}(\D) <+\infty.
	\end{align*}

	\section{Huber Theorem for 4-dimensional manifolds}\label{sec:Huber4D}
	
	\subsection{Main result}\label{sec:main}
	Consider $g$ an orbifold metric on $C(\s^3/\Gamma)\setminus \{0\}$, where $\Gamma$ is any finite subgroup of $\SOo{4}$. By definition, we can lift this metric to a smooth metric on $\B\setminus \{0\}$, where $\B$ is the unit ball of $\R^4$. Then, by working only in the space of functions which are invariant by $\Gamma$ (finite group) which does not change the analysis, we can assume that $\Gamma=\{1\}$. Hence, the goal of this section is to prove the following result.
	
	\begin{theorem}\label{th:Huber4D}
		Let $p\in(1,2]$. Consider a smooth metric $g$ on $\B\setminus\{0\}$ satisfying the following assumptions:
		\begin{enumerate}[label= $A$.\arabic*]
			\item\label{asump:finite_volp} The volume of $\B$ is finite: $\vol_g(\B)<+\infty$.
			\item\label{asump:Bachp} We have $\Riem^g\in L^2(\B,g)$ and  $B^g\in L^p(\B,g)$.
			\item\label{asump:Sobolevp} There exists a constant $\gamma_S>0$ such that the following Sobolev inequality holds:
			\begin{align*}
				\forall u\in C^{\infty}_c(\B\setminus \{0\}),\quad
				\|u\|_{L^4(\B)}^2 \leq \gamma_S \int_{\B} |\g u|^2_g \, d\vol_g.
			\end{align*}
			\item\label{asump:Laplacianp} There exists a constant $\gamma_L>0$ such that the following regularity estimate holds. For any $u\in C^{\infty}_c(\B\setminus \{0\})$, it holds
			\begin{align*}
				\|\g u\|_{L^4(\B,g)} \leq \gamma_L \|\lap_g u\|_{L^2(\B,g)} .
			\end{align*}
		\end{enumerate}
		There exists a constant $\ve>0$ depending only on $\gamma_S$ and $\gamma_L$ such that the following holds. Assume that
		\begin{align*}
			\|\Riem^g\|_{L^2(\B,g)} \leq \ve.
		\end{align*}
		Then there exist a diffeomorphism $\psi \colon \B\setminus\{0\}\to \B\setminus\{0\}$, a conformal factor $u\in C^{\infty}(\B\setminus \{0\})$, and a metric $h\in W^{2,2p}(\B,\xi)\cap C^{\infty}(\B\setminus \{0\})$ such that
		\begin{align*}
			\psi^*g = e^{2u}h.
		\end{align*}
	\end{theorem}
	
	Before proving the above result, we discuss all the assumptions of analytic nature. We will need to use Sobolev inequalities and elliptic regularity estimates for the metric $g$. To obtain uniform estimates, we need to assume that the constants are valid up to the origin.

	\subsubsection*{The isoperimetric inequality implies all the Sobolev embeddings.}
	
	The isoperimetric inequality is linked to Sobolev embeddings, see for instance \cite{carron1996}. From \cite[Chapter IV, Section 3, Theorem 4]{chavel1984}, we know that 
	\begin{align*}
		\Lambda^4 \coloneq \inf_{f\in C^{\infty}_c(\B\setminus \{0\})\setminus \{0\}} \frac{ \|\g f\|_{L^1(\B,g)}^4 }{ \|f\|_{L^{\frac{4}{3}}(\B,g)}^{3} } = \inf_{\Omega\subset \B\setminus \{0\}} \frac{\vol(\dr \Omega)^4}{\vol(\Omega)^3}.
	\end{align*}
	Hence, an isoperimetric inequality involves the following Sobolev inequality: for any $f\in C^{\infty}_c(\B\setminus \{0\})\setminus \{0\}$, it holds
	\begin{align}\label{eq:Sobolev_L1}
		\|f\|_{L^{\frac{4}{3}}(\B,g)} \leq \Lambda\, \|\g f\|_{L^1(\B,g)}.
	\end{align}
	By choosing $f= u^3$ for some $u\in C^{\infty}_c(\B\setminus \{0\})$, we obtain
	\begin{align*}
		\left( \int_{\B} u^4\, d\vol_g \right)^{\frac{3}{4}} = \|u^3\|_{L^{\frac{4}{3}}(\B,g)} & \leq 3\Lambda \int_{\B} |u|^2 |\g u|\, d\vol_g \leq 3\Lambda\, \|u\|_{L^4(\B,g)}^2\, \|\g u\|_{L^2(\B,g)}.
	\end{align*}
	We deduce the following Sobolev inequality:
	\begin{align*}
		\|u\|_{L^4(\B,g)} \leq 3\, \Lambda\, \|\g u\|_{L^2(\B,g)}.
	\end{align*}
	We now make another choice in \eqref{eq:Sobolev_L1}. Namely,we chose $f=e^{\gamma|u|} -1$ for some $\gamma>0$ small to be chosen later, and $u\in C^{\infty}_c(\B\setminus \{0\})$ such that $\|\g u\|_{L^4(\B,g)} = 1$, then it holds
	\begin{align*}
		\left\| e^{\frac{4\gamma}{3}|u|} \right\|_{L^1(\B,g)}^{\frac{3}{4}} & \leq \left\| e^{\gamma|u|} - 1\right\|_{L^{\frac{4}{3}}(\B,g)} + \vol_g(\B)^{\frac{3}{4}}  \leq \Lambda\gamma \int_{\B}  e^{\gamma|u|} |\g u|\, d\vol_g + \vol_g(\B)^{\frac{3}{4}}.
	\end{align*}
	By Hölder inequality and using the assumption $\|\g u\|_{L^4(\B,g)} = 1$, we obtain
	\begin{align*}
		\left\| e^{\frac{4\gamma}{3}|u|} \right\|_{L^1(\B,g)}^{\frac{3}{4}}\leq \Lambda\gamma \left\| e^{\frac{4\gamma}{3}|u|} \right\|_{L^1(\B,g)}^{\frac{3}{4}} + \vol_g(\B)^{\frac{3}{4}}.
	\end{align*}
	Chosing now $\gamma = (2\Lambda)^{-1}$, we obtain 
	\begin{align*}
		\left\| e^{\frac{4\gamma}{3}|u|} \right\|_{L^1(\B,g)}^{\frac{3}{4}}\leq 2\vol_g(\B)^{\frac{3}{4}}.
	\end{align*}
	
	\subsubsection*{Sobolev inequalities for hypersurfaces.}
	
	For the applications to generalized Willmore hypersurfaces, we can rely on Sobolev inequalities on hypersurfaces of $\R^5$, see for instance \cite[Theorem 1.1]{cabre2022}, \cite[Section 7]{allard1972}, \cite[Theorem 2.1]{michael1973} or \cite[Theorem 1]{brendle2021}. For any hypersurface $\Sigma^4\subset \R^d$ and $p\in[1,4)$, there exists a constant $C=C(p,d)>0$ such that 
	\begin{align*}
		\forall\vp \in C^\infty_c(\Sigma),\qquad \left( \int_{\Sigma} |\vp|^{\frac{4p}{4-p}}\, d\vol_{\Sigma} \right)^{\frac{4-p}{4}} \leq C(p)\int_\Sigma |\g^{\Sigma} \vp|^p + |H\vp|^p\ d\vol_\Sigma.
	\end{align*}
	By Hölder inequality and assuming that $\|H\|_{L^4(\Omega)}$ is small enough for some open set $\Omega\subset \Sigma$, we obtain the following Sobolev inequality:
	\begin{align}\label{eq:Soblev_immersions}
		\forall \vp \in C^\infty_c(\Omega),\qquad  \|\vp\|_{L^{\frac{4p}{4-p}}(\Omega)} \leq 2\, C(p)\, \|\g^\Sigma \vp\|_{L^p(\Omega)}.
	\end{align}
	
	\subsubsection*{Elliptic regularity}
	
	The assumption \ref{asump:Laplacianp} can be understood as a consequence of the following estimate on the Green kernel $G$ of the Laplacian $\lap_g$ with Dirichlet boundary data. We define $G\colon (\B\setminus \{0\})^2 \to \R$ as the solution to the following equation: for any $y\in \B\setminus \{0\}$, the function $G_y\coloneqq G(\cdot,y)$ satisfies
	\begin{align*}
		\begin{cases}
			-\lap_g G_y = 0 & \text{in }\B\setminus\{0\},\\[1mm]
			G_y = 0 & \text{on }\dr(\B\setminus \{0\}).
		\end{cases}
	\end{align*}  
	Then \ref{asump:Laplacianp} is a consequence of the existence of a constant $C>0$ such that 
	\begin{align*}
		\forall y\in \B\setminus \{0\},\qquad \sup_{\lambda>0}\, \lambda^{\frac{4}{3}}\, \vol_g\Big(\left\{ x\in \B\setminus \{0\} : |d G_y(x)|_{g(x)} > \lambda \right\}\Big) \leq C.
	\end{align*}
	It is proved in \cite{carron1996} that such an estimate is a consequence of a Sobolev inequality. Similarly as for \eqref{eq:Soblev_immersions}, we can also see the assumption \ref{asump:Laplacianp} as a localized version for a coercivity condition on the Paneitz operator
	\begin{align*}
		\forall f\in C^{\infty}(\B\setminus \{0\}), \qquad P_g(f) \coloneq \lap_g^2 f + \di_g\left( 2\, \Ric(\g^g f,\cdot) - \frac{2}{3}\, \Scal^g\, \g^g f \right).
	\end{align*}
	Indeed, \ref{asump:Laplacianp} is a consequence of the following inequality together with a smallness assumption on $\|\Ric\|_{L^2(\B,g)}$
	\begin{align*}
		\forall f\in C^{\infty}_c(\B\setminus \{0\}),\qquad \int_{\B} f\, (P_g f)\ d\vol_g \geq C\, \left( \int_{\B} |\g f|^4_g\, d\vol_g  \right)^{\frac{1}{2}}.
	\end{align*}

	\subsection{Definition of the Schouten tensor and relations}\label{sec:Schouten}
	In this section, we define the Schouten tensor and provide a few remarks on the Bach tensor.\\
	
	Let $g$ be a smooth metric on $\Mr = \B^4 \setminus \{0\}$. The Schouten tensor of $g$ is given by 
	\begin{align*}
		\Sch_{ab}^g \coloneq \frac{1}{2} \left( \Ric^g_{ab} - \frac{\Scal^g}{6} g_{ab}\right), & & J^g \coloneq \tr_g\, \Sch^g = \frac{1}{6}\, \Scal^g.
	\end{align*}
	If $g_u \coloneq e^{2u}g$ and $\lap_g = (\g^g)^a (\g^g)_a$ is the Laplace--Beltrami operator, then it holds, see for instance \cite[Theorem 1.159]{besse2008}:
	\begin{align}
		\Sch^{g_u}_{ab} &= \Sch^g_{ab} - \g_a^g \g^g_b u + (\g_a^g u)(\g_b^g u) - \frac{|du|^2_g}{2} g_{ab},\label{eq:conf_change_Schouten}\\[2mm]
		J^{g_u} &= e^{-2u} \Big( J^g - \lap_g u - |du|^2_g \Big). \label{eq:conf_chang_Scal}
	\end{align}
	
	Using the Bach tensor, we now deduce the equation relating $\lap \Sch$ and $\Hess(J)$ obtained in \cite[Section 2.1]{tian2005}. This is the fundamental system used to obtain $\ve$-regularity. For any metric $g$, it holds:
	\begin{align}\label{eq:Bach}
		(\lap_g \Sch^g )_{ij} & = B^g_{ij} + \g_{ij}^g J^g   + 4\, (\Sch^g)_i^{\ p}\, \Sch^g_{pj} - |\Sch^g|^2_g\, g_{ij} - 2\, (\Sch^g)^{pk}\, W^g_{kipj},
	\end{align}
	where $B^g$ is the Bach tensor.	Thanks to the Chern--Gauss--Bonnet formula, the Euler--Lagrange equation of $\int |\Sch|^2$ is the same as the Euler--Lagrange equation of $\int J^2$ when we restrict ourselves to a conformal class:
	\begin{align}\label{eq:CGB_boundary}
		\begin{aligned} 
			& 32\pi^2 \chi(\B\setminus \B_r) = \int_{\B\setminus \B_r} |W^g|^2_g + 8((J^g)^2 - |\Sch^g|^2_g)\ d\vol_g  \\[2mm]
			& + 8\int_{\dr (\B\setminus \B_r)} \left( \frac{1}{2} \Scal^g H - \Ric^g(\nu,\nu) H - \Riem^{\gamma}_{\ \alpha\beta \gamma} \II^{\alpha\beta} + \frac{1}{3} H^3 -H|\II|^2_g + \frac{2}{3} \tr_{g_{|\dr \B_r}} (\II^3) \right)\, d\vol_g. 
		\end{aligned} 
	\end{align}

	\subsection{Strategy}
	\label{strategy}
	Consider a metric $g$ on $\B\setminus \{0\}$. By \eqref{eq:Bach}, we have the following system where $\Riem \star\Sch$ stands for linear combinations of different contractions between the Riemann tensor and the Schouten tensor
	\begin{equation}\label{eq:lap_Sch}
			\lap_g \Sch^g  - \Hess_g(J^g) = B^g + \Riem\star \Sch^g.
	\end{equation}
	Since the Bach tensor is traceless (due to the conformal invariance), this system can be roughly understood as providing informations only on the traceless part of the Schouten tensor. Hence, some information is missing by look at the Bach tensor only. As observed in \cite[Equation (43)]{korzy2003} and \cite[Lemma 4.2]{gover2008}, the tractor bundle of $(\B\setminus \{0\},g)$ is Yang--Mills if and only if $B^g = 0$. Furthermore, it was proved in \cite[Theorem 1.1]{tao2004} or \cite[Theorem VI.9]{riviere2020}, that point singularities are removable for Yang--Mills connection. Hence, one could hope to get a similar result for Bach-flat manifolds. To do so, we first need to identify the Coulomb condition. For Yang--Mills connections, this condition can be understood as finding a critical point of the $L^2$-norm of the connection. Thanks to \cite[Equation (15)]{gover2008} and the Chern--Gauss--Bonnet formula \eqref{eq:CGB_boundary}, we are reduced to study the following energy:
	\begin{align*}
		\forall u \in W^{2,2}(\B,g),\qquad E(u) = \frac{1}{2} \int_{\B} (J^{g_u})^2\, d\vol_{g_u},
	\end{align*}
	where $g_u = e^{2u}g$. We prove in \Cref{lm:EL_J} that the Euler--Lagrange equation of $E$ is $\lap_{g_u} J^{g_u} = 1$. In particular, we obtained that $g_u$ is smooth as soon as $g$ is smooth in \Cref{lm:gr_smooth}.
	\begin{remark}\label{rk:J_cst}
		\begin{itemize}
			\item The main term $\int (J^{g_u})^2$ is positive and minimized when $J^{g_u} = 0$. Marques \cite[Theorem 1.2]{marques2007} proved that we can solve $\bar{J}=0$ under the condition that the first eigenvalue of the conformal Laplacian has to be positive, see for instance \cite{escobar1994}. Without this condition, it is unclear that we can solve $J^{g_u} = 0$ with estimates on $u$.
			
			\item If $\bar{J}=0$ and $\bar{B}=0$, then the $\ve$-regularity has been proved in \cite{tian2005,carron2014}. We need to bound the constant in Sobolev inequality, see for instance \cite{druet2002,aldana2021}, or an estimate on the growth of balls, see for instance \cite{tian2008}.
		\end{itemize}
	\end{remark}
	
	However, as in the 2-dimensional case, solving this problem directly on $\B$ is not obvious. Instead, we first solve it in $\B\setminus \B_r$, that is to say, we minimize the following functional:
	\begin{align*}
		\forall u \in W^{2,2}_0(\B\setminus \B_r,g),\qquad E_r(u) = \frac{1}{2}\int_{\B\setminus \B_r} (J^{g_u})^2\, d\vol_{g_u}.
	\end{align*}
	Then we pass to the limit $r\to 0$. To do so, we need some preliminary estimates. We prove an $\ve$-regularity in the spirit of \cite{tian2005,tian2008}. A priori, we have a system of the form $\lap u = fu + g$. If $f$ and $g$ are regular, then $u$ is regular as well. We then show that the limiting metric $g_0$ is continuous across the singularity by proceeding to a blow-up at the origin, see \cite{anderson1990,anderson2005}. In a neighbourhood of the origin, $(\B\setminus \{0\},\bar{g})$ will be diffeomorphic to a cone $C(\s^3/\Gamma)\setminus \{0\}$ where $\Gamma$ is a finite subgroup of $\SOo{4}$. Since $\B$ is simply connected, we obtain that $\Gamma$ is trivial. We can extend continuously $g_0$ to the origin by the flat metric, see \cite{bando1989,streets2010}.

	\subsection{The energy functional}\label{sec:energy_functional}
	
	In this section, we consider a smooth metric $g$ on $\B\setminus\{0\}$ satisfying the assumptions of \Cref{th:Huber4D}. We denote $\Omega_r \coloneq \B^4\setminus \B_r^4$. If $u\in C^\infty(\Omega_r)$, we define $g_u \coloneq e^{2u}g$. Given $r>0$, we define 
	\begin{align}\label{def:Er}
		\forall u\in C^\infty(\Omega_r),\qquad E_r(u) \coloneq \frac{1}{2} \int_{\Omega_r } (J^u)^2\, d\vol_{g_u}.
	\end{align}

	\begin{lemma}\label{lm:existence_minimizer}
		Up to replace $\B$ with $\B_{r_0}$ for some $r_0>0$ small enough, one can assume that 
		\begin{align}\label{hyp:upper_bound_E0}
			E_0(0) = \frac{1}{2}\int_{\B}(J^g)^2 \, d\vol_g < \frac{1}{4^6\, \gamma_L^4}.
		\end{align}
		Consider $(u_k)_{k\in\N} \subset C^\infty_c(\Omega_r)$ a minimizing sequence for the following optimization problem:
		\begin{align*}
			& \inf\left\{ E_r(w) : w\in C^\infty_c(\Omega_r),\ \|dw\|_{L^4(\Omega_r)} < \frac{1}{4\, \gamma_L} \right\} \\[2mm]
			& = \inf\left\{ E_r(w) : w\in W^{2,2}_0(\Omega_r,g),\ \|dw\|_{L^4(\Omega_r)} < \frac{1}{4\, \gamma_L} \right\}.
		\end{align*}
		Then $(u_k)_{k\in\N}$ weakly converges in $W^{2,2}_0$, up to a subsequence, to some $u_r\in W^{2,2}_0(\Omega_r,g)$ such that 
		\begin{align*}
			E_r(u_r) = \inf\left\{ E_r(w) : w\in C^\infty_c(\Omega_r),\ \|dw\|_{L^4(\Omega_r)} < \frac{1}{4\, \gamma_L}\right\}.
		\end{align*}
		In particular, $u_r$ is a critical point of $E_r$ in $W^{2,2}_0(\Omega_r,g)$ satisfying the following estimates:
		\begin{equation}\label{eq:est_du}
			\left\{
			\begin{aligned}
				& \|du_r\|_{L^4(\Omega_r,g)} < \frac{8}{3}\, \gamma_L\, E_r(0)^{\frac{1}{2}},\\[2mm]
				& \|\lap_g u_r\|_{L^2(\Omega_r,g)} \leq \|du_r\|_{L^4(\Omega_r,g)}^2 + 2\, E_r(0)^{\frac{1}{2}}.
			\end{aligned}
			\right.
		\end{equation}
	\end{lemma}

	\begin{proof}
		From\eqref{eq:conf_chang_Scal}, it holds
		\begin{align*}
			\forall k\in\N,\qquad E_r(u_k) = \frac{1}{2}\int_{\Omega_r} \left(J^g - \lap_g u_k - |du_k|^2_g \right)^2 \, d\vol_g\leq E_r(0).
		\end{align*}
		Since the exponential term is positive, we obtain
		\begin{align*}
			\left\| J^g - \lap_g u_k - |du_k|^2_g \right\|_{L^2(\Omega_r,g)} \leq \left(2\, E_r(0)\right)^{\frac{1}{2}}.
		\end{align*}
		Using the triangle inequality, we deduce 
		\begin{align*}
			\|\lap_g u_k\|_{L^2(\Omega_r,g)} &\leq \|J^g\|_{L^2(\Omega_r,g)} + \|du_k\|_{L^4(\Omega_r,g)}^2 +\left(2\, E_r(0)\right)^{\frac{1}{2}} \leq \|du_k\|_{L^4(\Omega_r,g)}^2 + 2^{\frac{3}{2}} E_r(0)^{\frac{1}{2}}.
		\end{align*}
		Since $\|du_k\|_{L^4(\Omega_r,g)} < (4\gamma_L)^{-1}$, we obtain
		\begin{align}\label{eq:boundW22_minimizing}
			\|\lap_g u_k\|_{L^2(\Omega_r,g)} &\leq \frac{1}{4\, \gamma_L}\|du_k\|_{L^4(\Omega_r,g)} + 4 E_r(0)^{\frac{1}{2}}.
		\end{align}
		Thanks to \ref{asump:Laplacianp}, we obtain
		\begin{align*}
			\|du_k\|_{L^4(\Omega_r,g)} &\leq \frac{1}{4}\|du_k\|_{L^4(\Omega_r,g)} + 4\, \gamma_L\,  E_r(0)^{\frac{1}{2}}.
		\end{align*}
		Hence, it holds
		\begin{align*}
			\|du_k\|_{L^4(\Omega_r,g)} \leq \frac{16}{3} \gamma_L\,  E_r(0)^{\frac{1}{2}}.
		\end{align*}
		Thanks to \eqref{hyp:upper_bound_E0}, we obtain:
		\begin{align*}
			\|du_k\|_{L^4(\Omega_r,g)} \leq \frac{1}{12\, \gamma_L}.
		\end{align*}
		Together with \eqref{eq:boundW22_minimizing}, we deduce that $(u_k)_{k\in\N}$ remains bounded in $W^{2,2}_0(\Omega_r,g)$ and also in the interior of the ball $\Br \coloneq \left\{ w\in W^{1,4}_0(\Omega_r): \|dw\|_{L^4(\Omega_r,g)} < (4 \gamma_L)^{-1} \right\}$.
	\end{proof}

	\begin{remark}\label{rk:Ineq_Riem}
		Since $u_r=0$ on $\dr\Omega_r$, $u_r$ also minimize the $L^2$-norm of $\Sch$ thanks to the Chern--Gauss--Bonnet formula \eqref{eq:CGB_boundary}. Thus, if $g_r \coloneq g_{u_r}$ then we have the following uniform bound:
		\begin{align*}
			\int_{\Omega_r} |\Riem^{g_r}|^2_{g_r}\, d\vol_{g_r} &\leq \int_{\B\setminus \B_r} |\Riem^{g}|^2_{g}\, d\vol_{g}  \leq \int_{\B\setminus \{0\}} |\Riem^{g}|^2_{g}\, d\vol_{g}.
		\end{align*}
	\end{remark}
	
	In the following lemma, we compute the Euler--Lagrange equation of $E_r$.
	
	\begin{lemma}\label{lm:EL_J}
		Let $u\in W^{2,2}_0(\Omega_r,g)$ be a critical point of $E_r$. Then, the metric $g_u$ satisfies the following equation:
		\begin{align*}
				\lap_{g_u} J^{g_u} = 0 \qquad  \text{in }\Dr'(\Omega_r).
		\end{align*}
	\end{lemma}

	\begin{proof}
		Consider a variation $(v_t)_{t\in(-1,1)}$ of $u$ with $v_0 = u$ and $\dot{v} \coloneq (\dr_t v_t)|_{t=0}$ having compact support in $\Omega_r$. We compute the first derivative thanks by using $g_{v_t} = e^{2v_t} g$ and the formula \eqref{eq:conf_chang_Scal}:
		\begin{align}
			0 = \left. \frac{d E_r(v_t)}{dt}\right|_{t=0} &= \frac{d}{dt}\left. \left( \frac{1}{2}\int_{\Omega_r} \left(J^g - \lap_g v_t - |dv_t|^2_g \right)^2 \, d\vol_g \right) \right|_{t=0} \nonumber \\[2mm]
			& = \int_{\Omega_r} \left(J^g - \lap_g u - |du|^2_g \right) \left(-\lap_g \dot{v} - 2\scal{du}{d\dot{v}}_g \right) \, d\vol_g \label{eq:computation_EL_J}\\[2mm]
			& = \int_{\Omega_r} e^{2u} J^{g_u}\left(-\lap_g \dot{v} - 2\scal{du}{d\dot{v}}_g \right) \, d\vol_g. \nonumber
		\end{align}
		After integration by parts, we obtain:
		\begin{align*}
			0 = \int_{\Omega_r} \dot{v}\left( -\lap_g(e^{2u}J^{g_u}) + 2\, \di_g(e^{2u}J^{g_u} \g^g u) \right)\, d\vol_g.
		\end{align*}
		Since this is valid for every $\dot{v}\in W^{2,2}_0$, we obtain:
		\begin{align*}
			0 & = -\lap_g(e^{2u}J^{g_u}) + 2\, \di_g(e^{2u}J^{g_u} \g^g u) \\[2mm]
			& = -\di_g\left( 2\, e^{2u}\, (\g^g u)\, J^{g_u} + e^{2u}\, \g^g J^{g_u} \right) + 2 \di_g(e^{2u}\, J^{g_u}\,  \g^g u) \\[2mm]
			& = -\di_g \left( e^{2u}\, \g^g J^{g_u} \right) \\[2mm]
			& = -\di_g\left( e^{4u}\, \g^{g_u} J^{g_u} \right)\\[2mm]
			&= -e^{4u} \lap_{g_u} J^{g_u}.
		\end{align*}
	\end{proof}
	
	Since $g$ is of class $C^{\infty}$, we obtain that $u_r$ is also of class $C^{\infty}$.
	
	\begin{lemma}\label{lm:gr_smooth}
		Let $u\in W^{2,2}_0(\Omega_r,g)$ be a critial point of $E_r$ obtained in \Cref{lm:existence_minimizer}. It holds $u\in C^\infty(\Omega_r)\cap C^0\left(\overline{\Omega_r}\right)$. Moreover, for any $\frac{1}{2}>s>r>0$ and $k\in\N$, there exist $C>1$ depending only on $s$, $k$, $\gamma_L$, $\gamma_S$, $\max_{1\leq i,j\leq 4}\|g_{ij}\|_{C^{k+2}(\Omega_s)}$ and $\max_{1\leq i,j\leq 4}\|g^{ij}\|_{C^{k+2}(\Omega_s)}$ such that 
		\begin{align}
			\|u\|_{C^k(\Omega_{2s})} \leq C.
		\end{align}
	\end{lemma}
	
	\begin{remark}
		The constant $C$ might blow up as $s\to 0$.
	\end{remark}
	
	\begin{proof}
		Thanks to \eqref{eq:computation_EL_J}, the Euler--Lagrange equation of $E_r$ can also be written as:
		\begin{align}\label{eq:EL_J2}
			-\lap_g\left( J^g - \lap_g u - |du|^2_g \right) +2\, \di_g\left( (J^g - \lap_g u - |du|^2_g)\, \g^g u \right)  = 0.
		\end{align}
		Hence the function $v \coloneq J^g - \lap_g u - |du|^2_g$ satisfies:
		\begin{align*}
			\forall p\in[1,\infty),\qquad \lap_g v \in  W^{-1,(\frac{4}{3},1)}(\Omega_r,g).
		\end{align*}
		Indeed, it holds $du \in L^{(4,2)}(\Omega_r)$, so that $(\lap_g u)\, du \in L^{(2,1)}(\Omega_r)$. We also have $du \in L^{(4,3)}(\Omega_r)$, so that $|du|^3 \in L^{(\frac{4}{3},1)}(\Omega_r)$. Hence, it holds $v \in W^{1,(\frac{4}{3},1)}(\Omega_r)\hookrightarrow L^{(2,1)}(\Omega_r)$, meaning that $\lap_g u \in L^{(2,1)}(\Omega_r)$. Thus, we obtain $u\in W^{2,(2,1)}(\Omega_r)\hookrightarrow C^0\left(\overline{\Omega_r}\right)$. Consequently, the function $J^{g_u} = e^{2u}v\in W^{1,\frac{4}{3}}(\Omega_r)$ is a solution to 
		\begin{align*}
			\dr_{\alpha} \left( e^{2u} g^{\alpha\beta}\, \sqrt{\det g}\, \dr_\alpha J^{g_u} \right) = 0.
		\end{align*}
		We now have $e^{2u} g^{\alpha\beta}\, \sqrt{\det g} \in W^{1,(4,1)}(\Omega_r) \subset C^0\left(\overline{\Omega_r}\right)$. By a straightforward adaptation of the proof of \cite[Theorem A1.1]{ancona2009}, we deduce that $J^{g_u} \in W^{1,p}$ for every $p>1$. We consider now the following equation:
		\begin{align*}
			\lap_g (e^u) = \left( e^{2u}J^{g_u} - J^g\right) e^u.
		\end{align*}
		The right-hand side belongs to $W^{1,4-\ve}(\Omega_r)$ for every $\ve\in(0,1)$. Thus, $e^u \in W^{2,4-\ve}(\Omega_r)$ for every $\ve\in(0,1)$. Thanks to Sobolev embeddings, we deduce that $e^u \in W^{1,p}(\Omega_r)$ for every $p\in[1,\infty)$. Since $u$ is continuous, we deduce that $u\in W^{1,p}(\Omega_r)$ for every $p\in[1,\infty)$. We now come back to the equation:
		\begin{align*}
			\lap_g u = |du|^2_g + e^{2u} J^{g_u} - J^g.
		\end{align*}
		We deduce that $u\in W^{3,p}(\Omega_r)$ for every $p\in[1,\infty)$. By a bootrstrap argument, we obtain $u\in C^\infty(\Omega_r)$.
	\end{proof}

	\subsection{Sobolev inequalities and integrability of the curvature tensors in the new metric}\label{sec:new_metric}
	
	Thanks to \Cref{lm:gr_smooth}, the metrics $g_{u_r} = e^{2u_r}\, g$ obtained in \Cref{lm:existence_minimizer} are uniformly bounded in $C^k_{\loc}(\B\setminus \{0\})$ for any $k\geq 1$. Up to a subsequence, we can pass to the limit $r\to 0$ and we obtain metric $g_0 \coloneq e^{2u}g$ which is $C^{\infty}$ on $\Omega_0 \coloneq \B\setminus \{0\}$ and satisfies $\lap_{g_0}J^{g_0}=0$ on $\Omega_0$. In this section, we prove some regularity estimates on the conformal factor $u$ in \Cref{lm:reg_eu}. We deduce some integrability of $B^{g_0}$ in \Cref{lm:integrability_Bach} and the Sobolev inequalities in \Cref{lm:Sobolev_gr}.\\

	Given $\vp \in C^\infty_c(\B\setminus \{0\})$, we denote 
	\begin{align*}
		(\vp)_{\B} \coloneq \frac{1}{\vol_g(\B\setminus \{0\})} \int_{\B\setminus \{0\}} \vp\, d\vol_g.
	\end{align*}
	We first prove that $e^{-u}\in W^{1,2}(\B,g)$.
	
	\begin{lemma}\label{lm:reg_eu}
		Let $\gamma_L>1$ and $\gamma_S>0$ be defined in \ref{asump:Laplacianp} and \ref{asump:Sobolevp}. Assume that
		\begin{align}\label{hyp:smallness_E0}
			\int_{\B} \frac{(J^g)^2}{2} \, d\vol_g \leq \frac{9}{4^5 \gamma^2_S}.
		\end{align}
		Then, it holds $e^{-u}\in W^{1,2}(\B,g)$ with the following estimates:
		\begin{align*}
			\begin{cases} 
				\displaystyle\|e^{-u}\|_{L^4(\B,g)}^2  \leq \frac{1}{4 \gamma_L^2}+ 8\, \vol_g(\B)^{\frac{1}{2}},\\[3mm]
				\displaystyle \|d(e^{-u})\|_{L^2(\B,g)}^2 \leq \frac{1}{8 \gamma_S^2}\left( \frac{1}{2 \gamma_L^2} + 8\vol_g(\B)^{\frac{1}{2}} \right).
			\end{cases} 
		\end{align*}
	\end{lemma}
	\begin{proof}
		Given $r>0$, we prove the following estimates for $u_r$ and then pass to the limit $r\to 0$: it holds $e^{-u_r}\in W^{1,2}(\Omega_r,g)$ with the following estimates
		\begin{align*}
			\|e^{-u_r}\|_{L^4(\Omega_r,g)}^2  \leq \frac{\gamma_S}{24\gamma_L^4}, & & \|d(e^{-u_r})\|_{L^2(\Omega_r,g)}^2 \leq \frac{1+\gamma_S}{32\gamma_L^4}.
		\end{align*}
		We compute the equation satisfied by $e^{-u_r}$:
		\begin{align*}
			\lap_g\left( e^{-u_r} \right) & = -\di_g\left( e^{-u_r} \g^g u_r\right) \\[2mm]
			&= e^{-u_r}|du_r|^2_g -e^{-u_r} (\lap_g u_r) \\[2mm]
			&= e^{-u_r} |du_r|^2_g + e^{-u_r}\left( |du_r|^2_g + e^{2u_r} J^{g_r} - J^g \right) \\[2mm]
			&= 2e^{-u_r} |du_r|^2_g + e^{-u_r}\left( e^{2u_r} J^{g_r} - J^g \right) .
		\end{align*}
		We multiply by $e^{-u_r}$ and integrate by parts:
		\begin{align*}
			& \int_{\dr \Omega_r} e^{-2u_r} (\dr_\nu u_r) d\vol_g - \int_{\Omega_r} |d(e^{-u_r})|^2_g\, d\vol_g \\[2mm]
			=&\ 2\int_{\Omega_r} e^{-2u_r} |du_r|^2_g + e^{-2u_r} \left( e^{2u_r} J^{g_r} - J^g \right)\, d\vol_g.
		\end{align*}
		We consider the first term of the left-hand side: since $u_r = 0$ on $\dr \Omega_r$, it holds
		\begin{align*}
			\int_{\dr \Omega_r} e^{-2u_r} (\dr_\nu u_r) d\vol_g & = \int_{\dr \Omega_r} (\dr_\nu u_r) d\vol_g \\
			& =\int_{\Omega_r} \lap_g u_r\, d\vol_g \\
			& = \int_{\Omega_r} -|du_r|^2_g - e^{2u_r}J^{g_r} + J^g\, d\vol_g.
		\end{align*}
		We end up with the following estimate:
		\begin{align*}
			& 3\int_{\Omega_r} |d(e^{-u_r})|^2\, d\vol_g\\
			= & - \int_{\Omega_r} e^{-2u_r} \left( e^{2u_r} J^{g_r} - J^g \right)\, d\vol_g -\int_{\Omega_r} |du_r|^2_g + e^{2u_r}J^{g_r} - J^g\, d\vol_g \\[2mm]
			\leq &\ \|e^{-2u_r} \|_{L^2(\Omega_r,g)}\|e^{2u_r} J^{g_r} - J^g\|_{L^2(\Omega_r,g)} + \vol_g(\Omega_r)^{\frac{1}{2}} \left( \|du_r\|_{L^4(\Omega_r,g)}^2 + \|e^{2u_r} J^{g_r} - J^g\|_{L^2(\Omega_r,g)} \right)\\[2mm]
			\leq &\ 4 E_r(0)^{\frac{1}{2}} \left( \|e^{-u_r}\|_{L^4(\Omega_r,g)}^2 + \|du_r\|_{L^4(\Omega_r,g)}^2 + E_r(0)^{\frac{1}{2}} \right).
		\end{align*}
		From \Cref{lm:existence_minimizer} and \eqref{hyp:upper_bound_E0}, we obtain:
		\begin{align}\label{eq:estimes_deu}
			3\int_{\Omega_r} |d(e^{-u_r})|^2\, d\vol_g \leq 4 E_r(0)^{\frac{1}{2}} \left( \|e^{-u_r}\|_{L^4(\Omega_r,g)}^2 + \frac{1}{4 \gamma_L^2}  \right).
		\end{align}
		Thanks to \Cref{asump:Sobolevp}, it holds
		\begin{align*}
			\|e^{-u_r}-1\|_{L^4(\Omega_r,g)}^2 \leq \frac{4\gamma_S}{3} E_r(0)^{\frac{1}{2}} \left( \|e^{-u_r}\|_{L^4(\Omega_r,g)}^2 + \frac{1}{4 \gamma_L^2}  \right)  .
		\end{align*}
		Thanks to \eqref{hyp:smallness_E0}, we obtain
		\begin{align*}
			\|e^{-u_r}\|_{L^4(\Omega_r,g)}^2 & \leq 4\|e^{-u_r}-1\|_{L^4(\Omega_r,g)}^2 + 4\, \vol_g(\Omega_r)^{\frac{1}{2}}\\[2mm]
			 & \leq 4\left( \frac{4\gamma_S}{3} E_r(0)^{\frac{1}{2}} \|e^{-u_r}\|_{L^4(\Omega_r,g)}^2 + \frac{1}{32 \gamma_L^2} \right) + 4\, \vol_g(\Omega_r)^{\frac{1}{2}} \\[2mm]
			& \leq \frac{1}{2}\|e^{-u_r}\|_{L^4(\Omega_r,g)}^2 + \frac{1}{8 \gamma_L^2}+ 4\, \vol_g(\Omega_r)^{\frac{1}{2}}.
		\end{align*}
		Hence, it holds
		\begin{align*}
			\|e^{-u_r}\|_{L^4(\Omega_r,g)}^2  \leq \frac{1}{4 \gamma_L^2}+ 8\, \vol_g(\Omega_r)^{\frac{1}{2}}.
		\end{align*}
		Coming back to \eqref{eq:estimes_deu} and using $\gamma_L>1$, we obtain
		\begin{align*}
			\|d(e^{-u_r})\|_{L^2(\Omega_r,g)}^2 \leq \frac{1}{8 \gamma_S^2}\left( \frac{1}{2 \gamma_L^2} + 8\, \vol_g(\Omega_r)^{\frac{1}{2}} \right).
		\end{align*}
	\end{proof}
	
	We now show that the Bach tensor $B^{g_0}$ satisfies certain integrability. 
	
	\begin{lemma}\label{lm:integrability_Bach}
		For any $q\in[1,+\infty)$, we have the following estimate:
		\begin{align*}
			\left\| B^{g_0} \right\|_{L^{\frac{2q}{1+q}}(\B,g_0)} \leq \left( \frac{1}{4 \gamma_L^2}+ 8\, \vol_g(\B)^{\frac{1}{2}} \right)^{\frac{q-1}{4q}} \left\| B^g \right\|_{L^q(\B,g)}.
		\end{align*}
	\end{lemma}
	
	\begin{proof}
		The Bach tensor satisfies the following transformation formula:
		\begin{align*}
			B^{g_0}_{ij} = e^{-2u} B^g_{ij}.
		\end{align*}
		Thus, it holds $|B^{g_0}|_{g_0} = e^{-4u} |B^g|_g$. If $q\geq 1$, we have $\frac{q-1}{q+1} + \frac{2}{q+1}=1$. Thus, we obtain
		\begin{align*}
			\int_{\B} |B^{g_0}|^{\frac{2q}{1+q}}_{g_0}\, d\vol_{g_0} & = \int_{\B} e^{4\left(1-\frac{2q}{1+q} \right) u}\, |B^g|^{\frac{2q}{1+q}}_g\, d\vol_g \\[2mm]
			& = \int_{\B} e^{4\frac{1-q}{1+q} u_r}\, |B^g|^{\frac{2q}{1+q}}_g\, d\vol_g \\[2mm]
			&\leq \left\|e^{-4\frac{q-1}{1+q}u} \right\|_{L^{\frac{q+1}{q-1}}(\B,g)} \left\| |B^g|^{\frac{2q}{1+q}}_g \right\|_{L^{\frac{q+1}{2}}(\B,g)} \\[2mm]
			& \leq \left\|e^{-u} \right\|_{L^4(\B,g)}^{\frac{q-1}{q+1}} \left\| B^g \right\|_{L^q(\B,g)}^{\frac{2q}{q+1}}.
		\end{align*}
		We conclude thanks to \Cref{lm:reg_eu}.
	\end{proof}

	We deduce from \Cref{lm:reg_eu} a Sobolev inequality for the metric $g_r$.
	
	\begin{lemma}\label{lm:Sobolev_gr}
		Assume that 
		\begin{align}\label{hyp:smallness_E0_2}
			\int_{\B} \frac{(J^g)^2}{2}\, d\vol_g \leq \left( \frac{3}{8\, \gamma_L} \right)^2\, \frac{1}{4\, \gamma_S^2}.
		\end{align}
		Then, for any $\psi\in C^\infty_c(\B\setminus \{0\})$, it holds
		\begin{align*}
			\left( \int_{\B} \psi^4\, d\vol_{g_0} \right)^{\frac{1}{2}} \leq 4\gamma_S \int_{\B} |d\psi|^2_{g_0}\, d\vol_{g_0}.
		\end{align*}
	\end{lemma}
	
	\begin{proof}
		Thanks to \eqref{eq:est_du}, we have 
		\begin{align}\label{hyp:smallness_du}
			\int_{\B} |du|^4_g\, d\vol_g \leq \frac{1}{4\gamma^2_S}.
		\end{align}
		Let $\psi\in C^\infty_c(\B\setminus \{0\})$ and $\vp \coloneq e^{u}\psi$. From \ref{asump:Sobolevp}, it holds
		\begin{align*}
			\left( \int_{\B} e^{4u} \psi^4\, d\vol_g\right)^{\frac{1}{2}} \leq \gamma_S \int_{\B} |d(e^{u}\psi)|^2_g \, d\vol_g.
		\end{align*}
		We write the above inequality in terms of the metric $g_r$:
		\begin{align*}
			\left( \int_{\B} \psi^4\, d\vol_{g_0} \right)^{\frac{1}{2}} & \leq \gamma_S \int_{\B} 2\psi^2 e^{2u}|du|^2_g +2e^{2u} |d\psi|^2_g\, d\vol_g\\[2mm]
			&\leq 2\gamma_S\int_{\B} |d\psi|^2_{g_0}\, d\vol_{g_r} + 2\gamma_S \int_{\Omega_r} \psi^2 e^{2u}|du|^2_g\, d\vol_g \\[2mm]
			&\leq 2\gamma_S\int_{\Omega_r} |d\psi|^2_{g_0}\, d\vol_{g_0} + 2\gamma_S \left( \int_{\B} \psi^4 \, d\vol_{g_0} \right)^{\frac{1}{2}} \|du\|_{L^4(\B,g)}^2.
		\end{align*}
		Thanks to \eqref{hyp:smallness_du}, we obtain
		\begin{align*}
			\left( \int_{\B} \psi^4\, d\vol_{g_0} \right)^{\frac{1}{2}} \leq 2\gamma_S\int_{\B} |d\psi|^2_{g_0}\, d\vol_{g_0} + \frac{1}{2} \left( \int_{\B} \psi^4\, d\vol_{g_0} \right)^{\frac{1}{2}}.
		\end{align*}
		We obtain the desired result by reabsorbing the last term of the left-hand side with the right-hand side.
	\end{proof}

	We now prove a volume bound for the metric $g_0$.
	
	\begin{lemma}\label{lm:finite_volume}
		Under the assumption \eqref{hyp:smallness_E0_2}, it holds 
		\begin{align*}
			\vol_{g_0}(\B) \leq 2\, (2\, \gamma_S)^{\frac{1}{2}}\, \vol_g(\B)^{\frac{1}{4}}.
		\end{align*}
	\end{lemma}
	\begin{proof}
		It suffices to prove that $e^u\in L^4(\B,g)$. To do so, we use the construction $u=\displaystyle \lim_{r\to 0}u_r$ and the fact that $u_r=0$ on $\dr (\B\setminus \B_r)$. We denote $\Omega_r\coloneqq \B\setminus \B_r$. Let $\delta>0$ a parameter to be chosen later and use the Sobolev inequality \ref{asump:Sobolevp} to the function 
		\begin{align*}
			v\coloneqq \exp\left(\delta \frac{|u_r|}{\|\g u_r\|_{L^4(\B,g)}}\right) \in 1+ W^{1,2}_0(\Omega_r,g).
		\end{align*}
		It holds
		\begin{align*}
			\left\| v \right\|_{L^4(\B\setminus \B_r,g)} & \leq \left\|v - 1 \right\|_{L^4(\B\setminus \B_r,g)} + \vol_g(\B)^{\frac{1}{4}} \\[2mm]
			& \leq \sqrt{\gamma_S}\left\| \g v \right\|_{L^2(\B\setminus \B_r,g)}+ \sqrt{\gamma_S}\, \vol_g(\B)^{\frac{1}{4}}.
		\end{align*}
		By Hölder inequality, we obtain
		\begin{align*}
			\left\| v \right\|_{L^4(\B\setminus \B_r,g)}\leq \delta \sqrt{\gamma_S}\left\| v \right\|_{L^4(\B\setminus \B_r,g)}+ \sqrt{\gamma_S}\, \vol_g(\B)^{\frac{1}{4}}.
		\end{align*}
		We choose $\delta = (2\gamma_S)^{-1/2}$ and we obtain
		\begin{align*}
			\left\| \exp\left( \frac{|u_r|}{(2\gamma_S)^{1/2}\|\g u_r\|_{L^4(\B,g)}}\right) \right\|_{L^4(\B\setminus \B_r,g)}\leq \frac{ \sqrt{2\, \gamma_S} }{ \sqrt{2}-1}\, \vol_g(\B)^{\frac{1}{4}}.
		\end{align*}
		By \eqref{hyp:smallness_E0_2} and \eqref{eq:est_du}, we obtain $(2\gamma_S)^{1/2}\|\g u_r\|_{L^4(\B\setminus \B_r,g)} \leq 1$, we obtain 
		\begin{align*}
			\left\|e^{u_r} \right\|_{L^4(\B\setminus \B_r,g)} \leq \left\|e^{|u_r|} \right\|_{L^4(\B\setminus \B_r,g)} \leq 2\sqrt{2\, \gamma_S}\, \vol_g(\B)^{\frac{1}{4}}.
		\end{align*}
	\end{proof}
	
	\subsection{Regularity estimates}
	
	In this section, we prove an $\ve$-regularity estimate in \Cref{pr:eps_regp}, inspired by Moser's iteration, see for instance \cite[Theorem 4.4]{han2011}. First, we need to prove some estimate on the derivatives of $J^{g_0}$.
	
	\begin{lemma}\label{lm:der_Jp}
		Let $g$ a metric on $\B\setminus \{0\}$ satisfying the hypothesis of theorem \Cref{th:Huber4D}, then there exists $\ve>0$ and $C>0$ depending on the operation $\star$ in \eqref{eq:lap_Sch}, $p$ and $\gamma_S$ such that the following holds. Assume that 
		\begin{align}\label{asump:Smallness_Riem0}
			\|\Riem^g\|_{L^2(\B,g)} \leq \ve.
		\end{align}
		Let $g_0$ be the metric conformal to $g$ on $\Omega_0$ defined at the beginning of \Cref{sec:new_metric} satisfying
		$$	\lap_{g_0} J^{g_0} = 0. $$
		Let $B_{g_0}(x,2s)\subset \B\setminus \{0\}$ be a geodesic ball and $\chi$ by a cut-off function such that $\chi=1$ in $B_{g_0}(x,s)$. It holds
		\begin{align*}
			\int_{\B} \chi^2\, \left| d\left( |J^{g_0}|^{\frac{p}{2}} \right) \right|^2_{g_0}\, d\vol_{g_0} \leq C\, \int_{\B} |d\chi|^2_{g_0}\, |J^{g_0}|^p\ d\vol_{g_0}.
		\end{align*}
	\end{lemma}
	
	\begin{proof}
		Let $\delta>0$. In order to shorten the notations, we denote $j_{\delta}\coloneq \left(\delta^2 + (J^{g_0})^2\right)^{\frac{p}{4}}$. We have $j_{\delta}>0$ and $j_{\delta}\in C^{\infty}(\Omega_0)$. Moreover, it holds
		\begin{align*}
			\g^{g_0} j_{\delta} = \frac{p}{2}\, \left( \delta^2 + (J^{g_0})^2 \right)^{\frac{p}{4}-1}\, J^{g_0}\, \g^{g_0} J^{g_0}.
		\end{align*}
		Since $\lap_{g_0} J^{g_0}=0$, the Laplacian of $j_{\delta}$ is given by
		\begin{align*}
			\lap_{g_0} j_{\delta} & = \frac{p}{2}\, \left( \delta^2 + (J^{g_0})^2 \right)^{\frac{p}{4}-1}\, \left| \g^{g_0} J^{g_0}\right|^2_{g_0} + \frac{p}{2}\left(\frac{p}{2}-2\right)\, \left(\delta^2 + (J^{g_0})^2\right)^{\frac{p}{4}-2}\, (J^{g_0})^2\, \left|\g^{g_0} J^{g_0}\right|^2_{g_0} \\[2mm]
			 & = \frac{p}{2}\, \left( \delta^2 + (J^{g_0})^2 \right)^{\frac{p}{4}-2}\, \left| \g^{g_0} J^{g_0}\right|^2_{g_0} \left[  \delta^2 +  \left(\frac{p}{2}-1\right)\,  (J^{g_0})^2 \right].
		\end{align*}
		Let $B_{g_0}(x,2s)\subset \B\setminus \{0\}$ be a geodesic ball and $\chi$ by a cut-off function such that $\chi=1$ in $B_{g_0}(x,s)$.	By integration by parts, we obtain 
		\begin{align*}
			& \int_{\B} \chi^2\, |\g^{g_0}j_{\delta}|^2_{g_0}\, d\vol_{g_0} \\[2mm]
			& = -\int_{\B} \scal{\g^{g_0}j_{\delta}}{ 2\chi\, \g^{g_0}\chi }_{g_0}\, j_{\delta} + \chi^2\, j_{\delta}\, (\lap_{g_0} j_{\delta})\ d\vol_{g_0} \\[2mm]
			& = -\int_{\B} \scal{\g^{g_0}j_{\delta}}{ 2\chi\, \g^{g_0}\chi }_{g_0}\, j_{\delta} + \frac{p}{2}\, \chi^2\,  \left|\g^{g_0} J^{g_0}\right|^2_{g_0}\, \left(\delta^2 + (J^{g_0})^2\right)^{\frac{p}{2}-2} \left[ \delta^2 + \left(\frac{p}{2}-1\right)\, (J^{g_0})^2 \right]\ d\vol_{g_0}.
		\end{align*}
		Hence, we obtain 
		\begin{align*}
			& \int_{\B} \chi^2\, |\g^{g_0}j_{\delta}|^2_{g_0}\, d\vol_{g_0} \\[2mm]
			& \leq \left( \int_{\B} \chi^2 \left| \g^{g_0} j_{\delta}\right|^2_{g_0}\, d\vol_{g_0}\right)^{\frac{1}{2}}\, \left(\int_{\B}  4 \left| \g^{g_0}\chi \right|^2_{g_0}\, j_{\delta}^2\, d\vol_{g_0}\right)^{\frac{1}{2}} \\[2mm]
			& \qquad + \frac{p}{2}\int_{\B} \chi^2\,  \left|\g^{g_0} J^{g_0}\right|^2_{g_0}\, \left(\delta^2 + (J^{g_0})^2\right)^{\frac{p}{2}-2} \left[- \delta^2 + \left(1-\frac{p}{2}\right)\, (J^{g_0})^2 \right]\ d\vol_{g_0}.
		\end{align*}
		By Young's inequality, we obtain for any $\theta\in(0,1)$,
		\begin{align}\label{eq:der_jdelta}
			\begin{aligned} 
			 (1-\theta)\, \int_{\B} \chi^2\, |\g^{g_0}j_{\delta}|^2_{g_0}\, d\vol_{g_0} 
			& \leq C(\theta)\, \int_{\B} \left| \g^{g_0}\chi \right|^2_{g_0}\, j_{\delta}^2\, d\vol_{g_0}  \\[2mm]
			& \quad + \frac{p}{2}\int_{\B} \chi^2\,  \left|\g^{g_0} J^{g_0}\right|^2_{g_0}\, \left(\delta^2 + (J^{g_0})^2\right)^{\frac{p}{2}-2} \left[ -\delta^2 + \left(1-\frac{p}{2}\right)\, (J^{g_0})^2 \right]\ d\vol_{g_0}.
			\end{aligned} 
		\end{align}
		We estimate the last term as follows
		\begin{align*}
			& \int_{\B} \chi^2\,  \left|\g^{g_0} J^{g_0}\right|^2_{g_0}\, \left(\delta^2 + (J^{g_0})^2\right)^{\frac{p}{2}-2} \left[ - \delta^2 + \left(1-\frac{p}{2}\right)\, (J^{g_0})^2 \right]\ d\vol_{g_0} \\[2mm]
			& = -\int_{\B} \chi^2\,  \left|\g^{g_0} J^{g_0}\right|^2_{g_0}\, \frac{\delta^2}{ \left(\delta^2 + (J^{g_0})^2\right)^{2-\frac{p}{2}} }\ d\vol_{g_0} + \left(1-\frac{p}{2}\right) \int_{\B} \chi^2\,  \left|\g^{g_0} J^{g_0}\right|^2_{g_0}\, \left(\delta^2 + (J^{g_0})^2\right)^{\frac{p}{2}-2}  \, (J^{g_0})^2 \ d\vol_{g_0} \\[2mm]
			& \leq 0 + \frac{4}{p^2}\left(1-\frac{p}{2}\right) \int_{\B} \chi^2\,  \left|\g^{g_0} j_{\delta}\right|^2_{g_0}\ d\vol_{g_0}.
		\end{align*}
		Coming back to \eqref{eq:der_jdelta}, we obtain 
		\begin{align*}
			\left(2-\theta - \frac{2}{p}\right)\, \int_{\B} \chi^2\, |\g^{g_0}j_{\delta}|^2_{g_0}\, d\vol_{g_0} 
			& \leq C(\theta)\, \int_{\B} \left| \g^{g_0}\chi \right|^2_{g_0}\, j_{\delta}^2\, d\vol_{g_0} .
		\end{align*}
		Since $1<p\leq 2$, we can choose $\theta =1-\frac{1}{p}$ and we obtain 
		\begin{align*}
			\int_{\B} \chi^2\, |\g^{g_0}j_{\delta}|^2_{g_0}\, d\vol_{g_0} 
			& \leq C(p)\, \int_{\B} \left| \g^{g_0}\chi \right|^2_{g_0}\, j_{\delta}^2\, d\vol_{g_0} .
		\end{align*}
		We obtain the conclusion by letting $\delta\to 0$.
	\end{proof}
	
	We now show the $\ve$-regularity estimate.
	\begin{theorem}\label{pr:eps_regp} Let $g$ a metric on $\B\setminus \{0\}$ satisfying the hypothesis of theorem \Cref{th:Huber4D}, then there exists $\ve>0$ and $C>0$ depending on the operation $\star$ in \eqref{eq:lap_Sch} and $\gamma_S$ such that the following holds. Assume that 
		\begin{align}\label{asump:Smallness_Riem}
			\|\Riem^g\|_{L^2(\B,g)} \leq \ve.
		\end{align}
		Then there exists a metric $g_0$ conformal to $g$ satisfying
		$$	\lap_{g_0} J^{g_0} = 0. $$
		\noindent
		Moreover, for  each geodesic ball $B_{g_0}(x,s)\subset \B\setminus \{0\}$, it holds
		\begin{align*}
			\|\Sch^{g_0}\|_{L^{2p}(B_{g_0}(x,s))} 
			\leq & \ \frac{C}{s^{\frac{2}{p}}}\, \vol_{g_0}(B_{g_0}(x,2s))^{\frac{1}{p}-\frac{1}{2}}\, \|\Sch^{g_0}\|_{L^2(B_{g_0}(x,2s))}   + C\|B^{g_0}\|_{L^{\frac{2p}{p+1}}(B_{g_0}(x,2s),g_0)}^{\frac{2}{p+1}}  .
		\end{align*}
	\end{theorem}
	
	\begin{remark}\label{rk:notcsc}
		This provides an $\ve$-regularity result for Bach-flat metrics without assuming the scalar curvature to be constant, see for instance \cite{carron2014,tian2005,tian2008}. 
	\end{remark}
	
	\begin{proof}
		The existence of $g_0$ has been justified at the beginning of \Cref{sec:new_metric}.
		The Schouten tensor satisfies (by \eqref{eq:lap_Sch}):
		\begin{align}\label{eq:syst_Sch}
			\begin{cases}
				\lap_{g_0}\Sch^{g_0} = \Hess_{g_0} J^{g_0} + B^{g_0} + \Riem^{g_0}\star\Sch^{g_0}, \\[2mm]
				\lap_{g_0} J^{g_0} = 0. 
			\end{cases}
		\end{align}
		We define $f \coloneq |\Sch^{g_0}|_{g_0}$. This function satisfies
		\begin{align}\label{eq:lap_f}
			\begin{aligned} 
			-\lap_{g_0} f & = -\di_{g_0}\left(\scal{\frac{\Sch^{g_0}}{|\Sch^{g_0}|_{g_0}}}{\g^{g_0} \Sch^{g_0}}_{g_0}\right) \\[2mm]
			& = -\frac{|\g \Sch^{g_0}|^2_{g_0}}{|\Sch^{g_0}|_{g_0}} + \scal{\frac{\Sch^{g_0}}{|\Sch^{g_0}|_{g_0}}}{-\lap_{g_0} \Sch^{g_0}}_{g_0}.
			\end{aligned}
		\end{align}
		Hence, for any function $\vp \in C^{\infty}_c(\B;[0,+\infty))$, we have
		\begin{align*}
			\int_{\B} \scal{df}{d\vp}_{g_0}\, d\vol_{g_0} \leq \int_{\B} \vp \scal{\frac{\Sch^{g_0}}{|\Sch^{g_0}|_{g_0}}}{ - \Hess_{g_0} J^{g_0} - B^{g_0} - \Riem^{g_0}\star \Sch^{g_0}}_{g_0}\, d\vol_{g_0}.
		\end{align*}
		Let $\chi \in C^{\infty}_c(\B\setminus \{0\};[0,1])$ be a cut-off function such that $\chi = 1$ in $B_{g_0}(x,s)$ and $\chi=0$ on $\B\setminus B_{g_0}(x,2s)$ for some $s>0$ such that $B_{g_0}(x,2s)\subset \B\setminus \{0\}$ and $|d\chi|_{g_0}\leq Cs^{-1}$. We consider the test function\footnote{Proceeding exactly as in \Cref{lm:der_Jp} and using \eqref{eq:lap_f} with the fact that $\Sch^{g_0}$ and $\lap_{g_0}\Sch^{g_0}$ lie in $L^{\infty}_{\loc}(\B\setminus \{0\})$, we have that $f^{\frac{p}{2}}\in W^{1,2}_{\loc}(\B\setminus \{0\})$, so that all the quantities involved are well-defined.} $\vp = \chi^2 f^{p-1}$:
		\begin{align*}
			& \int_{\B} (p-1)\chi^2 f^{p-2} |df|^2_{g_0} + 2\chi f^{p-1} \scal{d\chi}{df}_{g_0}\, d\vol_{g_0} \\[2mm]
			\leq &\ \int_{\B} \chi^2\left( f^{p-1} |B^{g_0}|_{g_0} + f^p\, |\Riem^{g_0}|_{g_0}\right)\, d\vol_{g_0} - \int_{\B} \chi^2 f^{p-2}\scal{\Sch^{g_0}}{\Hess_{g_0} J^{g_0}}_{g_0}\, d\vol_{g_0}.
		\end{align*}
		As a consequence of the above estimate, there exists a constant $C>0$ depending only on $p$ such that the following holds:
		\begin{equation}\label{eq:estdf1}
			\begin{aligned}
				 \int_{\B}  \left|d\left(\chi f^{\frac{p}{2}} \right) \right|^2_{g_0} \, d\vol_{g_0}
				\leq &\ C \int_{\B} |d\chi|^2 f^p +  \chi^2\left( f^{p-1} |B^{g_0}|_{g_0} + f^{p} |\Riem^{g_0}|_{g_0}\right)\, d\vol_{g_0} \\[2mm]
				&\ - \int_{\B} \chi^2 f^{p-2}\scal{\Sch^{g_0}}{\Hess_{g_0} J^{g_0}}_{g_0}\, d\vol_{g_0}.
			\end{aligned}
		\end{equation}
		We now estimate the last term by integration by parts:
		\begin{align*}
			- \int_{\B} \chi^2 f^{p-2}\scal{\Sch^{g_0}}{\Hess_{g_0} J^{g_0}}_{g_0}\, d\vol_{g_0}
			=  \int_{\B} \scal{\g^{g_0} J^{g_0}}{\di_{g_0}\left( \chi^2 f^{p-2} \Sch^{g_0} \right)}_{g_0}\, d\vol_{g_0}.
		\end{align*}
		Since $\di_{g_0}\Sch^{g_0} = \g^{g_0}J^{g_0}$, we obtain the following estimate:
		\begin{align*}
			& - \int_{\B} \chi^2 f^{p-2}\scal{\Sch^{g_0}}{\Hess_{g_0} J^{g_0}}_{g_0}\, d\vol_{g_0} \\[2mm]
			\leq &\ \int_{\B} 2\chi |d\chi|_{g_0} f^{p-1} |dJ^{g_0}|_{g_0} + |p-2|\, \chi^2 f^{p-2} |df|_{g_0} |dJ^{g_0}|_{g_0} + \chi^2 f^{p-2} |dJ^{g_0}|_{g_0}^2\, d\vol_{g_0}.
		\end{align*} 
		By Hölder inequality, we obtain
		\begin{equation} \label{eq:est_Badterm1}
			\begin{aligned} 
				- \int_{\B} \chi^2 f^{p-2}\scal{\Sch^{g_0}}{\Hess_{g_0} J^{g_0}}_{g_0}\, d\vol_{g_0}
				\leq &\  2\left( \int_{\B} |d\chi|^2_{g_0}\ f^p\, d\vol_{g_0} \right)^{\frac{1}{2}} \left(\int_{\B} \chi^2 f^{p-2} |dJ^{g_0}|_{g_0}^2\, d\vol_{g_0} \right)^{\frac{1}{2}} \\[2mm]
				& +|p-2|\, \left(\int_{\B} \chi^2 f^{p-2} |df|_{g_0}^2\, d\vol_{g_0} \right)^{\frac{1}{2}} \left( \int_{\B} \chi^2 f^{p-2} |dJ^{g_0}|^2_{g_0}\, d\vol_{g_0} \right)^{\frac{1}{2}}\\[2mm]
				& + \int_{\B} \chi^2 f^{p-2} |dJ^{g_0}|_{g_0}^2\, d\vol_{g_0}.
			\end{aligned}
		\end{equation}
		By Young's inequality, we obtain for any $\delta>0$:
		\begin{align*}
			 - \int_{\B} \chi^2 f^{p-2}\scal{\Sch^{g_0}}{\Hess_{g_0} J^{g_0}}_{g_0}\, d\vol_{g_0} 
			\leq &\  \delta  \int_{\B} |d\chi|^2_{g_0}\ f^p\, d\vol_{g_0}  + \delta \int_{\B} \chi^2 f^{p-2} |df|_{g_0}^2\, d\vol_{g_0}\\[2mm]
			& + C(\delta) \int_{\B} \chi^2 f^{p-2} |dJ^{g_0}|_{g_0}^2\, d\vol_{g_0}.
		\end{align*}
		Coming back to \eqref{eq:estdf1} and using Young inequality, we obtain 
		\begin{align*}
			\int_{\B}  \left|d\left(\chi f^{\frac{p}{2}} \right) \right|^2_{g_0} \, d\vol_{g_0} 
			\leq &\ C \int_{\B} |d\chi|^2 f^p +  \chi^2\left( f^{p-1} |B^{g_0}|_{g_0} + f^{p} |\Riem^{g_0}|_{g_0}\right)\, d\vol_{g_0} \\[2mm]
			&\ + C \int_{\B} \chi^2 f^{p-2} |dJ^{g_0}|^2_{g_0}\, d\vol_{g_0}.
		\end{align*}
		Thanks to \Cref{lm:Sobolev_gr}, we obtain
		\begin{align}
			\left( \int_{\B}  \chi^4 f^{2p}  \, d\vol_{g_0}\right)^{\frac{1}{2}} \leq &\ C\gamma_S\int_{\B} |d\chi|^2 f^p  d\vol_{g_0} \label{eq:est_f2p_1}\\[2mm]
			& +  C\gamma_S\int_{\B} \chi^2\left( f^{p-1} |B^{g_0}|_{g_0} + f^{p} |\Riem^{g_0}|_{g_0}\right)\, d\vol_{g_0} \label{eq:est_f2p_2}\\[2mm]
			&\ + C\gamma_S \int_{\B} \chi^2 f^{p-2} |dJ^{g_0}|^2_{g_0}\, d\vol_{g_0}. \label{eq:est_f2p_3}
		\end{align}
		We now estimate \eqref{eq:est_f2p_2} by Hölder inequality, \Cref{lm:integrability_Bach} and using the assumption $\supp(\chi)\subset B_{g_0}(x,2s)$: 
		\begin{align*}
			& \int_{\B} \chi^2\left( f^{p-1} |B^{g_0}|_{g_0} + f^{p} |\Riem^{g_0}|_{g_0}\right)\, d\vol_{g_0} \\
			\leq & \left( \int_{\B} \chi^{\frac{4p}{p-1}} f^{2p} \right)^{\frac{p-1}{2p}}\left( \int_{B_{g_0}(x,2s)} |B^{g_0}|_{g_0}^{\frac{2p}{p+1}}\, d\vol_{g_0} \right)^{\frac{p+1}{2p}}  + \left(\int_{\B} \chi^4 f^{2p} d\vol_{g_0} \right)^{\frac{1}{2}} \left(\int_{B_{g_0}(x,2s)} |\Riem^{g_0}|_{g_0}^2\, d\vol_{g_0} \right)^{\frac{1}{2}}.
		\end{align*}
		Thanks to the inequality $\|\Riem^{g_0}\|_{L^2(\B,g_0)} \leq \|\Riem^g\|_{L^2(\B,g)}\leq \ve$ by \eqref{asump:Smallness_Riem} and \Cref{rk:Ineq_Riem}, we obtain 
		\begin{align*}
			& \int_{\B} \chi^2\left( f^{p-1} |B^{g_0}|_{g_0} + f^{p} |\Riem^{g_0}|_{g_0}\right)\, d\vol_{g_0} \\
			\leq & \left( \int_{\B} \chi^{\frac{4p}{p-1}} f^{2p} \right)^{\frac{p-1}{2p}}\left( \int_{B_{g_0}(x,2s)} |B^{g_0}|_{g_0}^{\frac{2p}{p+1}}\, d\vol_{g_0} \right)^{\frac{p+1}{2p}}  + \ve \left(\int_{\B} \chi^4 f^{2p} d\vol_{g_0} \right)^{\frac{1}{2}} .
		\end{align*}
		We plug the above estimate in \eqref{eq:est_f2p_1}, \eqref{eq:est_f2p_2} and \eqref{eq:est_f2p_3}. Using Young inequality, the inequality $\chi^{\frac{4p}{p-1}}\leq \chi^4$ (since $\chi\leq 1$ and $\frac{p}{p-1}\geq 2$ by choice of $p$), we obtain 
		\begin{align}
			\left( \int_{\B}  \chi^4 f^{2p}  \, d\vol_{g_0}\right)^{\frac{1}{2}} \leq &\ C\int_{\B} |d\chi|^2 f^p  d\vol_{g_0} \label{eq:est_f2p_12}\\[2mm]
			& +  C\left( \int_{B_{g_0}(x,2s)}  |B^{g_0}|_{g_0}^{\frac{2p}{p+1}}\, d\vol_{g_0}\right)^{\frac{p+1}{p-1}} \label{eq:est_f2p_22}\\[2mm]
			&\ + C \int_{\B} \chi^2 f^{p-2} |dJ^{g_0}|^2_{g_0}\, d\vol_{g_0}. \label{eq:est_f2p_32}
		\end{align}
		We now use the assumption $p\leq 2$ and the inequality $|J^{g_0}|\leq f$ to obtain $|J^{g_0}|^{p-2}\geq f^{p-2}$. Using this into \eqref{eq:est_f2p_32}, we obtain 
		\begin{align}\label{eq:Int_J1}
			\int_{\B} \chi^2 f^{p-2} |dJ^{g_0}|^2_{g_0}\, d\vol_{g_0} \leq \frac{4}{p^2}\int_{\B} \chi^2 \left| d\left( (J^{g_0})^{p/2}\right) \right|^2_{g_0}\, d\vol_{g_0}.
		\end{align}
		Thanks to \Cref{lm:der_Jp}, we have
		\begin{align*}
			\int_{\B} \chi^2 \left| d\left( (J^{g_0})^{p/2}\right) \right|^2_{g_0}\, d\vol_{g_0} \leq C\int_{\B} |d\chi|^2_{g_0} (J^{g_0})^p\, d\vol_{g_0} \leq C\int_{\B} |d\chi|^2_{g_0} f^p\, d\vol_{g_0} .
		\end{align*}
		We plug this inequality into \eqref{eq:Int_J1} and \eqref{eq:est_f2p_32}:
		\begin{align*}
			\left( \int_{\B}  \chi^4 f^{2p}  \, d\vol_{g_0}\right)^{\frac{1}{2}} \leq  C\int_{\B} |d\chi|^2 f^p  d\vol_{g_0}  +  C\left( \int_{B_{g_0}(x,2s)}  |B^{g_0}|_{g_0}^{\frac{2p}{p+1}}\, d\vol_{g_0}\right)^{\frac{p+1}{p-1}}.
		\end{align*}
		Using the inequality $|d\chi|_{g_0}\leq Cs^{-1}$ and $p\leq 2$, we obtain 
		\begin{align*}
			 \left( \int_{\B}  \chi^4 f^{2p}  \, d\vol_{g_0}\right)^{\frac{1}{2}} 
			\leq &\ \frac{C}{s^2}\int_{B_{g_0}(x,2s)} f^p  d\vol_{g_0}  +  C\|B^{g_0}\|_{L^{\frac{2p}{p+1}}(B_{g_0}(x,2s),g_0)}^{\frac{2p}{p-1}} \\[2mm]
			\leq &\ C\frac{\vol_{g_0}(B_{g_0}(x,2s))^{\frac{2-p}{2}}}{s^2}\|f\|_{L^2(B_{g_0}(x,2s),g_0)}^p  +  C\|B^{g_0}\|_{L^{\frac{2p}{p+1}}(B_{g_0}(x,2s),g_0)}^{\frac{2p}{p-1}} .
		\end{align*}

	\end{proof}
	
	As a corollary, we obtain an upper bound for the volume of balls thanks to \cite{Carron20}. This is possible thanks to the above $\ve$-regularity.
	
	\begin{lemma}\label{lm:VolumeGrowth}
		There exist $\ve>0$ and $C>0$ depending on $p$, the operation $\star$ in \eqref{eq:lap_Sch} and the constants $\gamma_S$ and $\gamma_L$ such that the following holds. Assume that $\|\Riem^{g}\|_{L^2(\B,g)} \leq \ve$.
		Then, for each geodesic ball $B_{g_0}(x,s)\subset \B\setminus \{0\}$, we have
		\begin{align*}
			\frac{\vol_{g_0}(B_{g_0}(x,s))}{s^4} \leq C.
		\end{align*}
	\end{lemma}
	
	\begin{proof}
		By following the proof of Theorem A in \cite{Carron20} without the renormalization of the metric (in their notation, we have $\nu = 4p$), we obtain
		\begin{align}\label{eq:initial_vol}
			\theta(x,s)\coloneqq \frac{\vol_{g_0}(B_{g_0}(x,s))}{s^4} 
			\leq &\ C(p)\left(1+ s^{(4p-1)\frac{4p-4}{3}} \int_{B_{g_0}(x,s)} |\Ric^{g_0}|^{2p}\, d\vol_{g_0} \right).
		\end{align}
		We also have from \cite[page 462]{Carron20} that 
		\begin{align*}
			\vol_{g_0}(B_{g_0}(x,2s)) 
			\leq  C(p)\left( \vol_{g_0}(B_{g_0}(x,s)) + s^{4+(4p-1)\frac{4p-4}{3}} \int_{B_{g_0}(x,2s)} |\Ric^{g_0}|^{2p}\, d\vol_{g_0}\right). 
		\end{align*}
		Dividing by $(2s)^{-4}$, we obtain the following estimate:
		\begin{align}\label{eq:doubling}
			\theta(x,2s) \leq C\, \theta(x,s) + C\, s^{(4p-1)\frac{4p-4}{3}} \int_{B_{g_0}(x,2s)} |\Ric^{g_0}|^{2p}\, d\vol_{g_0}.
		\end{align}
		By \Cref{pr:eps_regp}, we have
		\begin{align*}
			\int_{B_{g_0}(x,s)} |\Ric^{g_0}|^{2p}\, d\vol_{g_0} 
			\leq & \ \frac{C}{s^4}\, \vol_{g_0}(B_{g_0}(x,2s))^{2-p}\, \|\Sch^{g_0}\|_{L^2(B_{g_0}(x,2s))}^{2p}   + \|B^{g_0}\|_{L^{\frac{2p}{p+1}}(B_{g_0}(x,2s),g_0)}^{\frac{4p}{p+1}} .
		\end{align*}
		Thus, we obtain
		\begin{align*}
			& s^{(4p-1)\frac{4p-4}{3}} \int_{B_{g_0}(x,s)} |\Ric^{g_0}|^{2p}\, d\vol_{g_0} \\[1mm]
			\leq & \ C\, s^{(4p-1)\frac{4p-4}{3}-4}\, \vol_{g_0}(B_{g_0}(x,2s))^{2-p}\, \|\Sch^{g_0}\|_{L^2(B_{g_0}(x,2s))}^{2p} + s^{(4p-1)\frac{4p-4}{3}}\, \|B^{g_0}\|_{L^{\frac{2p}{p+1}}(\B,g_0)}^{\frac{4p}{p+1}}   .
		\end{align*}
		Thanks to \ref{asump:Bachp} and \Cref{lm:integrability_Bach} we obtain
		\begin{align*}
			& s^{(4p-1)\frac{4p-4}{3}} \int_{B_{g_0}(x,s)} |\Ric^{g_0}|^{2p}\, d\vol_{g_0} \\[2mm]
			\leq &\ C +  C\, s^{(4p-1)\frac{4p-4}{3}-4}\, \vol_{g_0}(B_{g_0}(x,2s))^{2-p}\, \|\Sch^{g_0}\|_{L^2(B_{g_0}(x,2s))}^{2p} \\[2mm]
			\leq &\ C +  C\, s^{(4p-1)\frac{4p-4}{3}-4 +4(2-p)}\, \theta(x,2s)^{2-p}\, \|\Sch^{g_0}\|_{L^2(B_{g_0}(x,2s))}^{2p}\\[2mm]
			\leq & \ C +  C\, s^{(4p-1)\frac{4p-4}{3}+4(1-p)}\, \theta(x,2s)^{2-p}\, \|\Sch^{g_0}\|_{L^2(B_{g_0}(x,2s))}^{2p}.
		\end{align*}
		Using $p>1$, we obtain the following estimate on the exponent of $s$:
		\begin{align*}
			q\coloneqq (4p-1)\frac{4p-4}{3}+4(1-p) = 4(p-1)\left( \frac{4p-1}{3} - 1 \right) >0.
		\end{align*}
		Therefore, we obtain 
		\begin{align}\label{eq:Est2}
			s^{(4p-1)\frac{4p-4}{3}} \int_{B_{g_0}(x,s)} |\Ric^{g_0}|^{2p}\, d\vol_{g_0} 
			\leq  C+ C\, s^q\, \theta(x,2s)^{2-p}.
		\end{align}
		If $p=2$, then \Cref{lm:VolumeGrowth} is proved. If $1<p<2$, then $0<2-p<1$. Plugging \eqref{eq:Est2} into \eqref{eq:initial_vol}, we obtain a constant $C>0$ depending only on $p$ such that
		\begin{align}\label{eq:Est1}
			\theta(x,s) \leq C+ C\, s^q \, \theta(x,2s)^{2-p}.
		\end{align}
		Pluggin \eqref{eq:Est2} into \eqref{eq:doubling}, we obtain
		\begin{align}\label{eq:Est3}
			\theta(x,2s) \leq C\, \theta(x,s) + C\, s^q\, \theta(x,2s)^{2-p}.
		\end{align}
		Combining \eqref{eq:Est1} and \eqref{eq:Est3}, we obtain
		\begin{align*}
			\theta(x,2s) \leq C + C\, s^q\, \theta(x,2s)^{2-p}.
		\end{align*}
		Since $2-p<1$, we apply Young's inequality and obtain \Cref{lm:VolumeGrowth}.
	\end{proof}

	\subsection{Singularity removability}
	
	In this section, we prove that $g_0$ extends across the origin, up to reparametrization.
	
	\begin{lemma}\label{lm:singularity_remov}
		There exists a diffeomorphism $\Psi$ of $\B\setminus \{0\}$ such that $\Psi^*g_0\in W^{2,2p}(\B)$.
	\end{lemma}
	
	To prove \Cref{lm:singularity_remov}, we proceed to a blow-up around the origin and prove that the strong limit in the $W^{2,2p}_{\loc}$-topology is a flat space. Thus, we can extend our metric before the origin. For each $s\in (0,\frac{1}{8})$, we consider the metric $h_s \coloneq s^{-2} g_0(s\cdot)$ on $\B_{1/s}\setminus \{0\}$. The metric $h_s$ still satisfies the following property.

	\begin{lemma}\label{lm:BlowUp_schouten}
		There exist $\ve>0$ and $C>0$ depending on $p$, and the constants appearing in the assumptions of \Cref{th:Huber4D} such that the following holds. Assume that $\|\Riem^g\|_{L^2(\B,g)}\leq \ve$. Then, for each geodesic ball $B_{h_s}(x,1)\subset \B_{1/s}\setminus \{0\}$, it holds
		\begin{align*}
			\|\Sch^{h_s}\|_{L^{2p}(B_{h_s}(x,1))} \xrightarrow[s\to 0]{} 0.
		\end{align*}
	\end{lemma}

	\begin{proof}
		We have the following relations:
		\begin{align*}
			& s^2 \Riem^{h_s} = \Riem^{g_0}, \quad s^4 d\vol_{h_s} = d\vol_{g_0},\quad s^{-2} B^{h_s} = B^{g_0},\\[2mm]
			& \Sch^{h_s} =\Sch^{g_0}, \quad s^{-2} J^{h_s} = J^{g_0}, \quad B_{h_s}(x,1) = B_{g_0}(x,s),\\[2mm]
			& |\Sch^{g_0}|^2_{g_0} = s^{-4}|\Sch^{h_s}|^2_{h_s},\quad |B^{g_0}|_{g_0} = s^{-4}|B^{h_s}|_{h_s}.
		\end{align*}
		Fix a ball $B_{h_s}(x,1)\subset \B_{1/s}\setminus \{0\}$ and let $\theta \coloneq \frac{1}{4}\dist_{\mathrm{eucl}}(x,\{0\})$.	By \Cref{pr:eps_regp}, we have the following estimate:
		\begin{align*}
			 \|\Sch^{g_0}\|_{L^{2p}(B_{g_0}(sx,s\theta ))} 
			\leq & \ \frac{C}{s^{\frac{2}{p}}}\vol_{g_0}(B_{g_0}(sx,2s\theta))^{\frac{1}{p}-\frac{1}{2}}\|\Sch^{g_0}\|_{L^2(B_{g_0}(sx,2s\theta))} 
			+C \|B^{g_0}\|_{L^{\frac{2p}{p+1}}(\B,g_0)}^{\frac{2}{p+1}}  .
		\end{align*}
		We multiply by $s^{2-\frac{2}{p}}$ in order to obtain a scale-invariant formulation of the left-hand side:
		\begin{align*}
			& \left(s^{4p-4}\int_{B_{g_0}(sx,s\theta)} |\Sch^{g_0}|^{2p}\, d\vol_{g_0} \right)^{\frac{1}{2p}} \\[2mm]
			\leq &\ C\left( \frac{ \vol_{g_0}(B_{g_0}(sx,2s\theta)) }{s^4} \right) ^{\frac{1}{p}-\frac{1}{2}} \|\Sch^{g_0}\|_{L^2(B_{g_0}(sx,2s\theta))} + Cs^{2-\frac{2}{p}} 
			\|B^{g_0}\|_{L^{\frac{2p}{p+1}}(\B,g_0)}^{\frac{2}{p+1}}    .
		\end{align*}
		By \Cref{lm:VolumeGrowth}, we obtain
		\begin{align*}
			 \left(s^{4p-4}\int_{B_{g_0}(sx,s\theta)} |\Sch^{g_0}|^{2p}\, d\vol_{g_0} \right)^{\frac{1}{2p}} 
			\leq &\ C \|\Sch^{g_0}\|_{L^2(B_{g_0}(sx,2s\theta))}  + Cs^{2-\frac{2}{p}} 
			\|B^{g_0}\|_{L^{\frac{2p}{p+1}}(\B,g_0)}^{\frac{2}{p+1}}  .
		\end{align*}
		Here, the constant $C$ depends on $\theta$, by \Cref{lm:VolumeGrowth}. Since $p>1$, it holds $2-\frac{2}{p}>0$ and we obtain
		\begin{align*}
			\|\Sch^{h_s}\|_{L^{2p}(B_{h_s}(x,1))} = \left(s^{4p-4}\int_{B_{g_0}(sx,s\theta)} |\Sch^{g_0}|^{2p}\, d\vol_{g_0} \right)^{\frac{1}{2p}} \xrightarrow[s\to 0]{} 0.
		\end{align*}
	\end{proof}

	Thanks to \cite[Theorem 7.1]{yang1992II} (see also \cite[Theorem 4.2]{marini2024} and \cite{gallot1988,hiroshima1995,petersen2001,yang1992I}) together with \Cref{lm:BlowUp_schouten}, for any ball $B_{h_s}(x,t)$, there exist harmonic coordinates $(x^1,\ldots,x^4)$ such that the components $(h_s)_{ij} \in W^{2,2p}(B_{\xi}(x,t/2))$ satisfy the following properties
	\begin{align*}
		& \|x^i\|_{L^{\infty}(B_{\xi}(x,t/2))} \leq 2t,\\
		& \|\g x^i\otimes \g x^j - \xi^{ij} \|_{L^{\infty}(B_{\xi}(x,t/2))}\leq C\left(t^{4p-4} \int_{B_{h_s}(x,t)} |\Sch^{h_s}|^{2p}_{h_s}\, d\vol_{h_s} \right)^{\frac{1}{2p}},\\
		& \|\g^2 (h_s)_{ij} \|_{L^{2p}(B_{\xi}(x,t/2))} \leq Ct^{-2\frac{p-1}{p}} \|\Sch^{h_s}\|_{L^{2p}(B_{h_s}(x,t))}.
	\end{align*}
	Up to a subsequence, the metrics $(h_s)_{s>0}$ converge as $s\to 0$ in the weak $W^{2,2p}_{\loc}$-topology to some metric $h_0$ on $\R^4\setminus \{0\}$. In particular, we deduce that $(\Riem^{h_s})_{s>0}$ converges to $\Riem^{h_0}$ weakly in $L^{2p}_{\loc}$. From the following relation, we obtain that $\Riem^{h_0}=0$:
	\begin{align*}
		\|\Riem^{h_s}\|_{L^2\left(B_{h_s}(x,t),h_s\right)} = \|\Riem^{g_0}\|_{L^2\left(B_{g_0}(sx,st),g_0 \right)} \xrightarrow[s\to 0]{} 0.
	\end{align*}
	In the spirit of the singularity removability for Ricci-flat metrics, see for instance \cite{gao1990,smith1992,streets2010}, we have the following result.
	
	\begin{claim}
		There exists a diffeomorphism $\Phi\colon \R^4\setminus \{*\} \to \R^4\setminus \{0\}$ such that $\Phi^*h_0 = \xi$.
	\end{claim}
	\begin{remark}
		Here, we strongly use the fact that we are in dimension greater than 2, so that $\R^4\setminus \{0\}$ is simply connected.
	\end{remark}
	
	\begin{proof}
		Since $\Riem^{h_0}=0$, for every $x\in \R^4\setminus\{0\}$, there exists an open set $U_x\subset \R^4\setminus\{0\}$ containing $x$, another open set $V_x\subset \R^4$ and an isometry $\phi_x \colon U_x\to V_x$. If $y\in \R^4\setminus\{0,x\}$ is such that $U_x\cap U_y\neq \emptyset$, then $\phi_x\circ\phi_y^{-1}\colon \phi_y(U_y\cap U_x) \to \phi_x(U_y\cap U_x)$ is an isometry of $\R^4$. Thus, $\phi_x\circ\phi_y^{-1}$ is the composition of some rotations and translations. Up to changing $\phi_y$ by an isometry of $\R^4$, we can assume that $\phi_x\circ\phi_y^{-1} = \id$. \\
		
		Therefore, given any $x\in \R^4\setminus\{0\}$ and any path $\gamma\colon [0,1]\to \R^4\setminus\{0\}$ such that $\gamma(0)=x$, the map $\phi_x$ can be analytically continued along $\gamma$, in the sense of \cite[Corollary 12.3]{lee2018}. Given $y\in \R^4\setminus\{0,x\}$, we define $\Phi(y)$ to be the value at $y$ of any analytic continuation of $\phi_x$ from $x$ to $y$. This is well defined, since $X$ is simply connected, thanks to \cite[Theorem 12.2]{lee2018}. Then, the map $\Phi \colon (\R^4\setminus\{0\},h_0)\to (\R^4,\xi)$ is a local isometry. \\
		
		Consider now $(X_n)_{n\in\N}$ an exhaustion of $\R^4\setminus\{0\}$ by simply connected compact sets, \textit{i.e.} each $X_n$ is compact in $\R^4\setminus\{0\}$, if $n\leq m$ then $X_n \subset X_m$ and $\R^4\setminus\{0\} = \bigcup_{n\in\N} X_n$. Since each $X_n$ is complete, we obtain that $\Phi :X_n\to \Phi(X_n)$ is a Riemannian covering map thanks to \cite[Theorem 6.3]{lee2018}. Applying \cite[Corollary A.59]{lee2018}, we deduce that each $\Phi \colon X_n\to \Phi(X_n)$ is a diffeomorphism. Thus $\Phi\colon X_n\to \Phi(X_n)$ is a global isometry (a local isometry preserves the arc-length, so that a bijective local isometry is a global isometry). This is valid for any $n$. Hence $\Phi : \R^4\setminus\{0\}\to \Phi(\R^4\setminus\{0\})$ is a bijective global isometry. Thus $\Phi(\R^4\setminus\{0\})$ has the same topology as $\R^4\setminus\{0\}$ and we obtain that the map $\Phi\colon (\R^4\setminus \{0\},h_0)\to (\R^4\setminus \{*\},\xi)$ is a global isometry.
	\end{proof}
	
	Consequently, $(\Phi^*h_s)_{s>0}$ has only one adherence value as $s\to 0$, meaning that $\Phi^* h_s\xrightarrow[s\to 0]{} \xi$ in the weak $W^{2,2p}_{\loc}$-topology. In other words there exists a diffeomorphism $\Psi\colon\B\setminus \{0\}\to \B\setminus \{0\}$ such that $\Psi^* g_0$ converges to $\xi$ at the origin in the weak $W^{2,2p}$-topology. Hence, $\Psi^* g_0\in W^{2,2p}(\B)\hookrightarrow C^0(\B)$. This proves \Cref{lm:singularity_remov}.

	\section{Asymptotically flat and asymptotically locally Euclidean manifolds }\label{sec:asympt_flat}
	
	In this section we are interested in  non-compact Riemannian manifolds with a single end, which are asymptotically locally Euclidean. By this term, we mean the following definition:
	
	\begin{defi}
		Let $k\in\N$ and $\alpha\in(0,1)$.	A Riemannian manifold $(M, g)$, non-compact with a single end, is said to be $C^{k,\alpha}$-asymptotically locally Euclidean (ALE) of order $\tau > 0$ if there exist a compact subset $K \subset M$, $R > 0$ and $\Gamma$ is a finite group of $SO(n)$ acting freely and a $C^{k+1,\alpha}$-diffeomorphism (coordinate chart at infinity)
		\[
		\Phi\colon M \setminus K \to (\mathbb{R}^n \setminus B(0, R)) / \Gamma
		\]
		such that the components of the metric in this chart, $g_{ij}$, satisfy:
		\begin{align*}
			\begin{cases}
				g_{ij}(z) = \delta_{ij} + O(|z|^{-\tau}),\\[2mm]
				\partial^m g_{ij}(z) = O(|z|^{-\tau-m}), \quad \forall m \in \{1, \dots, k\},\\[2mm]
				[\partial^k g_{ij}]_\alpha(z) = O(|z|^{-\tau-k-\alpha}),
			\end{cases}
		\end{align*}
		where
		\[
		[\partial^k g_{ij}]_\alpha(z) = \sup_{0<d_g(z,z')\leq 1} \frac{|\partial^k g_{ij}(z) - \partial^k g_{ij}(z')|}{|z - z'|^\alpha}.
		\]
		For brevity, we write: $g - \delta \in C^{k, \alpha}_{-\tau}(M)$. If $\Gamma = \{1\}$ then $(M,g)$ is said to be   $C^{k,\alpha}$-asymptotically flat of order $\tau$.
	\end{defi}
	
	In \Cref{th:Herzlich}, Herzlich proved that a necessary condition for an ALE manifold $(M,g)$ to have a $C^{0,\delta}$-compactification is the existence of some $1> \alpha \geq \eta>0$ such that $W^g\in C^{0,\alpha}_{-2-\eta}(M)$ and $\Cot^g\in C^{0,\alpha}_{-3-\eta}(M)$. Differentiating once more, a similar assumption concerning the Bach tensor would be $B^g\in C^{0,\alpha}_{-4-\eta}(M)$, which implies $B^g\in L^p$ for some $p>1$.
	In our case, we consider a metric $g$ such that $g-\delta \in C^{1}_\tau(M)$ for some $\tau>0$. Then we make the naive compactification, considering $h\coloneq  \vert x\vert^4 \iota^*g$  where $\iota\colon z\mapsto \frac{z}{|z|^2}$ is the inversion. Hence, we have the following asymptotic expansion:
	$$
	h_{ij} \ust{r\to 0}{=}  \delta_{ij} + O\left(r^\tau\right), \qquad \det h \ust{r\to 0}{=} 1+o(1)\qquad  \text{ and } \qquad \Gamma_h  \ust{r\to 0}{=}O\left(r^\tau\right),
	$$
	where $\Gamma_h$ are the Christoffel symbol of $h$. Thanks to those estimates, we can easily check that  the assumptions 1 and  3 to 6 of \Cref{th:global_Huber_4D} are satisfied. \\
	
	We now check the assumption $\Riem^{h}\in L^2(\B,h)$. First, since the Weyl tensor is conformally invariant, we have 
	\begin{align*}
		| W^{h} |_{h}^2\, \sqrt{\det h}\in L^1.
	\end{align*}
	Let $\tilde h \coloneq \iota^*g$, then the Schouten tensor changes according to the formula \eqref{eq:conf_change_Schouten}:
	\begin{align*}
		|\Sch^h|_{h} & \leq |\Sch^{\tilde{h}}|_{h} + \left| \Hess_{\tilde{h}}(2\log r) - d(2\log r)\otimes d(2\log r)+ \frac{|d(2\log r)|^2_{\tilde{h}}}{2} {\tilde{h}} \right|_{h} \\
		& \leq r^{-4}|\Sch^{\tilde{h}}|_{\tilde{h}} +  \left| \Hess_{\tilde{h}}(2\log r) - d(2\log r)\otimes d(2\log r)+ \frac{|d(2\log r)|^2_{\tilde{h}}}{2} \tilde{h} \right|_{h}.
	\end{align*}
	The Hessian changes according to the following formula, see for instance \cite[Equation (2)]{curry2018}: if $\tilde{h}=e^{2u} h$ and $f\in C^{\infty}$, then 
	\begin{align*}
		(\Hess_{\tilde{h}} f)_{ij} & = \g^{\tilde{h}}_i \g^{\tilde{h}}_j f \\[2mm]
		& = \g^{h}_i \g^{\tilde{h}}_j f - (\g^h_i u) (\g^{\tilde{h}}_j f) - (\g^h_j u) (\g^{\tilde{h}}_i f) + \scal{ \g^h u}{ \g^{\tilde{h}} f}_h h_{ij} \\[2mm]
		& = (\Hess_h f)_{ij} -2(\g^h u\otimes \g^h f)_{ij} + \scal{du}{df}_h h_{ij} .
	\end{align*}	
	We apply this formula to the choice $\tilde{h}=e^{2u} h$ with $u=-2\log r$ and $f=2\log r$:
	\begin{align*}
		& \Hess_ {\tilde{h}}(2\log r) - d(2\log r)\otimes d(2\log r)+ \frac{|d(2\log r)|^2_{ \tilde{h}}}{2}  \tilde{h} \\
		=&\ \Hess_{h}(2\log r) +d(2\log r)\otimes d(2\log r) - \frac{|d(2\log r)|^2_{h}}{2} h.
	\end{align*}
	The first order of the above expression in the expansion as $r\to 0$ is obtained by changing $h$ with the Euclidean metric:
	\begin{align*}
		-\frac{2}{r^2}dr\otimes dr +\frac{4}{r^2}dr\otimes dr - \frac{2}{r^2} (dr^2 +r^2d\theta^2) = -2\, d\theta^2.
	\end{align*}
	Hence, we obtain
	\begin{align*}
		\int_{\B} |\Sch^{h}|^2_{h}\, d\vol_{h} \leq C\int_{\B} |\Sch^{ \tilde{h}}|^2_{ \tilde{h}}\, d\vol_{ \tilde{h}} + C.
	\end{align*}
	Finally the Bach tensor satisfies $\vert B^h\vert_h = \vert x\vert ^{-8} \vert B^{\tilde{h}}\vert_{\tilde{h}}$, hence we easily check that  $\vert B^h\vert_h \in L^{p'}$ for $1<p'<\frac{3p}{1+2p}$.
	Thus, the metric $h$ can be conformally extended in the $W^{2,2p'}$-topology across the origin and \Cref{cor:AF} is proved.
	
	\section{Isolated singularities of $W^{4,\frac{4}{3}+\ve}$-immersions}\label{sec:immersion}
	
	In this section, we apply \Cref{th:Huber4D} to immersions whose second fundamental form is twice differentiable.
	
	\begin{theorem}
		Let $\ve\in(0,\frac{1}{10})$ and $d\geq 5$ be an integer. There exists $\delta>0$ depending only on $d$ and $\ve$ such that the following holds. Consider $\Phi\colon \B^4\setminus \{0\}\to\R^d$ be a smooth immersion such that its second fundamental form $\II_{\Phi}$ satisfies
		\begin{align*}
			\int_{\B} |\g^2\II_{\Phi}|^{\frac{4}{3}+\ve}_{g_{\Phi}} + |\g\II|^{2+\ve}_{g_{\Phi}} + |\II|^{4+\ve}_{g_{\Phi}}+1\, d\vol_{g_{\Phi}} <+\infty.
		\end{align*}
		Assume that
		\begin{align*}
			\int_{\B}  |\II|^4_{g_{\Phi}}\, d\vol_{g_{\Phi}} <\delta.
		\end{align*}
		Then, there exist a diffeomorphism $f\in \Diff(\B^4\setminus \{0\})$, a conformal factor $u\in C^{\infty}(\B^4\setminus \{0\})$, an number $p>2$ and a metric $h$ on $\B^4$ of class $W^{2,p}(\B^4)$ such that $g_{\Phi\circ f} = e^{2u} h$.
	\end{theorem}
	
	\begin{proof}
		By Gauss--Codazzi equations, we have
		\begin{align*}
			\Riem^{g_{\Phi}}_{ijkl} = \scal{\II_{ik}}{\II_{jl}} - \scal{\II_{il}}{\II_{jk}} \in L^2(\B,g_{\Phi}).
		\end{align*}
		We deduce that the Bach tensor of $g_{\Phi}$ is also quadratic in $\II$:
		\begin{align*}
			|B^{g_{\Phi}}|_{g_{\Phi}} \leq C \left( |\g^2\II|_{g_{\Phi}} |\II|_{g_{\Phi}} + |\g \II|_{g_{\Phi}}^2 + |\II|_{g_{\Phi}}^4 \right).
		\end{align*}
		Thus, we obtain $B^{g_{\Phi}}\in L^p(\B,g_{\Phi})$ for some $p=p(\ve)>1$. Therefore, the assumptions \ref{asump:finite_volp} and \ref{asump:Bachp} of \Cref{th:Huber4D} are satisfied. The Sobolev inequality \ref{asump:Sobolevp} is verified in \eqref{eq:Soblev_immersions}. The assumption \ref{asump:Laplacianp} follows from a uniform estimate of the Green function for Dirichlet boundary conditions. As explained in \cite{carron1996}, this is a consequence of the Sobolev inequality.
	\end{proof}
	
	\bibliographystyle{plain}
	\bibliography{Huber_bib.bib}

@book {doCarmo,
	AUTHOR = {do Carmo, Manfredo P.},
	TITLE = {Differential forms and applications},
	SERIES = {Universitext},
	NOTE = {Translated from the 1971 Portuguese original},
	PUBLISHER = {Springer-Verlag, Berlin},
	YEAR = {1994},
	PAGES = {x+118},
	ISBN = {3-540-57618-5},
	MRCLASS = {58A10 (53-01 53A05 58-01)},
	MRNUMBER = {1301070},
	MRREVIEWER = {Frans\ Cantrijn},
	DOI = {10.1007/978-3-642-57951-6},
	URL = {https://doi.org/10.1007/978-3-642-57951-6},
}

@article {LeguilRosenberg,
	AUTHOR = {Leguil, Martin and Rosenberg, Harold},
	TITLE = {On harmonic diffeomorphisms from conformal annuli to
	{R}iemannian annuli},
	JOURNAL = {Mat. Contemp.},
	FJOURNAL = {Matem\'atica Contempor\^anea},
	VOLUME = {43},
	YEAR = {2014},
	PAGES = {171--221},
	ISSN = {0103-9059,2317-6636},
	MRCLASS = {53C43 (53A10)},
	MRNUMBER = {3426261},
	MRREVIEWER = {Joe\ S.\ Wang},
}

@article {Troyanov10,
	AUTHOR = {Troyanov, Marc},
	TITLE = {On the {H}odge decomposition in {$\Bbb R^n$}},
	JOURNAL = {Mosc. Math. J.},
	FJOURNAL = {Moscow Mathematical Journal},
	VOLUME = {9},
	YEAR = {2009},
	NUMBER = {4},
	PAGES = {899--926, 936},
	ISSN = {1609-3321,1609-4514},
	MRCLASS = {58A14 (42B20 46F05)},
	MRNUMBER = {2663996},
	MRREVIEWER = {Gilles\ Carron},
	DOI = {10.17323/1609-4514-2009-9-4-899-926},
	URL = {https://doi.org/10.17323/1609-4514-2009-9-4-899-926},
}

@article {Her97,
	AUTHOR = {Herzlich, Marc},
	TITLE = {Compactification conforme des vari\'et\'es asymptotiquement
	plates},
	JOURNAL = {Bull. Soc. Math. France},
	FJOURNAL = {Bulletin de la Soci\'et\'e{} Math\'ematique de France},
	VOLUME = {125},
	YEAR = {1997},
	NUMBER = {1},
	PAGES = {55--91},
	ISSN = {0037-9484,2102-622X},
	MRCLASS = {53C21 (53C80)},
	MRNUMBER = {1459298},
	MRREVIEWER = {Piotr\ T.\ Chru\'sciel},
	URL = {http://www.numdam.org/item?id=BSMF_1997__125_1_55_0},
}

@article {CL22,
	AUTHOR = {Chen, Bo and Li, Yuxiang},
	TITLE = {Huber's theorem for manifolds with {$L^{\frac n2}$} integrable
	{R}icci curvatures},
	JOURNAL = {Trans. Amer. Math. Soc.},
	FJOURNAL = {Transactions of the American Mathematical Society},
	VOLUME = {375},
	YEAR = {2022},
	NUMBER = {8},
	PAGES = {5907--5922},
	ISSN = {0002-9947,1088-6850},
	MRCLASS = {53C18 (53C21)},
	MRNUMBER = {4469241},
	MRREVIEWER = {Mohameden\ Ahmedou},
	DOI = {10.1090/tran/8703},
	URL = {https://doi.org/10.1090/tran/8703},
}

@article{CH02,
	Author = {Carron, Gilles and Herzlich, Marc},
	Title = {The {Huber} theorem for non-compact conformally flat manifolds},
	FJournal = {Commentarii Mathematici Helvetici},
	Journal = {Comment. Math. Helv.},
	ISSN = {0010-2571},
	Volume = {77},
	Number = {1},
	Pages = {192--220},
	Year = {2002},
	Language = {English},
	DOI = {10.1007/s00014-002-8336-0},
	zbMATH = {1756816},
	Zbl = {1009.53031}
}

@article {Carron20,
	AUTHOR = {Carron, Gilles},
	TITLE = {Euclidean volume growth for complete {R}iemannian manifolds},
	JOURNAL = {Milan J. Math.},
	FJOURNAL = {Milan Journal of Mathematics},
	VOLUME = {88},
	YEAR = {2020},
	NUMBER = {2},
	PAGES = {455--478},
	ISSN = {1424-9286,1424-9294},
	MRCLASS = {53C21 (58C40 58J35 58J50)},
	MRNUMBER = {4182081},
	MRREVIEWER = {Peijun\ Wang},
	DOI = {10.1007/s00032-020-00321-8},
	URL = {https://doi.org/10.1007/s00032-020-00321-8},
}

@article {CQY,
	AUTHOR = {Chang, Sun-Yung A. and Qing, Jie and Yang, Paul C.},
	TITLE = {Compactification of a class of conformally flat 4-manifold},
	JOURNAL = {Invent. Math.},
	FJOURNAL = {Inventiones Mathematicae},
	VOLUME = {142},
	YEAR = {2000},
	NUMBER = {1},
	PAGES = {65--93},
	ISSN = {0020-9910,1432-1297},
	MRCLASS = {53C21 (58J60)},
	MRNUMBER = {1784799},
	MRREVIEWER = {Robert\ McOwen},
	DOI = {10.1007/s002220000083},
	URL = {https://doi.org/10.1007/s002220000083},
}

@article {Toro,
	AUTHOR = {Toro, Tatiana},
	TITLE = {Surfaces with generalized second fundamental form in {$L^2$}
	are {L}ipschitz manifolds},
	JOURNAL = {J. Differential Geom.},
	FJOURNAL = {Journal of Differential Geometry},
	VOLUME = {39},
	YEAR = {1994},
	NUMBER = {1},
	PAGES = {65--101},
	ISSN = {0022-040X,1945-743X},
	MRCLASS = {49Q05},
	MRNUMBER = {1258915},
	MRREVIEWER = {J.\ E.\ Brothers},
	URL = {http://projecteuclid.org/euclid.jdg/1214454677},
}

@book {Gra1,
	AUTHOR = {Grafakos, Loukas},
	TITLE = {Classical {F}ourier analysis},
	SERIES = {Graduate Texts in Mathematics},
	VOLUME = {249},
	EDITION = {Second},
	PUBLISHER = {Springer, New York},
	YEAR = {2008},
	PAGES = {xvi+489},
	ISBN = {978-0-387-09431-1},
	MRCLASS = {42-01 (42Bxx)},
	MRNUMBER = {2445437},
	MRREVIEWER = {Andreas\ Seeger},
}

@article {CMLS,
	AUTHOR = {Coifman, R. and Lions, P.-L. and Meyer, Y. and Semmes, S.},
	TITLE = {Compensated compactness and {H}ardy spaces},
	JOURNAL = {J. Math. Pures Appl. (9)},
	FJOURNAL = {Journal de Math\'ematiques Pures et Appliqu\'ees. Neuvi\`eme
	S\'erie},
	VOLUME = {72},
	YEAR = {1993},
	NUMBER = {3},
	PAGES = {247--286},
	ISSN = {0021-7824},
	MRCLASS = {46E99 (35S05 42B30 46F10 46N20 47G30)},
	MRNUMBER = {1225511},
}

@article {Ri14,
	AUTHOR = {Rivi\`ere, Tristan},
	TITLE = {Variational principles for immersed surfaces with
	{$L^2$}-bounded second fundamental form},
	JOURNAL = {J. Reine Angew. Math.},
	FJOURNAL = {Journal f\"ur die Reine und Angewandte Mathematik. [Crelle's
	Journal]},
	VOLUME = {695},
	YEAR = {2014},
	PAGES = {41--98},
	ISSN = {0075-4102,1435-5345},
	MRCLASS = {58E30 (53C42)},
	MRNUMBER = {3276154},
	MRREVIEWER = {Themistocles\ M.\ Rassias},
	DOI = {10.1515/crelle-2012-0106},
	URL = {https://doi.org/10.1515/crelle-2012-0106},
}

@article {Huber,
	AUTHOR = {Huber, Alfred},
	TITLE = {On subharmonic functions and differential geometry in the
	large},
	JOURNAL = {Comment. Math. Helv.},
	FJOURNAL = {Commentarii Mathematici Helvetici},
	VOLUME = {32},
	YEAR = {1957},
	PAGES = {13--72},
	ISSN = {0010-2571,1420-8946},
	MRCLASS = {30.00 (31.00)},
	MRNUMBER = {94452},
	MRREVIEWER = {E.\ F.\ Beckenbach},
	DOI = {10.1007/BF02564570},
	URL = {https://doi.org/10.1007/BF02564570},
}

@book{helein2002,
  title = {Harmonic Maps, Conservation Laws and Moving Frames},
  author = {H{\'e}lein, Fr{\'e}d{\'e}ric},
  year = {2002},
  series = {Cambridge {{Tracts}} in {{Mathematics}}},
  edition = {Second},
  volume = {150},
  publisher = {Cambridge University Press, Cambridge},
  doi = {10.1017/CBO9780511543036},
  urldate = {2023-06-11},
  isbn = {978-0-521-81160-6},
  mrnumber = {1913803}
}

@article{laurain2014,
  title = {Angular Energy Quantization for Linear Elliptic Systems with Antisymmetric Potentials and Applications},
  author = {Laurain, Paul and Rivi{\`e}re, Tristan},
  year = {2014},
  month = may,
  journal = {Analysis \& PDE},
  volume = {7},
  number = {1},
  pages = {1--41},
  publisher = {Mathematical Sciences Publishers},
  issn = {1948-206X},
  doi = {10.2140/apde.2014.7.1},
  urldate = {2023-01-04}
}

@book{warner1983,
  title = {Foundations of Differentiable Manifolds and {{Lie}} Groups},
  author = {Warner, Frank W.},
  year = {1983},
  series = {Graduate {{Texts}} in {{Mathematics}}},
  volume = {94},
  publisher = {Springer-Verlag, New York-Berlin},
  isbn = {978-0-387-90894-6},
  mrnumber = {722297}
}

@book{spivak1999,
	title =     {A Comprehensive Introduction to Differential Geometry},
	author =    {Michael Spivak},
	publisher = {Publish or Perish},
	isbn =      {9780914098706,0914098705},
	year =      {1999},
	edition =   {3},
	volume =    {Volume 1},
}

@article {muller1995,
	AUTHOR = {M\"{u}ller, S. and \v{S}ver\'{a}k, V.},
	TITLE = {On surfaces of finite total curvature},
	JOURNAL = {J. Differential Geom.},
	FJOURNAL = {Journal of Differential Geometry},
	VOLUME = {42},
	YEAR = {1995},
	NUMBER = {2},
	PAGES = {229--258},
	ISSN = {0022-040X,1945-743X},
	MRCLASS = {53A05 (53C21)},
	MRNUMBER = {1366547},
	MRREVIEWER = {Tatiana\ Toro},
	URL = {http://projecteuclid.org/euclid.jdg/1214457233},
}

@article {bernard2014,
	AUTHOR = {Bernard, Yann and Rivi\`ere, Tristan},
	TITLE = {Energy quantization for {W}illmore surfaces and applications},
	JOURNAL = {Ann. of Math. (2)},
	FJOURNAL = {Annals of Mathematics. Second Series},
	VOLUME = {180},
	YEAR = {2014},
	NUMBER = {1},
	PAGES = {87--136},
	ISSN = {0003-486X,1939-8980},
	MRCLASS = {53C42},
	MRNUMBER = {3194812},
	MRREVIEWER = {Haizhong\ Li},
	DOI = {10.4007/annals.2014.180.1.2},
	URL = {https://doi.org/10.4007/annals.2014.180.1.2},
}

@incollection {curry2018,
	AUTHOR = {Curry, Sean N. and Gover, A. Rod},
	TITLE = {An introduction to conformal geometry and tractor calculus,
	with a view to applications in general relativity},
	BOOKTITLE = {Asymptotic analysis in general relativity},
	SERIES = {London Math. Soc. Lecture Note Ser.},
	VOLUME = {443},
	PAGES = {86--170},
	PUBLISHER = {Cambridge Univ. Press, Cambridge},
	YEAR = {2018},
	ISBN = {978-1-316-64940-4},
	MRCLASS = {53A30 (35Q75 53B15 53C29 83C05)},
	MRNUMBER = {3792084},
	MRREVIEWER = {Travis\ John\ Willse},
}

@phdthesis{vyatkin,
	title={Manufacturing Conformal Invariants of Hypersurfaces},
	author={Yuri Vyatkin},
	year={2013},
	school={University of Auckland}
}

@article {marques2007,
	AUTHOR = {Marques, Fernando C.},
	TITLE = {Conformal deformations to scalar-flat metrics with constant
	mean curvature on the boundary},
	JOURNAL = {Comm. Anal. Geom.},
	FJOURNAL = {Communications in Analysis and Geometry},
	VOLUME = {15},
	YEAR = {2007},
	NUMBER = {2},
	PAGES = {381--405},
	ISSN = {1019-8385,1944-9992},
	MRCLASS = {53C21 (53C44)},
	MRNUMBER = {2344328},
	MRREVIEWER = {Mohameden\ Ahmedou},
	URL = {http://projecteuclid.org/euclid.cag/1186755300},
}

@article {tian2005,
	AUTHOR = {Tian, Gang and Viaclovsky, Jeff},
	TITLE = {Bach-flat asymptotically locally {E}uclidean metrics},
	JOURNAL = {Invent. Math.},
	FJOURNAL = {Inventiones Mathematicae},
	VOLUME = {160},
	YEAR = {2005},
	NUMBER = {2},
	PAGES = {357--415},
	ISSN = {0020-9910,1432-1297},
	MRCLASS = {53C21 (58E11)},
	MRNUMBER = {2138071},
	MRREVIEWER = {John\ Urbas},
	DOI = {10.1007/s00222-004-0412-1},
	URL = {https://doi.org/10.1007/s00222-004-0412-1},
}

@article {carron2014,
	AUTHOR = {Carron, Gilles},
	TITLE = {Some old and new results about rigidity of critical metrics},
	JOURNAL = {Ann. Sc. Norm. Super. Pisa Cl. Sci. (5)},
	FJOURNAL = {Annali della Scuola Normale Superiore di Pisa. Classe di
	Scienze. Serie V},
	VOLUME = {13},
	YEAR = {2014},
	NUMBER = {4},
	PAGES = {1091--1113},
	ISSN = {0391-173X,2036-2145},
	MRCLASS = {53C20 (53C24 53C25 58E11)},
	MRNUMBER = {3362120},
	MRREVIEWER = {Yawei\ Chu},
}

@article {druet2002,
	AUTHOR = {Druet, Olivier and Hebey, Emmanuel},
	TITLE = {The {$AB$} program in geometric analysis: sharp {S}obolev
	inequalities and related problems},
	JOURNAL = {Mem. Amer. Math. Soc.},
	FJOURNAL = {Memoirs of the American Mathematical Society},
	VOLUME = {160},
	YEAR = {2002},
	NUMBER = {761},
	PAGES = {viii+98},
	ISSN = {0065-9266,1947-6221},
	MRCLASS = {58J05 (35J20 46E35 53C21 58E35)},
	MRNUMBER = {1938183},
	MRREVIEWER = {Gilles\ Carron},
	DOI = {10.1090/memo/0761},
	URL = {https://doi.org/10.1090/memo/0761},
}

@article {escobar1994,
	AUTHOR = {Escobar, Jos\'e{} F.},
	TITLE = {Addendum: ``{C}onformal deformation of a {R}iemannian metric
	to a scalar flat metric with constant mean curvature on the
	boundary'' [{A}nn. of {M}ath. (2) {\bf 136} (1992), no. 1,
	1--50; {MR}1173925 (93e:53046)]},
	JOURNAL = {Ann. of Math. (2)},
	FJOURNAL = {Annals of Mathematics. Second Series},
	VOLUME = {139},
	YEAR = {1994},
	NUMBER = {3},
	PAGES = {749--750},
	ISSN = {0003-486X,1939-8980},
	MRCLASS = {53C21 (53C25 58G30)},
	MRNUMBER = {1283876},
	DOI = {10.2307/2118578},
	URL = {https://doi.org/10.2307/2118578},
}

@incollection {carron1996,
	AUTHOR = {Carron, Gilles},
	TITLE = {In\'egalit\'es isop\'erim\'etriques de {F}aber-{K}rahn et
	cons\'equences},
	BOOKTITLE = {Actes de la {T}able {R}onde de {G}\'eom\'etrie
	{D}iff\'erentielle ({L}uminy, 1992)},
	SERIES = {S\'emin. Congr.},
	VOLUME = {1},
	PAGES = {205--232},
	PUBLISHER = {Soc. Math. France, Paris},
	YEAR = {1996},
	ISBN = {2-85629-047-7},
	MRCLASS = {58G25 (31C15 46E35 58G11)},
	MRNUMBER = {1427759},
	MRREVIEWER = {Thierry\ Coulhon},
}

@article {anderson1990,
	AUTHOR = {Anderson, Michael T.},
	TITLE = {Convergence and rigidity of manifolds under {R}icci curvature
	bounds},
	JOURNAL = {Invent. Math.},
	FJOURNAL = {Inventiones Mathematicae},
	VOLUME = {102},
	YEAR = {1990},
	NUMBER = {2},
	PAGES = {429--445},
	ISSN = {0020-9910,1432-1297},
	MRCLASS = {53C23 (53C21 58D27)},
	MRNUMBER = {1074481},
	MRREVIEWER = {Gudlaugur\ Thorbergsson},
	DOI = {10.1007/BF01233434},
	URL = {https://doi.org/10.1007/BF01233434},
}

@article {anderson2005,
	AUTHOR = {Anderson, Michael T.},
	TITLE = {Orbifold compactness for spaces of {R}iemannian metrics and
	applications},
	JOURNAL = {Math. Ann.},
	FJOURNAL = {Mathematische Annalen},
	VOLUME = {331},
	YEAR = {2005},
	NUMBER = {4},
	PAGES = {739--778},
	ISSN = {0025-5831,1432-1807},
	MRCLASS = {53C21 (58E11)},
	MRNUMBER = {2148795},
	MRREVIEWER = {John\ Urbas},
	DOI = {10.1007/s00208-004-0603-5},
	URL = {https://doi.org/10.1007/s00208-004-0603-5},
}

@book {besse2008,
	AUTHOR = {Besse, Arthur L.},
	TITLE = {Einstein manifolds},
	SERIES = {Classics in Mathematics},
	NOTE = {Reprint of the 1987 edition},
	PUBLISHER = {Springer-Verlag, Berlin},
	YEAR = {2008},
	PAGES = {xii+516},
	ISBN = {978-3-540-74120-6},
	MRCLASS = {53C25 (53-02)},
	MRNUMBER = {2371700},
}

@article {tian2008,
	AUTHOR = {Tian, Gang and Viaclovsky, Jeff},
	TITLE = {Volume growth, curvature decay, and critical metrics},
	JOURNAL = {Comment. Math. Helv.},
	FJOURNAL = {Commentarii Mathematici Helvetici. A Journal of the Swiss
	Mathematical Society},
	VOLUME = {83},
	YEAR = {2008},
	NUMBER = {4},
	PAGES = {889--911},
	ISSN = {0010-2571,1420-8946},
	MRCLASS = {53C21 (58E11)},
	MRNUMBER = {2442967},
	MRREVIEWER = {John\ Urbas},
	DOI = {10.4171/CMH/147},
	URL = {https://doi.org/10.4171/CMH/147},
}

@article {aldana2021,
	AUTHOR = {Aldana, Clara L. and Carron, Gilles and Tapie, Samuel},
	TITLE = {{$A_\infty$}-weights and compactness of conformal metrics
	under {$L^{n/2}$} curvature bounds},
	JOURNAL = {Anal. PDE},
	FJOURNAL = {Analysis \& PDE},
	VOLUME = {14},
	YEAR = {2021},
	NUMBER = {7},
	PAGES = {2163--2205},
	ISSN = {2157-5045,1948-206X},
	MRCLASS = {53C20 (53C18 53C23 58J60)},
	MRNUMBER = {4353568},
	MRREVIEWER = {Yan\ He},
	DOI = {10.2140/apde.2021.14.2163},
	URL = {https://doi.org/10.2140/apde.2021.14.2163},
}

@article {bando1989,
	AUTHOR = {Bando, Shigetoshi and Kasue, Atsushi and Nakajima, Hiraku},
	TITLE = {On a construction of coordinates at infinity on manifolds with
	fast curvature decay and maximal volume growth},
	JOURNAL = {Invent. Math.},
	FJOURNAL = {Inventiones Mathematicae},
	VOLUME = {97},
	YEAR = {1989},
	NUMBER = {2},
	PAGES = {313--349},
	ISSN = {0020-9910,1432-1297},
	MRCLASS = {53C20 (53C25)},
	MRNUMBER = {1001844},
	MRREVIEWER = {Thomas\ H.\ Otway},
	DOI = {10.1007/BF01389045},
	URL = {https://doi.org/10.1007/BF01389045},
}

@article {ancona2009,
	AUTHOR = {Ancona, Alano},
	TITLE = {Elliptic operators, conormal derivatives and positive parts of
	functions},
	NOTE = {With an appendix by Ha\"im Brezis},
	JOURNAL = {J. Funct. Anal.},
	FJOURNAL = {Journal of Functional Analysis},
	VOLUME = {257},
	YEAR = {2009},
	NUMBER = {7},
	PAGES = {2124--2158},
	ISSN = {0022-1236,1096-0783},
	MRCLASS = {31B25 (31B35 31C40 35J15 46E35)},
	MRNUMBER = {2548032},
	MRREVIEWER = {Anna\ Maria\ Candela},
	DOI = {10.1016/j.jfa.2008.12.019},
	URL = {https://doi.org/10.1016/j.jfa.2008.12.019},
}

@article {cabre2022,
	AUTHOR = {Cabr\'e, Xavier and Miraglio, Pietro},
	TITLE = {Universal {H}ardy-{S}obolev inequalities on hypersurfaces of
	{E}uclidean space},
	JOURNAL = {Commun. Contemp. Math.},
	FJOURNAL = {Communications in Contemporary Mathematics},
	VOLUME = {24},
	YEAR = {2022},
	NUMBER = {5},
	PAGES = {Paper No. 2150063, 25},
	ISSN = {0219-1997,1793-6683},
	MRCLASS = {53A10 (26D10 46E35 53A07)},
	MRNUMBER = {4435963},
	MRREVIEWER = {Serena\ Dipierro},
	DOI = {10.1142/S0219199721500632},
	URL = {https://doi.org/10.1142/S0219199721500632},
}

@article {allard1972,
	AUTHOR = {Allard, William K.},
	TITLE = {On the first variation of a varifold},
	JOURNAL = {Ann. of Math. (2)},
	FJOURNAL = {Annals of Mathematics. Second Series},
	VOLUME = {95},
	YEAR = {1972},
	PAGES = {417--491},
	ISSN = {0003-486X},
	MRCLASS = {49F20},
	MRNUMBER = {307015},
	MRREVIEWER = {M.\ Klingmann},
	DOI = {10.2307/1970868},
	URL = {https://doi.org/10.2307/1970868},
}

@article {michael1973,
	AUTHOR = {Michael, J. H. and Simon, L. M.},
	TITLE = {Sobolev and mean-value inequalities on generalized
	submanifolds of {$R\sp{n}$}},
	JOURNAL = {Comm. Pure Appl. Math.},
	FJOURNAL = {Communications on Pure and Applied Mathematics},
	VOLUME = {26},
	YEAR = {1973},
	PAGES = {361--379},
	ISSN = {0010-3640,1097-0312},
	MRCLASS = {49F10 (46E35)},
	MRNUMBER = {344978},
	MRREVIEWER = {David\ Kinderlehrer},
	DOI = {10.1002/cpa.3160260305},
	URL = {https://doi.org/10.1002/cpa.3160260305},
}

@article {hiroshima1995,
	AUTHOR = {Hiroshima, Tsutomu},
	TITLE = {{$C^\alpha$} compactness theorem for {R}iemannian manifolds
	with bounds on diameter, injectivity radius, and some integral
	norms of {R}icci curvature},
	JOURNAL = {Indiana Univ. Math. J.},
	FJOURNAL = {Indiana University Mathematics Journal},
	VOLUME = {44},
	YEAR = {1995},
	NUMBER = {2},
	PAGES = {397--411},
	ISSN = {0022-2518,1943-5258},
	MRCLASS = {53C21 (58G30)},
	MRNUMBER = {1355404},
	MRREVIEWER = {Zhi\ Ren\ Jin},
	DOI = {10.1512/iumj.1995.44.1993},
	URL = {https://doi.org/10.1512/iumj.1995.44.1993},
}

@incollection {gallot1988,
	AUTHOR = {Gallot, Sylvestre},
	TITLE = {Isoperimetric inequalities based on integral norms of {R}icci
	curvature},
	NOTE = {Colloque Paul L\'evy sur les Processus Stochastiques
	(Palaiseau, 1987)},
	JOURNAL = {Ast\'erisque},
	FJOURNAL = {Ast\'erisque},
	NUMBER = {157-158},
	YEAR = {1988},
	PAGES = {191--216},
	ISSN = {0303-1179,2492-5926},
	MRCLASS = {58G25 (53C20)},
	MRNUMBER = {976219},
	MRREVIEWER = {R.\ Schimming},
}

@article {marini2024,
	AUTHOR = {Marini, Ludovico and Meda, Stefano and Pigola, Stefano and
	Veronelli, Giona},
	TITLE = {{$L^p$} gradient estimates and {C}alder\'on-{Z}ygmund
	inequalities under {R}icci lower bounds},
	JOURNAL = {Rev. Mat. Iberoam.},
	FJOURNAL = {Revista Matem\'atica Iberoamericana},
	VOLUME = {40},
	YEAR = {2024},
	NUMBER = {3},
	PAGES = {803--826},
	ISSN = {0213-2230,2235-0616},
	MRCLASS = {58J05 (35B45 35J05 35R01 42B20 53C21)},
	MRNUMBER = {4739823},
	DOI = {10.4171/rmi/1476},
	URL = {https://doi.org/10.4171/rmi/1476},
}

@article {yang1992I,
	AUTHOR = {Yang, Deane},
	TITLE = {Convergence of {R}iemannian manifolds with integral bounds on
	curvature. {I}},
	JOURNAL = {Ann. Sci. \'Ecole Norm. Sup. (4)},
	FJOURNAL = {Annales Scientifiques de l'\'Ecole Normale Sup\'erieure.
	Quatri\`eme S\'erie},
	VOLUME = {25},
	YEAR = {1992},
	NUMBER = {1},
	PAGES = {77--105},
	ISSN = {0012-9593},
	MRCLASS = {53C23 (53C21)},
	MRNUMBER = {1152614},
	MRREVIEWER = {Xiao\ Wei\ Peng},
	URL = {http://www.numdam.org/item?id=ASENS_1992_4_25_1_77_0},
}

@article {yang1992II,
	AUTHOR = {Yang, Deane},
	TITLE = {Convergence of {R}iemannian manifolds with integral bounds on
	curvature. {II}},
	JOURNAL = {Ann. Sci. \'Ecole Norm. Sup. (4)},
	FJOURNAL = {Annales Scientifiques de l'\'Ecole Normale Sup\'erieure.
	Quatri\`eme S\'erie},
	VOLUME = {25},
	YEAR = {1992},
	NUMBER = {2},
	PAGES = {179--199},
	ISSN = {0012-9593},
	MRCLASS = {53C23 (53C20 53C21)},
	MRNUMBER = {1169351},
	MRREVIEWER = {Zhongmin\ Shen},
	URL = {http://www.numdam.org/item?id=ASENS_1992_4_25_2_179_0},
}

@article {gao1990,
	AUTHOR = {Gao, L. Zhiyong},
	TITLE = {Einstein metrics},
	JOURNAL = {J. Differential Geom.},
	FJOURNAL = {Journal of Differential Geometry},
	VOLUME = {32},
	YEAR = {1990},
	NUMBER = {1},
	PAGES = {155--183},
	ISSN = {0022-040X,1945-743X},
	MRCLASS = {58D27 (53C25 58D17)},
	MRNUMBER = {1064870},
	MRREVIEWER = {Xiao\ Wei\ Peng},
	URL = {http://projecteuclid.org/euclid.jdg/1214445042},
}

@article {streets2010,
	AUTHOR = {Streets, Jeffrey},
	TITLE = {Asymptotic curvature decay and removal of singularities of
	{B}ach-flat metrics},
	JOURNAL = {Trans. Amer. Math. Soc.},
	FJOURNAL = {Transactions of the American Mathematical Society},
	VOLUME = {362},
	YEAR = {2010},
	NUMBER = {3},
	PAGES = {1301--1324},
	ISSN = {0002-9947,1088-6850},
	MRCLASS = {53C25 (53C21 58E11)},
	MRNUMBER = {2563730},
	MRREVIEWER = {John\ Urbas},
	DOI = {10.1090/S0002-9947-09-04960-5},
	URL = {https://doi.org/10.1090/S0002-9947-09-04960-5},
}

@article {smith1992,
	AUTHOR = {Smith, P. D. and Yang, Deane},
	TITLE = {Removing point singularities of {R}iemannian manifolds},
	JOURNAL = {Trans. Amer. Math. Soc.},
	FJOURNAL = {Transactions of the American Mathematical Society},
	VOLUME = {333},
	YEAR = {1992},
	NUMBER = {1},
	PAGES = {203--219},
	ISSN = {0002-9947,1088-6850},
	MRCLASS = {53C21 (53C22)},
	MRNUMBER = {1052910},
	MRREVIEWER = {Chun-Li\ Shen},
	DOI = {10.2307/2154106},
	URL = {https://doi.org/10.2307/2154106},
}

@book {chavel1984,
	AUTHOR = {Chavel, Isaac},
	TITLE = {Eigenvalues in {R}iemannian geometry},
	SERIES = {Pure and Applied Mathematics},
	VOLUME = {115},
	NOTE = {Including a chapter by Burton Randol,
	With an appendix by Jozef Dodziuk},
	PUBLISHER = {Academic Press, Inc., Orlando, FL},
	YEAR = {1984},
	PAGES = {xiv+362},
	ISBN = {0-12-170640-0},
	MRCLASS = {58G25 (35P99 53C20)},
	MRNUMBER = {768584},
	MRREVIEWER = {G\'erard\ Besson},
}

@article {petersen2001,
	AUTHOR = {Petersen, Peter and Wei, Guofang},
	TITLE = {Analysis and geometry on manifolds with integral {R}icci
	curvature bounds. {II}},
	JOURNAL = {Trans. Amer. Math. Soc.},
	FJOURNAL = {Transactions of the American Mathematical Society},
	VOLUME = {353},
	YEAR = {2001},
	NUMBER = {2},
	PAGES = {457--478},
	ISSN = {0002-9947,1088-6850},
	MRCLASS = {53C20 (53C21)},
	MRNUMBER = {1709777},
	MRREVIEWER = {Joseph\ E.\ Borzellino},
	DOI = {10.1090/S0002-9947-00-02621-0},
	URL = {https://doi.org/10.1090/S0002-9947-00-02621-0},
}

@book {lee2018,
	AUTHOR = {Lee, John M.},
	TITLE = {Introduction to {R}iemannian manifolds},
	SERIES = {Graduate Texts in Mathematics},
	VOLUME = {176},
	NOTE = {Second edition of [ MR1468735]},
	PUBLISHER = {Springer, Cham},
	YEAR = {2018},
	PAGES = {xiii+437},
	ISBN = {978-3-319-91754-2; 978-3-319-91755-9},
	MRCLASS = {53-01 (53B20 53B30 53C20 53C21)},
	MRNUMBER = {3887684},
	MRREVIEWER = {Robert J. Low},
}

@article {korzy2003,
	AUTHOR = {Korzy\'{n}ski, Miko\l aj and Lewandowski, Jerzy},
	TITLE = {The normal conformal {C}artan connection and the {B}ach
	tensor},
	JOURNAL = {Classical Quantum Gravity},
	FJOURNAL = {Classical and Quantum Gravity},
	VOLUME = {20},
	YEAR = {2003},
	NUMBER = {16},
	PAGES = {3745--3764},
	ISSN = {0264-9381},
	MRCLASS = {53C07 (53A30 53C80)},
	MRNUMBER = {2001693},
	MRREVIEWER = {Simonetta Frittelli},
	DOI = {10.1088/0264-9381/20/16/314},
	URL = {https://doi.org/10.1088/0264-9381/20/16/314},
}

@article {gover2008,
	AUTHOR = {Gover, A. Rod and Somberg, Petr and Sou\v{c}ek, Vladim\'{\i}r},
	TITLE = {Yang-{M}ills detour complexes and conformal geometry},
	JOURNAL = {Comm. Math. Phys.},
	FJOURNAL = {Communications in Mathematical Physics},
	VOLUME = {278},
	YEAR = {2008},
	NUMBER = {2},
	PAGES = {307--327},
	ISSN = {0010-3616},
	MRCLASS = {58J10 (53C07 53C27 53C28)},
	MRNUMBER = {2372762},
	MRREVIEWER = {Jean-Louis Milhorat},
	DOI = {10.1007/s00220-007-0401-5},
	URL = {https://doi.org/10.1007/s00220-007-0401-5},
}

@incollection {riviere2020,
	AUTHOR = {Rivi\`ere, Tristan},
	TITLE = {The variations of {Y}ang-{M}ills {L}agrangian},
	BOOKTITLE = {Geometric analysis---in honor of {G}ang {T}ian's 60th
	birthday},
	SERIES = {Progr. Math.},
	VOLUME = {333},
	PAGES = {305--379},
	PUBLISHER = {Birkh\"{a}user/Springer, Cham},
	YEAR = {[2020] \copyright 2020},
	MRCLASS = {58E15 (35J25 35J47)},
	MRNUMBER = {4181007},
	MRREVIEWER = {Dana S. Fine},
	DOI = {10.1007/978-3-030-34953-0\_15},
	URL = {https://doi.org/10.1007/978-3-030-34953-0_15},
}

@article {tao2004,
	AUTHOR = {Tao, Terence and Tian, Gang},
	TITLE = {A singularity removal theorem for {Y}ang-{M}ills fields in
	higher dimensions},
	JOURNAL = {J. Amer. Math. Soc.},
	FJOURNAL = {Journal of the American Mathematical Society},
	VOLUME = {17},
	YEAR = {2004},
	NUMBER = {3},
	PAGES = {557--593},
	ISSN = {0894-0347},
	MRCLASS = {58E15 (53C07)},
	MRNUMBER = {2053951},
	MRREVIEWER = {J\"{u}rgen Eichhorn},
	DOI = {10.1090/S0894-0347-04-00457-6},
	URL = {https://doi.org/10.1090/S0894-0347-04-00457-6},
}

@article{cartan1924,
	AUTHOR={Cartan, Elie},
	TITLE = {Sur les espaces à connexion conforme},
	JOURNAL = {Ann. Soc. Polon. Math.},
	VOLUME = {22},
	YEAR = {1924},
	PAGES = {171-221},
}

@inproceedings {biquard2014,
	AUTHOR = {Biquard, Olivier},
	TITLE = {Einstein 4-manifolds and singularities},
	BOOKTITLE = {Proceedings of the {I}nternational {C}ongress of
	{M}athematicians---{S}eoul 2014. {V}ol. {II}},
	PAGES = {853--865},
	PUBLISHER = {Kyung Moon Sa, Seoul},
	YEAR = {2014},
	ISBN = {978-89-6105-805-6; 978-89-6105-803-2},
	MRCLASS = {53C25 (53A30)},
	MRNUMBER = {3728641},
	MRREVIEWER = {Gordon\ Craig},
}

@article {uhlenbeck1982,
	AUTHOR = {Uhlenbeck, Karen K.},
	TITLE = {Connections with {$L\sp{p}$}\ bounds on curvature},
	JOURNAL = {Comm. Math. Phys.},
	FJOURNAL = {Communications in Mathematical Physics},
	VOLUME = {83},
	YEAR = {1982},
	NUMBER = {1},
	PAGES = {31--42},
	ISSN = {0010-3616,1432-0916},
	MRCLASS = {53C05 (49F10 58E20 81E10)},
	MRNUMBER = {648356},
	MRREVIEWER = {Wolfgang\ L\"ucke},
	URL = {http://projecteuclid.org/euclid.cmp/1103920743},
}

@incollection {fefferman1985,
	AUTHOR = {Fefferman, Charles and Graham, C. Robin},
	TITLE = {Conformal invariants},
	NOTE = {The mathematical heritage of \'Elie Cartan (Lyon, 1984)},
	JOURNAL = {Ast\'erisque},
	FJOURNAL = {Ast\'erisque},
	YEAR = {1985},
	PAGES = {95--116},
	ISSN = {0303-1179,2492-5926},
	MRCLASS = {53C20 (32C10 53A30)},
	MRNUMBER = {837196},
	MRREVIEWER = {Claude\ LeBrun},
}

@article {laurain2018,
	AUTHOR = {Laurain, Paul and Rivi\`ere, Tristan},
	TITLE = {Energy quantization of {W}illmore surfaces at the boundary of
	the moduli space},
	JOURNAL = {Duke Math. J.},
	FJOURNAL = {Duke Mathematical Journal},
	VOLUME = {167},
	YEAR = {2018},
	NUMBER = {11},
	PAGES = {2073--2124},
	ISSN = {0012-7094,1547-7398},
	MRCLASS = {53A30 (35J60 35J93 35R01 49Q10)},
	MRNUMBER = {3843372},
	MRREVIEWER = {Haizhong\ Li},
	DOI = {10.1215/00127094-2018-0010},
	URL = {https://doi.org/10.1215/00127094-2018-0010},
}

@book {iwaniec,
	AUTHOR = {Iwaniec, Tadeusz and Martin, Gaven},
	TITLE = {Geometric function theory and non-linear analysis},
	SERIES = {Oxford Mathematical Monographs},
	PUBLISHER = {The Clarendon Press, Oxford University Press, New York},
	YEAR = {2001},
	PAGES = {xvi+552},
	ISBN = {0-19-850929-4},
	MRCLASS = {30-02 (00A05 30C65 35J60 42B20 58-02)},
	MRNUMBER = {1859913},
	MRREVIEWER = {Mario\ Bonk},
}

@article{martino2023,
	title={Energy quantization for Willmore surfaces with bounded index}, 
	author={Dorian Martino},
	year={2023},
	journal= {J. Eur. Math. Soc. (Accepted), \url{https://arxiv.org/abs/2305.08668}}
}

@article{gauvrit2024,
	title={Morse index stability for {Yang-Mills} connections}, 
	author={Mario Gauvrit and Paul Laurain},
	year={2024},
	journal={\url{https://arxiv.org/abs/2402.09039}}
}

@article {bahuaud2011,
	AUTHOR = {Bahuaud, Eric and Helliwell, Dylan},
	TITLE = {Short-time existence for some higher-order geometric flows},
	JOURNAL = {Comm. Partial Differential Equations},
	FJOURNAL = {Communications in Partial Differential Equations},
	VOLUME = {36},
	YEAR = {2011},
	NUMBER = {12},
	PAGES = {2189--2207},
	ISSN = {0360-5302,1532-4133},
	MRCLASS = {53C44 (35K25 35K59)},
	MRNUMBER = {2852074},
	MRREVIEWER = {Meng\ Zhu},
	DOI = {10.1080/03605302.2011.593015},
	URL = {https://doi.org/10.1080/03605302.2011.593015},
}

@article {chen2023,
	AUTHOR = {Chen, Jiaqi and Lu, Peng and Qing, Jie},
	TITLE = {Conformal {B}ach flow},
	JOURNAL = {Ann. Global Anal. Geom.},
	FJOURNAL = {Annals of Global Analysis and Geometry},
	VOLUME = {63},
	YEAR = {2023},
	NUMBER = {2},
	PAGES = {Paper No. 19, 30},
	ISSN = {0232-704X,1572-9060},
	MRCLASS = {53E20 (35K41 53C43 58J35)},
	MRNUMBER = {4569009},
	MRREVIEWER = {Yong\ Wei},
	DOI = {10.1007/s10455-023-09897-x},
	URL = {https://doi.org/10.1007/s10455-023-09897-x},
}

@article{dalio2023,
	title={Morse Index Stability for Critical Points to Conformally invariant Lagrangians}, 
	author={Francesca Da Lio and Matilde Gianocca and Tristan Rivière},
	year={2023},
	journal={\url{https://arxiv.org/abs/2212.03124}}
}

@book {han2011,
	AUTHOR = {Han, Qing and Lin, Fanghua},
	TITLE = {Elliptic partial differential equations},
	SERIES = {Courant Lecture Notes in Mathematics},
	VOLUME = {1},
	EDITION = {Second},
	PUBLISHER = {Courant Institute of Mathematical Sciences, New York; American
	Mathematical Society, Providence, RI},
	YEAR = {2011},
	PAGES = {x+147},
	ISBN = {978-0-8218-5313-9},
	MRCLASS = {35Jxx (35-01 35B50)},
	MRNUMBER = {2777537},
}

@article {brendle2021,
	AUTHOR = {Brendle, Simon},
	TITLE = {The isoperimetric inequality for a minimal submanifold in
	{E}uclidean space},
	JOURNAL = {J. Amer. Math. Soc.},
	FJOURNAL = {Journal of the American Mathematical Society},
	VOLUME = {34},
	YEAR = {2021},
	NUMBER = {2},
	PAGES = {595--603},
	ISSN = {0894-0347,1088-6834},
	MRCLASS = {53A10 (53A07)},
	MRNUMBER = {4280868},
	MRREVIEWER = {Jianquan\ Ge},
	DOI = {10.1090/jams/969},
	URL = {https://doi.org/10.1090/jams/969},
}

@book {gilbarg2001,
	AUTHOR = {Gilbarg, David and Trudinger, Neil S.},
	TITLE = {Elliptic partial differential equations of second order},
	SERIES = {Classics in Mathematics},
	NOTE = {Reprint of the 1998 edition},
	PUBLISHER = {Springer-Verlag, Berlin},
	YEAR = {2001},
	PAGES = {xiv+517},
	ISBN = {3-540-41160-7},
	MRCLASS = {35-02 (35Jxx)},
	MRNUMBER = {1814364},
}

@article {helein1991,
	AUTHOR = {H\'elein, Fr\'ed\'eric},
	TITLE = {R\'egularit\'e{} des applications faiblement harmoniques entre
	une surface et une vari\'et\'e{} riemannienne},
	JOURNAL = {C. R. Acad. Sci. Paris S\'er. I Math.},
	FJOURNAL = {Comptes Rendus de l'Acad\'emie des Sciences. S\'erie I.
	Math\'ematique},
	VOLUME = {312},
	YEAR = {1991},
	NUMBER = {8},
	PAGES = {591--596},
	ISSN = {0764-4442},
	MRCLASS = {58E20},
	MRNUMBER = {1101039},
	MRREVIEWER = {John\ C.\ Wood},
}

@article{lan2025,
	title={The Analysis of Willmore Surfaces and its Generalizations in Higher Dimensions}, 
	author={Tian Lan and Dorian Martino and Tristan Rivière},
	year={2025},
	journal ={ {\url{https://arxiv.org/abs/2511.01777}} }, 
}

@article {laurain2013,
	AUTHOR = {Laurain, Paul and Rivi\`ere, Tristan},
	TITLE = {Energy quantization for biharmonic maps},
	JOURNAL = {Adv. Calc. Var.},
	FJOURNAL = {Advances in Calculus of Variations},
	VOLUME = {6},
	YEAR = {2013},
	NUMBER = {2},
	PAGES = {191--216},
	ISSN = {1864-8258,1864-8266},
	MRCLASS = {35J48 (35J50)},
	MRNUMBER = {3043576},
	MRREVIEWER = {Marcelo\ F.\ Furtado},
	DOI = {10.1515/acv-2012-0105},
	URL = {https://doi.org/10.1515/acv-2012-0105},
}

@article {ai2017,
	AUTHOR = {Ai, Wanjun and Yin, Hao},
	TITLE = {Neck analysis of extrinsic polyharmonic maps},
	JOURNAL = {Ann. Global Anal. Geom.},
	FJOURNAL = {Annals of Global Analysis and Geometry},
	VOLUME = {52},
	YEAR = {2017},
	NUMBER = {2},
	PAGES = {129--156},
	ISSN = {0232-704X,1572-9060},
	MRCLASS = {35J47 (35B44 35J50)},
	MRNUMBER = {3690012},
	MRREVIEWER = {Luca\ Battaglia},
	DOI = {10.1007/s10455-017-9551-7},
	URL = {https://doi.org/10.1007/s10455-017-9551-7},
}

@article {dalio2015,
	AUTHOR = {Da Lio, Francesca},
	TITLE = {Compactness and bubble analysis for 1/2-harmonic maps},
	JOURNAL = {Ann. Inst. H. Poincar\'e{} C Anal. Non Lin\'eaire},
	FJOURNAL = {Annales de l'Institut Henri Poincar\'e{} C. Analyse Non
	Lin\'eaire},
	VOLUME = {32},
	YEAR = {2015},
	NUMBER = {1},
	PAGES = {201--224},
	ISSN = {0294-1449,1873-1430},
	MRCLASS = {58E20 (35B65 35J50 35J60 35R11 35S30)},
	MRNUMBER = {3303947},
	MRREVIEWER = {Fr\'ed\'eric\ Robert},
	DOI = {10.1016/j.anihpc.2013.11.003},
	URL = {https://doi.org/10.1016/j.anihpc.2013.11.003},
}
	
\end{document}